\RequirePackage{ifthen}
\RequirePackage{ifpdf}
\newcommand{\driverOption}{}
\ifthenelse{\boolean{pdf}}{
  \renewcommand{\driverOption}{pdftex}
} { 
  \renewcommand{\driverOption}{dvips}
}

\documentclass[reqno, 11pt, letterpaper, commented, oneside, \driverOption]{amsart}

\newboolean{isCommented}
\usepackage{geometry}
\usepackage{bbm}
\usepackage[final]{graphicx}
\usepackage{upref}
\usepackage{enumerate}
\usepackage{latexsym}
\usepackage{amssymb}
\usepackage[ansinew]{inputenc}
\usepackage[T1]{fontenc}
\usepackage{mathtools}
\usepackage{mycommands}

\ifthenelse{\boolean{pdf}}{
  \usepackage[final]{ps4pdf}
} { 
  \usepackage[inactive]{ps4pdf}
}
\PSforPDF{\usepackage{psfrag}}

\newcommand{\hyperrefDriverOption}{}
\ifthenelse{\boolean{pdf}}{
	\renewcommand{\hyperrefDriverOption}{pdftex}
} { 
	\renewcommand{\hyperrefDriverOption}{hypertex}
}
\usepackage[\hyperrefDriverOption,
  colorlinks = false, 
  pdfauthor={Torsten Mütze},
  pdftitle={Proof of the middle levels conjecture}]
  {hyperref}
    
\ifthenelse{\boolean{isCommented}} {
	\newcommand{\TM}[1]{\marginpar{\parbox{4cm}{{\small {\bf TM:} #1}}}}
	\newcommand{\FW}[1]{\marginpar{\parbox{4cm}{{\small {\bf FW:} #1}}}}
} { 
	\newcommand{\TM}[1]{}
	\newcommand{\FW}[1]{}
}

\newtheorem{theorem}{Theorem}
\newtheorem{lemma}[theorem]{Lemma}

\newtheorem{proposition}[theorem]{Proposition}

\theoremstyle{definition}

\theoremstyle{remark}
\newtheorem{remark}[theorem]{Remark}

\long\def\symbolfootnote[#1]#2{\begingroup
\def\thefootnote{\fnsymbol{footnote}}\footnote[#1]{#2}\endgroup}

\ifthenelse{\boolean{isCommented}} {
	\geometry{
	  hmargin={25mm, 50mm},
	  marginparwidth=40mm, 
	  vmargin={25mm, 25mm},
	  headsep=10mm,
	  headheight=5mm,
	  footskip=10mm
	}
} { 
	\geometry{
	  hmargin={25mm, 25mm}, 
	  vmargin={25mm, 25mm},
	  headsep=10mm,
	  headheight=5mm,
	  footskip=10mm
	}
}

\begin{document}

\begin{center}

\LARGE Proof of the middle levels conjecture
\vspace{2mm}

\Large Torsten Mütze
\vspace{2mm}

\large
  Department of Computer Science \\
  ETH Zürich, 8092 Zürich, Switzerland \\
  {\small\tt torsten.muetze@inf.ethz.ch}
\vspace{5mm}

\small

\begin{minipage}{0.8\linewidth}
\textsc{Abstract.}
Define the middle layer graph as the graph whose vertex set consists of all bitstrings of length $2n+1$ that have exactly $n$ or $n+1$ entries equal to 1, with an edge between any two vertices for which the corresponding bitstrings differ in exactly one bit. The middle levels conjecture asserts that this graph has a Hamilton cycle for every $n\geq 1$. This conjecture originated probably with Havel, Buck and Wiedemann, but has also been attributed to Dejter, Erd{\H{o}}s, Trotter and various others, and despite considerable efforts it remained open during the last 30 years.
In this paper we prove the middle levels conjecture. In fact, we construct $2^{2^{\Omega(n)}}$ different Hamilton cycles in the middle layer graph, which is best possible.
\end{minipage}

\vspace{2mm}

\begin{minipage}{0.8\linewidth}
\end{minipage}

\vspace{2mm}

\begin{minipage}{0.8\linewidth}
\end{minipage}

\end{center}

\vspace{5mm}

\setcounter{tocdepth}{2}

\section{Introduction}

The question whether a graph has a Hamilton cycle --- a cycle that visits every vertex exactly once --- is a fundamental graph theoretical problem. Answering this question is one of the prototypical NP-complete problems, as shown by Karp in his landmark paper \cite{Karp72}.
Even for families of graphs defined by very simple algebraic constructions this question turns out to be surprisingly difficult.
One prominent example is the \emph{middle layer graph} whose vertex set consists of all bitstrings of length $2n+1$ that have exactly $n$ or $n+1$ entries equal to 1, with an edge between any two vertices for which the corresponding bitstrings differ in exactly one bit. Note that the middle layer graph is a subgraph of the discrete cube of dimension $2n+1$, the graph whose vertex set are \emph{all} bitstrings of length $2n+1$, with an edge between any two vertices that differ in exactly one bit.
The middle layer graph is bipartite, connected, the number of vertices is $N:=\binom{2n+1}{n}+\binom{2n+1}{n+1}=2^{\Theta(n)}$, and all vertices have degree $n+1=\Theta(\log(N))$ (i.e., the graph is sparse). Moreover, the middle layer graph is vertex-transitive, i.e., any pair of vertices can be mapped onto each other by an automorphism (informally speaking, the graph `looks' the same from the point of view of any vertex).
The \emph{middle levels conjecture}, also known as revolving door conjecture, asserts that the middle layer graph has a Hamilton cycle for every $n\geq 1$. This conjecture originated probably with Havel~\cite{MR737021} and Buck and Wiedemann~\cite{MR737262}, but has also been attributed to Dejter, Erd{\H{o}}s, Trotter~\cite{MR962224} and various others. It also appears as Exercise~56 in Knuth's book~\cite[Section~7.2.1.3]{knuth}.
There are two main motivations for tackling the middle levels conjecture. The first motivation are Gray codes: In its simplest form, a Gray code is a cyclic list of all binary code words (=bitstrings) of a certain length such that any two adjacent code words in the list differ in exactly one bit. Clearly, such a Gray code corresponds to a Hamilton cycle in the entire cube, and a Hamilton cycle in the middle layer graph is a restricted Gray code (see \cite{Savage:1997} for various applications of Gray codes in all their different flavours).
The second motivation is a classical conjecture due to Lov{\'a}sz~\cite{MR0263646}, which asserts that every connected vertex-transitive graph (as e.g.\ the middle layer graph) has a Hamilton path and, apart from five exceptional graphs, even a Hamilton cycle.
This vastly more general conjecture is still wide open today: Even for explicit families of vertex-transitive graphs as e.g.\ the so-called Kneser graphs and bipartite Kneser graphs, only the denser ones are known to have a Hamilton cycle \cite{MR1999733,MR2020936,Johnson11} (Kneser graphs were introduced by Lov{\'a}sz in his celebrated proof of Kneser's conjecture \cite{MR514625}). In fact, the middle layer graph is the sparsest bipartite Kneser graph, so in some sense it is the hardest obstacle in proving Hamiltonicity for this family of graphs.
For further results and references concerning Lov{\'a}sz' conjecture, in particular with respect to other interesting families of vertex-transitive graphs that are defined via group actions (Cayley graphs), we refer to the surveys \cite{Kutnar20095491,MR2548568}.

The middle levels conjecture has attracted considerable attention over the last 30 years.
In a sequence of algorithmic improvements and with the availability of more powerful computers, so far the conjecture has been verified for all $n\leq 19$ \cite{MR1745213,MR2548541,shimada-amano} (for $n=19$ the middle layer graph has $N=137.846.528.820$ vertices).
The first notable asymptotic result is \cite{savage:93}, where it was shown that the middle layer graph has a cycle of length $N^{0.836}$. Improving on this, it was shown in \cite{MR1350586} that there is a cycle that visits $0.25N$ many vertices of the middle layer graph, and in \cite{MR1329390} that there is a cycle that visits $0.839N$ many vertices. Another major step towards the conjecture was \cite{MR2046083}, where the existence of a cycle of length $(1-c/\sqrt{n})N$ was established, where $c$ is some constant.
Unfortunately, attempts to obtain a Hamilton cycle from the union of two perfect matchings in the middle layer graph have not been successful so far \cite{MR962223, MR962224, MR1268348} (even though these constructions of perfect matchings deepened our understanding of the structure of the middle layer graph).
For other relaxations of the middle levels conjecture and partial results, see e.g.\ \cite{horakEtAl:05,Gregor20102448}.

\subsection{Our results}

In this paper we prove the middle levels conjecture.

\begin{theorem}
\label{thm:main1}
For any $n\geq 1$, the middle layer graph has a Hamilton cycle.
\end{theorem}

In fact, we prove the following more general result:

\begin{theorem}
\label{thm:main2}
For any $n\geq 1$, the middle layer graph has at least $\frac{1}{4}2^{2^{\lfloor(n+1)/4\rfloor}}=2^{2^{\Omega(n)}}$ different Hamilton cycles.
\end{theorem}

Note that $2^{2^{\Omega(n)}}$ Hamilton cycles are substantially more than we get from applying all $2(2n+1)!=2^{\Theta(n\log n)}$ automorphisms of the middle layer graph to a single Hamilton cycle (these automorphisms are given by bit permutations and possibly bit inversion). In fact, any graph $G$ has at most $|V(G)|!$ different Hamilton cycles, where $V(G)$ denotes the vertex set of $G$. This establishes an upper bound of $N!=2^{2^{\cO(n)}}$ for the number of Hamilton cycles in the middle layer graph and shows that Theorem~\ref{thm:main2} is best possible (up to the constant in the exponent).

Our arguments are constructive and yield an algorithm that outputs each of the Hamilton cycles referred to in Theorem~\ref{thm:main2} in polynomial time per cycle (polynomial in the size of the middle layer graph, which is exponential in $n$).

\subsection{Proof ideas}
\label{sec:outline}

Before starting work in earnest, we give an informal overview of the main ideas and techniques used in the proofs.
On a very high level, our construction of Hamilton cycles in the middle layer graph consists of two steps. In the first step, we construct a 2-factor in this graph, i.e., a set of disjoint cycles that visit all vertices of the graph, or equivalently, a $2$-regular spanning subgraph. Constructing a 2-factor is clearly much easier than constructing a Hamilton cycle directly. In the second step we modify the 2-factor locally by what we call \emph{flippable pairs} to join all of its cycles to a single cycle (which is then a Hamilton cycle). In fact, the concept of flippable pairs allows us to reduce the problem of proving that the middle layer graph has a Hamilton cycle (or many Hamilton cycles) to the problem of proving that a suitably defined auxiliary graph is connected (or has many spanning trees), which is considerably easier.
We emphasize here that the concept of flippable pairs in principle applies to \emph{any} graph, not just the middle layer graph. We therefore believe that this two-step approach to proving Hamiltonicity can be extended to other interesting families of graphs (e.g.\ the above-mentioned Kneser graphs).

In the following we explain the two steps of our construction in more detail. 
For the reader's convenience, the notions and ideas introduced below are illustrated in Figure~\ref{fig:idea}.

\subsubsection{Structure of Hamilton cycles in the middle layer graph}
\label{sec:outline-structure}

Our constructions are based on some simple observations about the structure of Hamilton cycles in the middle layer graph.
To discuss those, we need to introduce some definitions (throughout this paper, key definitions will be highlighted by \emph{italic} headings).

\textit{The discrete cube and its layers.}
For any $n\geq 1$ we define $B_n:=\{0,1\}^n$ as the set of all bitstrings of length $n$, and we let $B_n(k)\seq B_n$, $0\leq k\leq n$, denote the set of all bitstrings of length $n$ with exactly $k$ entries equal to 1 (and the other $n-k$ entries equal to 0).
We define the \emph{$n$-dimensional cube $Q_n$} as the graph with vertex set $B_n$ and an edge between any two vertices for which the corresponding bitstrings differ in exactly one bit. Moreover, we define the graph $Q_n(k,k+1)$, $0\leq k\leq n-1$, as the subgraph of $Q_n$ induced by the vertex sets $B_n(k)$ and $B_n(k+1)$, and we refer to $Q_n(k,k+1)$ as a \emph{layer} of $Q_n$. In particular, we refer to the layers $Q_{2n}(k,k+1)$, $k=n,n+1,\ldots,2n-1$, as the \emph{upper layers} of $Q_{2n}$, and to $Q_{2n+1}(n,n+1)$ as the \emph{middle layer} of $Q_{2n+1}$ (this is the middle layer graph).

In the following, for any bitstring $x$, we denote by $B_n(k)\circ x$ the set of bitstrings obtained from $B_n(k)$ by attaching $x$ to each bitstring from $B_n(k)$, and by $Q_n(k,k+1)\circ x$ the graph obtained by from $Q_n(k,k+1)$ by attaching $x$ to the bitstring at each vertex.
By partitioning the vertices of the middle layer graph $Q_{2n+1}(n,n+1)$ into two sets according to the value of the last bit (0 or 1), we observe that this graph has the following structure (see the top of Figure~\ref{fig:idea}): It consists of a copy of $Q_{2n}(n,n+1)\circ (0)$ (with edges between the vertex sets $B_{2n}(n)\circ (0)$ and $B_{2n}(n+1)\circ (0)$) and a copy of $Q_{2n}(n-1,n)\circ (1)$ (with edges between the vertex sets $B_{2n}(n-1)\circ(1)$ and $B_{2n}(n)\circ(1)$) plus a perfect matching $M_{2n+1}$ between the vertex sets $B_{2n}(n)\circ (0)$ and $B_{2n}(n)\circ(1)$ (these are exactly the edges on which the last bit flips).
As a consequence, \emph{any} Hamilton cycle $H$ in the middle layer graph $Q_{2n+1}(n,n+1)$ has the following structure: Removing from $H$ all edges from the matching $M_{2n+1}$, what is left are sets of disjoint paths $\cP$ and $\cP'$ in the subgraphs $Q_{2n}(n,n+1)\circ (0)$ and $Q_{2n}(n-1,n)\circ (1)$ that visit all vertices in these subgraphs, and that start and end in the vertex sets $B_{2n}(n)\circ(0)$ and $B_{2n}(n)\circ (1)$, respectively (see the top of Figure~\ref{fig:idea}). Note that $|\cP|=|\cP'|=|B_{2n}(n)|-|B_{2n}(n+1)|=\binom{2n}{n}-\binom{2n}{n+1}=\frac{1}{n+1}\binom{2n}{n}=C_n$, where $C_n$ denotes the $n$-th Catalan number, i.e., the number of paths in $\cP$ and $\cP'$ is $C_n=2^{\Theta(n)}$ (regardless of the Hamilton cycle $H$). In particular, any Hamilton cycle has exactly $2C_n$ edges on which the last bit flips (these are edges from $M_{2n+1}$). As this argument can be repeated for every bit position, any Hamilton cycle in the middle layer graph $Q_{2n+1}(n,n+1)$ has exactly $2C_n$ edges on which the $i$-th bit flips, for every $i=1,\ldots,2n+1$. This enforced balancedness of the number of bitflips in each coordinate explains the difficulty of proving the middle levels conjecture inductively: Any inductive argument has to show how to connect exponentially many paths to a single Hamilton cycle.

Observe that for the overall structure of the Hamilton cycle $H$, only the end vertices of the paths in $\cP$ and $\cP'$ are relevant. In fact, we can think of these paths as single edges of a matching, connecting the end vertices of the paths. Moreover, as far as the cycle structure of $H$ is concerned, we can ignore the edges from the matching $M_{2n+1}$ and think of $\cP$ and $\cP'$ as subgraphs of $Q_{2n}(n,n+1)$ and $Q_{2n}(n-1,n)$, respectively, graphs that share the set of vertices $B_{2n}(n)$ (see the middle of Figure~\ref{fig:idea}). Note that this corresponds to contracting the edges from the matching $M_{2n+1}$, or deleting the $(2n+1)$-th bit from all vertices in $\cP$ and $\cP'$.
This simplified way of thinking about Hamilton cycles in the middle layer graph will be very fruitful, and is also used at the bottom of Figure~\ref{fig:idea} (paths are drawn as matching edges, and edges from the matching $M_{2n+1}$ are ignored).

\begin{figure}
\centering
\PSforPDF{
 \psfrag{q2np1}{$Q_{2n+1}(n,n+1)$}
 \psfrag{q2n0}{$Q_{2n}(n,n+1)\circ(0)$}
 \psfrag{q2n1}{$Q_{2n}(n-1,n)\circ(1)$}
 \psfrag{q2n0n}{$Q_{2n}(n,n+1)$}
 \psfrag{q2n1n}{$Q_{2n}(n-1,n)$}
 \psfrag{b2nnp1}{$B_{2n}(n+1)\circ(0)$}
 \psfrag{b2n0}{$B_{2n}(n)\circ(0)$}
 \psfrag{b2n1}{$B_{2n}(n)\circ(1)$}
 \psfrag{b2nnm1}{$B_{2n}(n-1)\circ(1)$}
 \psfrag{m2}{$M_{2n+1}$}
 \psfrag{ham}{$H$}
 \psfrag{pa}{$\cP$}
 \psfrag{pap}{$\cP'$}
 \psfrag{pa0}{$\cP\circ(0)$}
 \psfrag{pa1}{$f(\cP)\circ(1)$}
 \psfrag{cc}{$\cC_{2n+1}$}
 \psfrag{p1}{\small $P_1$}
 \psfrag{p2}{\small $P_2$}
 \psfrag{p3}{\small $P_3$}
 \psfrag{p4}{\small $P_4$}
 \psfrag{p5}{\small $P_5$}
 \psfrag{p6}{\small $P_6$}
 \psfrag{r2}{\small $R_2$}
 \psfrag{r3}{\small $R_3$}
 \psfrag{fp1}{\small $f(P_1)$}
 \psfrag{fp2}{\small $f(P_2)$}
 \psfrag{fp3}{\small $f(P_3)$}
 \psfrag{fp4}{\small $f(P_4)$}
 \psfrag{fp5}{\small $f(P_5)$}
 \psfrag{fp6}{\small $f(P_6)$}
 \psfrag{c1}{$C_1$}
 \psfrag{c2}{$C_2$}
 \psfrag{c3}{$C_3$}
 \psfrag{gcc}{$\cG(\cC_{2n+1},\cX)$}
 \psfrag{p1p4}{\small $(P_1,P_4)$}
 \psfrag{p2p3}{\small $(P_2,P_3)$}
 \psfrag{p5p6}{\small $(P_5,P_6)$}
 \psfrag{sp}{\parbox{15mm}{the \\ flippable pair $(P_2,P_3)$}}
 \includegraphics{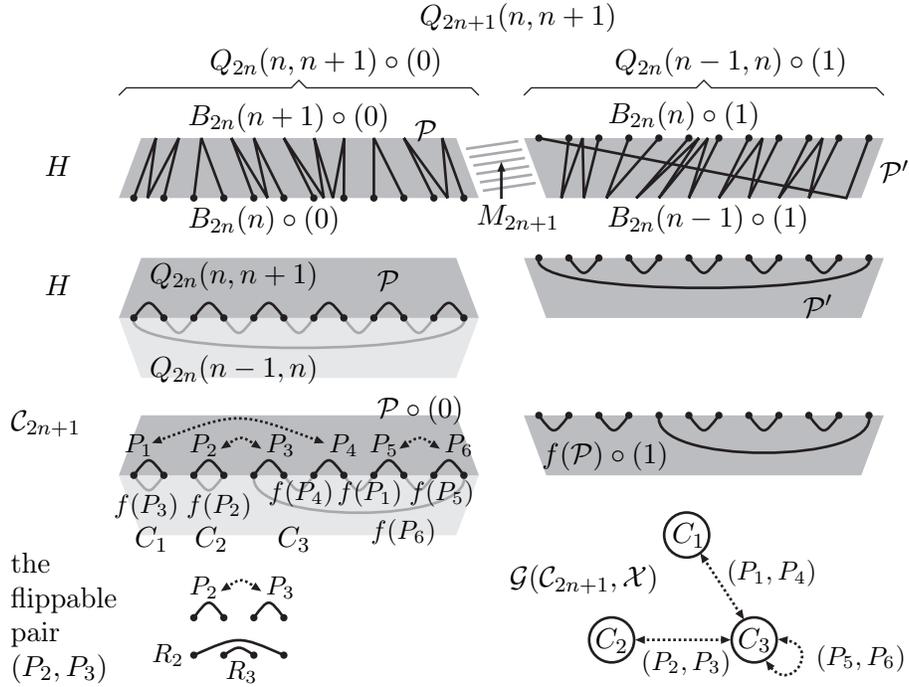}
}
\caption{A Hamilton cycle $H$ in the middle layer graph $Q_{2n+1}(n,n+1)$ (top), a schematic view of the end vertices of the sets of paths $\cP$ and $\cP'$ induced by this Hamilton cycle (middle), a 2-factor $\cC_{2n+1}=\{C_1,C_2,C_3\}$ and flippable pairs $\cX=\{(P_1,P_4),(P_2,P_3),(P_5,P_6)\}$ (bottom left), and the corresponding graph $\cG(\cC_{2n+1},\cX)$ (bottom right). Flippable pairs of paths are indicated by dotted arrows.}
\label{fig:idea}
\end{figure}

\subsubsection{Inductive construction of 2-factors in the middle layer graph}
\label{sec:outline-2factor}

As indicated before, to prove that the middle layer graph has a Hamilton cycle, we begin by constructing a 2-factor in this graph. This construction has already been presented and analyzed partially in our earlier work \cite{muetze-weber:12} (joint with Franziska Weber). In this approach, we inductively construct sets of disjoint paths in \emph{all} upper layers of $Q_{2n}$ (not just in a single layer). In particular, we obtain a set $\cP$ of exactly $C_n$ many disjoint paths in $Q_{2n}(n,n+1)$ that visit all vertices of this graph and that start and end in $B_{2n}(n)$.
In an intermediate step of the construction, the paths $\cP$ in $Q_{2n}(n,n+1)$ are used to build a 2-factor $\cC_{2n+1}$ in the middle layer of $Q_{2n+1}$ as follows (see the bottom of Figure~\ref{fig:idea}): By applying an isomorphism $f$ between the graphs $Q_{2n}(n,n+1)$ and $Q_{2n}(n-1,n)$, we obtain a set of disjoint paths $f(\cP)$ in $Q_{2n}(n-1,n)$ that visit all vertices of this graph and that start and end in $B_{2n}(n)$. The end vertices of paths in $\cP$ and the isomorphism $f$ are such that the set of all path end vertices (a subset of $B_{2n}(n)$) is mapped onto itself (however, the end vertices of one particular path are in general \emph{not} mapped onto themselves, see Figure~\ref{fig:idea}).
I.e., we obtain a 2-factor $\cC_{2n+1}$ in $Q_{2n+1}(n,n+1)$ by taking the union of $\cP\circ(0)$, $f(\cP)\circ(1)$ and the appropriate matching edges from $M_{2n+1}$. As mentioned before, for analyzing the cycle structure of $\cC_{2n+1}$, it suffices to consider the two matchings on the set of end vertices of $\cP$ and $f(\cP)$ induced by these paths.

It turns out that the choice of the isomorphism $f$ in each induction step allows some freedom, so this construction yields in total $2^{\binom{n}{2}}=2^{\Theta(n^2)}$ different 2-factors in the middle layer graph $Q_{2n+1}(n,n+1)$. Unfortunately, only few of them (for very simple choices of $f$) seem to be amenable to theoretical analysis. This is because varying $f$ changes many end vertices of paths in $f(\cP)$ simultaneously and therefore affects the resulting 2-factor $\cC_{2n+1}$ \emph{globally} (in a way that is hard to control).
Even though numerical experiments performed in \cite{muetze-weber:12} suggest that the entire family of 2-factors arising from this construction contains a Hamilton cycle for every $n\geq 1$, the 2-factors that could be analyzed theoretically have many short cycles (the cycle length is at most quadratic in $n$, so the number of cycles is exponential).

\subsubsection{Flippable pairs}
\label{sec:outline-flipp}

The construction of 2-factors from \cite{muetze-weber:12} outlined before is used as a basis for our construction. However, we add a new ingredient, and this is the concept of \emph{flippable pairs}. Let $P_1,P_1',\ldots,P_\ell,P_\ell'$ be pairwise different paths from the set $\cP$ as constructed before. We call $\cX=\{(P_1,P_1'),\ldots,(P_\ell,P_\ell')\}$ a set of flippable pairs, if for each pair $(P_i,P_i')$ there is an alternative pair of paths $(R_i,R_i')$, where $R_i$ and $R_i'$ are subgraphs of $Q_{2n}(n,n+1)$ such that $R_i$ and $R_i'$ together visit the same vertices as $P_i$ and $P_i'$, and such that $R_i$ and $R_i'$ connect the end vertices of $P_i$ and $P_i'$ the opposite way (see the bottom left of Figure~\ref{fig:idea}, where a flippable pair of paths $(P_2,P_3)$ and the corresponding pair $(R_2,R_3)$ is shown). Note that $R_i$ and $R_i'$ are \emph{not} contained in $\cP$.
We can think of replacing $P_i\circ(0)$ and $P_i'\circ(0)$ in the 2-factor $\cC_{2n+1}$ by the paths $R_i\circ(0)$ and $R_i'\circ(0)$ as a flipping operation (connecting the end vertices of the paths the other way).
Note that this flipping operation can be performed independently for each flippable pair, i.e., from a set of $\ell$ flippable pairs for the set of paths $\cP$ we obtain in total $2^\ell$ different 2-factors from the basic 2-factor $\cC_{2n+1}$ and very precise \emph{local} control over them (this is in stark contrast to what happens when varying the isomorphism $f$ in the above basic construction).
A set of flippable pairs $\cX$ for $\cP$ in the 2-factor $\cC_{2n+1}$ gives rise to the graph $\cG(\cC_{2n+1},\cX)$, in which each cycle of $\cC_{2n+1}$ becomes a node, and two nodes are connected by an edge whenever there is a flippable pair $(P,P')$ in $\cX$ such that $P\circ(0)$ and $P'\circ(0)$ are contained in the corresponding cycles (see the bottom right of Figure~\ref{fig:idea}). Observe that if $\cG(\cC_{2n+1},\cX)$ is connected, then we obtain a Hamilton cycle from the basic 2-factor $\cC_{2n+1}$ by flipping all pairs of paths that form a spanning tree in $\cG(\cC_{2n+1},\cX)$. Moreover, each spanning tree of $\cG(\cC_{2n+1},\cX)$ gives rise to a different Hamilton cycle in the middle layer graph.
This reduction step is crucial: It reduces the problem of proving that a graph has a Hamilton cycle (or many Hamilton cycles) to the problem of proving that some auxiliary graph is connected (or has many spanning trees), which is considerably easier.

Fortunately, flippable pairs are not just a void theoretical concept, but they can be constructed inductively along the lines of the construction of the 2-factor $\cC_{2n+1}$ outlined above. In fact, this construction gives rise to a set $\cX$ of $C_{n-1}$ many flippable pairs. As $2|\cX|/|\cP|=2C_{n-1}/C_n=(n+1)/(2n-1)>1/2$, more than half of all paths from $\cP$ are contained in a flippable pair in $\cX$, which is rather promising (on the other hand, we clearly need exponentially many flippable pairs to connect exponentially many cycles in the 2-factor $\cC_{2n+1}$ to a Hamilton cycle).

\subsubsection{Analysis of the graph \texorpdfstring{$\cG(\cC_{2n+1},\cX)$}{G(C2n+1,X)}}

While describing and proving the inductive constructions of the 2-factor $\cC_{2n+1}$ and the corresponding flippable pairs $\cX$ is relatively straightforward, the analysis of the graph $\cG(\cC_{2n+1},\cX)$ (proving that it is connected and that it has many spanning trees) is rather technical. We show that for one particular choice of construction parameters the cycles of $\cC_{2n+1}$ are in one-to-one correspondence with all plane trees with $n$ edges, i.e., each node of $\cG(\cC_{2n+1},\cX)$ can be interpreted as a plane tree (this rather unexpected correspondence has already been described in \cite{muetze-weber:12}). Moreover, each edge of $\cG(\cC_{2n+1},\cX)$, i.e., each flippable pair from $\cX$, can be interpreted as an elementary transformation between the corresponding plane trees, namely removing a leaf of the tree and attaching it to a different vertex (take a peek at Figure~\ref{fig:g6} below). Proving that $\cG(\cC_{2n+1},\cX)$ is connected then amounts to showing that each plane tree can be transformed into every other plane tree by a sequence of such elementary transformations.
For all these arguments we will repeatedly employ Catalan-type bijections between different sets of combinatorial objects such as certain types of bitstrings, lattice paths and trees.

\subsection{Outline of this paper}

As mentioned before, our construction of Hamilton cycles in the middle layer graph is based on the construction of 2-factors described in our earlier work \cite{muetze-weber:12}. Since both constructions are inherently linked, and since we aim for a self-contained paper, we reproduce some of the required results (including proofs) from \cite{muetze-weber:12} in this paper (without repeatedly mentioning this again). Basically, the contents of Section~\ref{sec:defs-construction}, \ref{sec:correctness} and some parts of Section~\ref{sec:structure-2-factor} of this paper already appeared in \cite{muetze-weber:12}.

In Section~\ref{sec:defs-construction} we describe the basic construction of 2-factors in the middle layer graph.
The proof of a key lemma which ensures that the construction works as claimed is deferred to Section~\ref{sec:correctness}.
In Section~\ref{sec:flipp} we describe the corresponding construction of flippable pairs for those 2-factors.
Moreover, in this section we spell out the details of the abovementioned reduction from a Hamiltonicity to a connectivity problem, and present the proofs of Theorem~\ref{thm:main1} and \ref{thm:main2}. These proofs rely on two propositions (Propositions~\ref{prop:main1} and \ref{prop:main2} below), which state that the graph $\cG(\cC_{2n+1},\cX)$ is connected and that is has the required number of spanning trees.
The rest of the paper is devoted to proving Proposition~\ref{prop:main1} and \ref{prop:main2}, i.e., to analyze the graph $\cG(\cC_{2n+1},\cX)$. Specifically, in Section~\ref{sec:structure-2-factor} we analyze the structure of the 2-factor $\cC_{2n+1}$, and in Section~\ref{sec:structure-flipp} we analyze the structure of the flippable pairs $\cX$. The proofs of Proposition~\ref{prop:main1} and \ref{prop:main2} are completed in Section~\ref{sec:proofs-prop12}.

\section{Construction of 2-factors in the middle layer graph}
\label{sec:defs-construction}

In this section we describe the construction of 2-factors in the middle layer graph outlined in Section~\ref{sec:outline-2factor}.

\subsection{Definitions and notation}
\label{sec:notation}

We start by introducing a few more definitions that will be used throughout the paper.

\textit{Composition of mappings, bijections between combinatorial objects.}
We write the composition of mappings $f,g$ as $f\bullet g$, where $(f\bullet g)(x):=f(g(x))$.
Given a mapping $f$ defined on a set of combinatorial objects $X$, and a bijection $g$ between $X$ and some other set of combinatorial objects $Y$, the function $f$ can be extended in a natural way to a mapping on $Y$ by setting
\begin{equation}
\label{eq:biject-map}
   f:=g\bullet f\bullet g^{-1} \enspace.
\end{equation}
In this paper we specifically deal with functions $f$ defined on certain types of bitstrings, lattice paths or trees (the precise definitions will be given later) and with bijections between these sets, and in understanding $f$ it is often useful to consider how $f$ operates on one of the other sets of objects.

\textit{Notational conventions.}
To simplify notation we regularly adopt the following conventions:
Singleton sets $\{x\}$ are denoted as $x$. For any function $f:X\rightarrow Y$ and any subset $X'\seq X$ we write $f(X'):=\bigcup_{x\in X'} f(x)$. Similarly, for sequences $(x_1,\ldots,x_k)\in X^k$ we write $f(x_1,\ldots,x_k):=(f(x_1),\ldots,f(x_k))$. Furthermore, for any function $f:X_1\times \cdots\times X_k\rightarrow Y$ and subsets $X_i'\seq X_i$, $i=1,\ldots,k$, we define $f(X_1',\ldots,X_k'):=f(X_1'\times \cdots\times X_k')$. For any function $f:X\rightarrow Y$ and any graph $G$ with vertex set $V(G)\seq X$ we denote by $f(G)$ the graph obtained from $G$ by replacing each vertex $v$ by $f(v)$ (so the vertex set of $f(G)$ is $f(V(G))$).

\textit{Reversing, inverting and concatenating bitstrings.}
For any bitstring $x=(x_1,x_2,\ldots,x_n)\in\{0,1\}^n$ we define $\rev(x):=(x_n,x_{n-1},\ldots,x_1)$. Furthermore, setting $\ol{0}:=1$, $\ol{1}:=0$, by the above conventions we have $\ol{x}=(\ol{x_1},\ol{x_2},\ldots,\ol{x_n})$.
For bitstrings $x$ and $y$ we denote by $x\circ y$ the concatenation of $x$ and $y$.
For any bitstring $x$ we define $x^0:=()$ and $x^k:=x\circ x^{k-1}$ for any integer $k\geq 1$.
By these definitions and the above conventions we can write e.g.\ $\rev\big((0)^2\circ\{(1,1),(0,1)\}\circ \ol{(1,0)}\big)=\rev(\{(0,0,1,1,0,1),(0,0,0,1,0,1)\})=\{(1,0,1,1,0,0),(1,0,1,0,0,0)\}$.
Several examples how the concatenation $\circ$ operates on graphs whose vertices are bitstrings were already presented in Section~\ref{sec:outline-structure}.

\textit{Inductive decomposition of the discrete cube.}
In addition to the decomposition of $Q_n$ into layers discussed at the beginning of Section~\ref{sec:outline}, there is another important inductive decomposition of this graph. Note that $Q_n$ consists of a copy of $Q_{n-1}\circ(0)$, a copy of $Q_{n-1}\circ(1)$ and a perfect matching $M_n$ that connects corresponding vertices in the two subgraphs (along the edges of $M_n$, the last bit flips). Unrolling this inductive construction for another step, $Q_n$ is obtained from $Q_{n-2}\circ(0,0)$, $Q_{n-2}\circ(1,0)$, $Q_{n-2}\circ(0,1)$ and $Q_{n-2}\circ(1,1)$ plus two perfect matchings $M_n$ and $M_n':=M_{n-1}\circ(0)\cup M_{n-1}\circ (1)$ (see Figure~\ref{fig:cube-decomposition}). Our inductive construction of 2-factors in the middle layer of $Q_{2n+1}$ is based on this inductive decomposition of $Q_{2n+2}$ into four copies of $Q_{2n}$ plus the two perfect matchings $M_{2n+2}$ and $M_{2n+2}'$.

\begin{figure}
\centering
\PSforPDF{
 \psfrag{b20}{$B_2(0)$}
 \psfrag{b21}{$B_2(1)$}
 \psfrag{b22}{$B_2(2)$}
 \psfrag{b40}{$B_4(0)$}
 \psfrag{b40}{$B_4(0)$}
 \psfrag{b41}{$B_4(1)$}
 \psfrag{b42}{$B_4(2)$}
 \psfrag{b43}{$B_4(3)$}
 \psfrag{b44}{$B_4(4)$}
 \psfrag{q5}{\Large $Q_2$}
 \psfrag{q6}{\Large $Q_4$}
 \psfrag{q1}{$Q_2\circ(0,0)$}
 \psfrag{q2}{$Q_2\circ(1,0)$}
 \psfrag{q3}{$Q_2\circ(0,1)$}
 \psfrag{q4}{$Q_2\circ(1,1)$}
 \psfrag{m30}{$M_3\circ(0)$}
 \psfrag{m31}{$M_3\circ(1)$}
 \psfrag{m4}{$M_4$}
 \psfrag{m4p}{$M_4'=M_3\circ(0)\cup M_3\circ(1)$}
 
 \psfrag{b2n2n}{\small $B_{2n}(2n)$}
 \psfrag{vdots}{$\vdots$}
 \psfrag{b2nnp1}{\small $B_{2n}(n+1)$}
 \psfrag{b2nn}{\small $B_{2n}(n)$}
 \psfrag{b2nnm1}{\small $B_{2n}(n-1)$}
 \psfrag{b2n0}{\small $B_{2n}(0)$}
 \psfrag{b2np22np2}{\small $B_{2n+2}(2n+2)$}
 \psfrag{b2np2np2}{\small $B_{2n+2}(n+2)$}
 \psfrag{b2np2np1}{\small $B_{2n+2}(n+1)$}
 \psfrag{b2np2n}{\small $B_{2n+2}(n)$}
 \psfrag{b2np20}{\small $B_{2n+2}(0)$}
 \psfrag{q2n}{\Large $Q_{2n}$}
 \psfrag{q2np2}{\Large $Q_{2n+2}$}
 \psfrag{q2n1}{$Q_{2n}\circ(0,0)$}
 \psfrag{q2n2}{$Q_{2n}\circ(1,0)$}
 \psfrag{q2n3}{$Q_{2n}\circ(0,1)$}
 \psfrag{q2n4}{$Q_{2n}\circ(1,1)$}
 \psfrag{m1}{$M_{2n+1}\circ(0)$}
 \psfrag{m2}{$M_{2n+1}\circ(1)$}
 \psfrag{m3}{$M_{2n+2}$}
 \psfrag{m2np2p}{$M_{2n+2}'=M_{2n+1}\circ(0)\cup M_{2n+1}\circ(1)$}
 \includegraphics{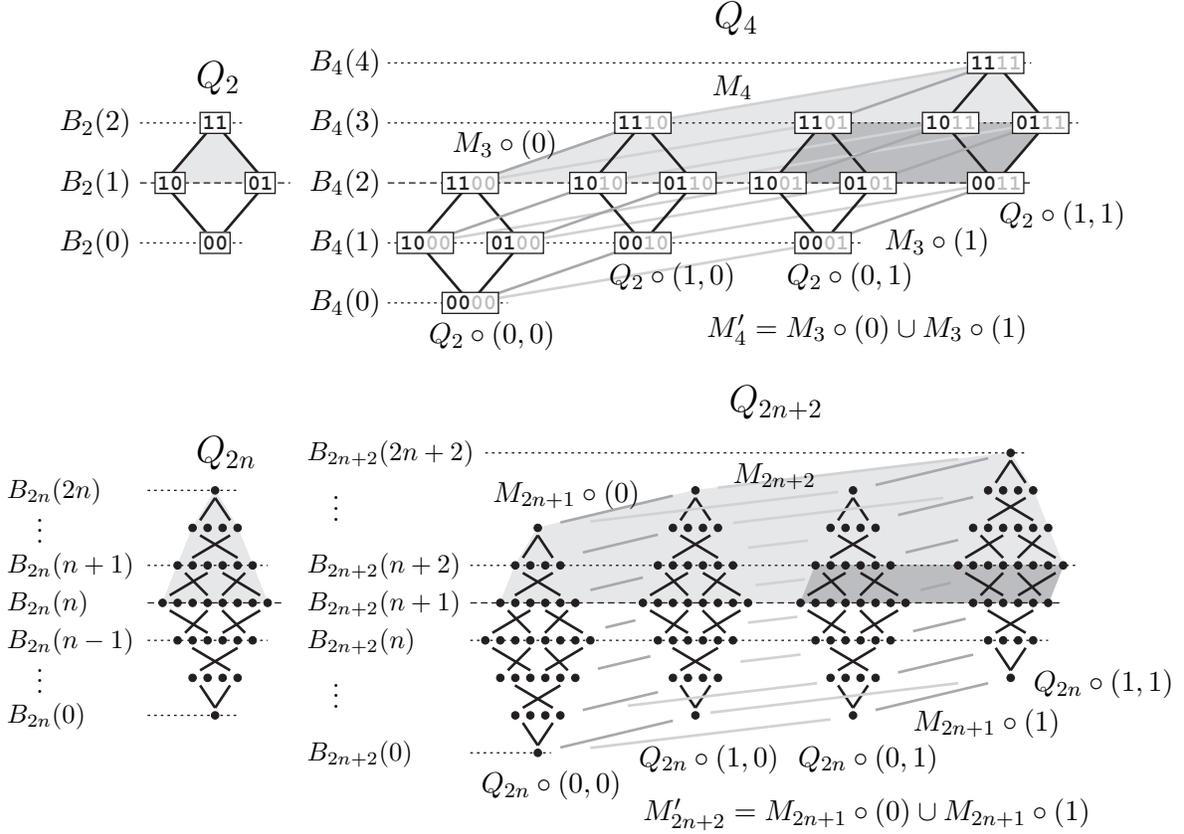}
}
\caption{Decomposition of $Q_{2n+2}$ into four copies of $Q_{2n}$ plus two perfect matchings (the top part shows a concrete example, the bottom part a schematic representation of the general structure). The light grey regions show the upper layers of $Q_{2n}$ and $Q_{2n+2}$ and the dark grey region the middle layer of $Q_{2n+1}\circ(1)$.}
\label{fig:cube-decomposition}
\end{figure}

\textit{Oriented paths, dangling paths.}
In our approach we construct certain paths as subgraphs of layers of the cube. The order of vertices along those paths is important for us, i.e., $P=(v_1,v_2,\ldots,v_\ell)$ is a different \emph{oriented path} than $P'=(v_\ell,v_{\ell-1},\ldots,v_1)$.
For an oriented path $P=(v_1,v_2,\ldots,v_\ell)$ we define $F(P):=v_1$, $S(P):=v_2$ and $L(P):=v_\ell$, as the \emph{first}, \emph{second} and \emph{last} vertex of $P$, respectively.
We refer to a path $P$ in $Q_n(k,k+1)$ that starts and ends at a vertex in the set $B_n(k)$ as a \emph{dangling path}. As $Q_n(k,k+1)$ is bipartite, every second vertex of such a path $P$ is contained in the set $B_n(k+1)$ (and $P$ has even length).

\subsection{Construction of 2-factors}
\label{sec:construction}

The construction is parametrized by some sequence $(\alpha_{2i})_{i\geq 1}$, $\alpha_{2i}\in\{0,1\}^{i-1}$. Given this sequence, we inductively construct a set $\cP_{2n}(k,k+1)$ of disjoint dangling oriented paths in $Q_{2n}(k,k+1)$ for all $n\geq 1$ and all $k=n,n+1,\ldots,2n-1$ such that the following conditions hold:
\begin{enumerate}[(i)]
\item The paths in $\cP_{2n}(n,n+1)$ visit all vertices in the sets $B_{2n}(n+1)$ and $B_{2n}(n)$.
\item For $k=n+1,\ldots,2n-1$, the paths in $\cP_{2n}(k,k+1)$ visit all vertices in the set $B_{2n}(k+1)$, and the only vertices not visited in the set $B_{2n}(k)$ are exactly the elements in the set $S(\cP_{2n}(k-1,k))$.
\end{enumerate}
For simplicity we do not make the dependence of the sets $\cP_{2n}(k,k+1)$ from the parameters $(\alpha_{2i})_{i\geq 1}$ explicit, but we will discuss those dependencies in detail in Section~\ref{sec:dependence-alpha} below.

\textbf{Induction basis $n=1$ ($Q_2$):}
For the induction basis we define
\begin{equation} \label{eq:ind-base-P}
  \cP_{2}(1,2):=\{((1,0),(1,1),(0,1))\} \enspace,
\end{equation}
i.e., the set $\cP_{2}(1,2)$ consists only of a single oriented path on three vertices. It is easily checked that this set of paths in \emph{the} upper layer of $Q_2$ satisfies the conditions~(i) and (ii) (condition~(ii) is satisfied trivially).

\textbf{Induction step $n\rightarrow n+1$ ($Q_{2n}\rightarrow Q_{2n+2}$), $n\geq 1$:}
The inductive construction consists of two intermediate steps. For the reader's convenience those steps are illustrated in Figure~\ref{fig:construction}.

\textit{First intermediate step: Construction of a 2-factor in the middle layer of $Q_{2n+1}$.}
Using only the paths in the set $\cP_{2n}(n,n+1)$ and the parameter $\alpha_{2n}=(\alpha_{2n}(1),\ldots,\alpha_{2n}(n-1))\in\{0,1\}^{n-1}$ we first construct a 2-factor in the middle layer of $Q_{2n+1}$.

Note that the graphs $Q_{2n}(n,n+1)$ and $Q_{2n}(n-1,n)$ are isomorphic to each other. We define an isomorphism $f_{\alpha_{2n}}$ between these graphs as follows: Let $\pi_{\alpha_{2n}}$ denote the permutation on the set $B_{2n}=\{0,1\}^{2n}$ that swaps any two adjacent bits at positions $2i$ and $2i+1$ for all $i=1,\ldots,n-1$, if and only if $\alpha_{2n}(i)=1$, and that leaves the bits at position $1$ and $2n$ unchanged. If e.g.\ $\alpha_{2n}=(0,\ldots,0)$, then no bits are swapped and $\pi_{\alpha_{2n}}=\id$ is simply the identity mapping.
For any bitstring $x\in B_{2n}$ we then define
\begin{equation} \label{eq:f-alpha}
  f_{\alpha_{2n}}(x):=\ol{\rev(\pi_{\alpha_{2n}}(x))} \enspace.
\end{equation}
The fact that this mapping is indeed an isomorphism between the graphs $Q_{2n}(n,n+1)$ and $Q_{2n}(n-1,n)$ follows easily by observing that $\rev(\pi_{\alpha_{2n}}())$ is an automorphism of the graph $Q_{2n}(n,n+1)$ (this mapping just permutes bits).

\begin{figure}
\centering
\PSforPDF{
 \psfrag{b2n2n}{\scriptsize $B_{2n}(2n)$}
 \psfrag{p2n2nm12n}{\scriptsize $\cP_{2n}(2n-1,2n)$}
 \psfrag{vdots}{$\vdots$}
 \psfrag{b2nnp1}{\scriptsize $B_{2n}(n+1)$}
 \psfrag{p2nnnp1}{\scriptsize $\cP_{2n}(n,n+1)$}
 \psfrag{b2nn}{\scriptsize $B_{2n}(n)$}
 \psfrag{fp}{\scriptsize $f_{\alpha_{2n}}(\cP_{2n}(n,n+1))$}
 \psfrag{b2nnm1}{\scriptsize $B_{2n}(n-1)$}
 \psfrag{b2np22np2}{\scriptsize $B_{2n+2}(2n+2)$}
 \psfrag{p2np2np2np3}{\scriptsize $\cP_{2n+2}(n+2,n+3)$}
 \psfrag{b2np2np2}{\scriptsize $B_{2n+2}(n+2)$}
 \psfrag{p2np2np1np2}{\scriptsize $\cP_{2n+2}(n+1,n+2)$}
 \psfrag{b2np2np1}{\scriptsize $B_{2n+2}(n+1)$}
 \psfrag{q2n}{\Large $Q_{2n}$}
 \psfrag{q2np1}{\Large $Q_{2n+1}\circ(1)$}
 \psfrag{q2np2}{\Large $Q_{2n+2}$}
 \psfrag{q2n1}{$Q_{2n}\circ(0,0)$}
 \psfrag{q2n2}{$Q_{2n}\circ(1,0)$}
 \psfrag{q2n3}{$Q_{2n}\circ(0,1)$}
 \psfrag{q2n4}{$Q_{2n}\circ(1,1)$}
 \psfrag{m2}{$M_{2n+1}^{FL}\circ(1)$} 
 \psfrag{m1}{$M_{2n+2}^{S}$}
 \includegraphics{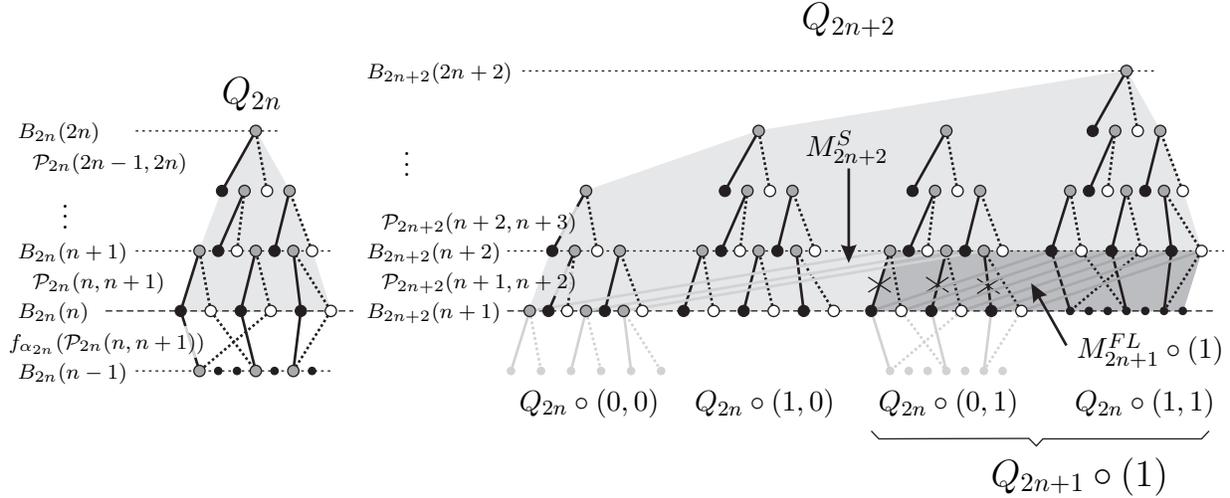}
}
\caption{Schematic illustration of the induction step. The light grey regions show the upper layers of $Q_{2n}$ and $Q_{2n+2}$ and the dark grey region the middle layer of $Q_{2n+1}\circ(1)$. For each dangling oriented path $P=(v_1,v_2,\ldots,v_\ell)$ contained in one of the layers, only the first vertex $F(P)=v_1$ (black), the second vertex $S(P)=v_2$ (grey) and the last vertex $L(P)=v_\ell$ (white) are shown, and the path between the vertices $S(P)=v_2$ and $L(P)=v_\ell$ is represented by a dotted black line (even if this path has more than one edge). The crossed-out edges are deleted from the 2-factor in the middle layer of $Q_{2n+1}\circ(1)$ to construct the paths in $\cP_{2n+2}(n+1,n+2)$.}
\label{fig:construction}
\end{figure}

We will later prove the following crucial lemma.
It states that the sets of first and last vertices of paths from $\cP_{2n}(n,n+1)$ are preserved under the mapping $f_{\alpha_{2n}}$ (see the left hand side of Figure~\ref{fig:construction}). Note however, that the end vertices of one particular path are in general \emph{not} mapped onto themselves.

\begin{lemma} \label{lemma:FL-invariant}
For any $n\geq 1$ and any $\alpha_{2n}\in\{0,1\}^{n-1}$, we have
\begin{equation} \label{eq:FL-invariant}
  f_{\alpha_{2n}}(F(\cP_{2n}(n,n+1)))=F(\cP_{2n}(n,n+1)) \quad\text{and}\quad f_{\alpha_{2n}}(L(\cP_{2n}(n,n+1)))=L(\cP_{2n}(n,n+1)) \enspace,
\end{equation}
where $\cP_{2n}(n,n+1)$ is the set of paths in $Q_{2n}(n,n+1)$ constructed in previous steps for an arbitrary sequence of parameters $(\alpha_{2i})_{1\leq i\leq n-1}$, $\alpha_{2i}\in\{0,1\}^{i-1}$.
\end{lemma}

As explained in Section~\ref{sec:outline-structure}, the middle layer graph $Q_{2n+1}(n,n+1)$ can be decomposed into the graphs $Q_{2n}(n,n+1)\circ(0)$ and $Q_{2n}(n-1,n)\circ(1)$ plus the edges from $M_{2n+1}$ that connect the vertices in the set $B_{2n}(n)\circ(0)$ to the vertices in the set $B_{2n}(n)\circ(1)$ (see the top of Figure~\ref{fig:idea} and the right hand side of Figure~\ref{fig:construction}). Denoting by $M_{2n+1}^{FL}$ the edges from $M_{2n+1}$ that have one end vertex in the set $\big(F(\cP_{2n}(n,n+1))\cup L(\cP_{2n}(n,n+1))\big)\circ(0)\seq B_{2n}(n)\circ(0)$ (and the other in the set $\big(F(\cP_{2n}(n,n+1))\cup L(\cP_{2n}(n,n+1))\big)\circ(1)\seq B_{2n}(n)\circ(1)$), by Lemma~\ref{lemma:FL-invariant} the graph
\begin{equation} \label{eq:2-factor}
  \cC_{2n+1}:=\cP_{2n}(n,n+1)\circ(0)\cup f_{\alpha_{2n}}(\cP_{2n}(n,n+1))\circ(1)\cup M_{2n+1}^{FL}
\end{equation}
is a 2-factor in the middle layer of $Q_{2n+1}$, with the property that on every cycle of $\cC_{2n+1}$, every edge of the form $(F(P),S(P))\circ(0)$ for some $P\in\cP_{2n}(n,n+1)$ is oriented the same way. Even though we are eventually only interested in the 2-factor $\cC_{2n+1}$ defined in \eqref{eq:2-factor}, we need to specify how to proceed with the inductive construction of the sets of paths $\cP_{2n+2}(k,k+1)$.

Before doing this, we define two 2-factors that are obtained from the general definition \eqref{eq:2-factor} for two particular parameter sequences. These 2-factors will become important in later sections of this paper.
We define
\begin{equation}
\label{eq:2-factor-0}
  \cC_{2n+1}^0 := \cC_{2n+1} \quad \text{for} \enspace \alpha_{2i}:=(0,0,\ldots,0)\seq \{0,1\}^{i-1}, \enspace i=1,\ldots,n
\end{equation}
(the all-zero vector is used as a parameter in each construction step).
We also define
\begin{equation}
\label{eq:2-factor-1}
  \cC_{2n+1}^1 := \cC_{2n+1} \quad \text{for} \enspace
  \alpha_{2i}:= \begin{cases}
                (1,1,\ldots,1)\seq \{0,1\}^{i-1}, \enspace i=1,\ldots,n-1 \enspace, \\
                (0,0,\ldots,0)\seq \{0,1\}^{n-1}, \enspace i=n \\
                \end{cases}
\end{equation}
(up to the last construction step the all-one vector is used as a parameter, and in the last step the all-zero vector).

\textit{Second intermediate step: Splitting up the 2-factor into dangling paths.}
We proceed by describing how the sets of paths $\cP_{2n+2}(k,k+1)$ for all $k=n+1,n+2,\ldots,2n+1$ satisfying the conditions~(i) and (ii) are defined, using the previously constructed sets $\cP_{2n}(k,k+1)$ and the 2-factor $\cC_{2n+1}$ defined in the first intermediate step.

Consider the decomposition of $Q_{2n+2}$ into $Q_{2n}\circ(0,0)$, $Q_{2n}\circ(1,0)$, $Q_{2n}\circ(0,1)$ and $Q_{2n}\circ(1,1)$ plus the two perfect matchings $M_{2n+2}$ and $M_{2n+2}'$ as described in Section~\ref{sec:notation}.
For all $k=n+2,\ldots,2n+1$ we define
\begin{equation}
\label{eq:ind-step1-P}
\begin{split}
  \cP_{2n+2}(k,k+1) &:= \cP_{2n}(k,k+1)\circ(0,0)\cup\cP_{2n}(k-1,k)\circ(1,0) \\
                    &\qquad \cup \cP_{2n}(k-1,k)\circ(0,1)\cup\cP_{2n}(k-2,k-1)\circ(1,1)  \enspace,
\end{split}
\end{equation}
where we use the convention $\cP_{2n}(2n,2n+1):=\emptyset$ and $\cP_{2n}(2n+1,2n+2):=\emptyset$ to unify treatment of the sets of paths $\cP_{2n+2}(2n,2n+1)$ and $\cP_{2n+2}(2n+1,2n+2)$ in the two uppermost layers of $Q_{2n+2}$ (see Figure~\ref{fig:construction}). Note that so far none of the edges from the matchings $M_{2n+2}$ or $M_{2n+2}'$ is used.

The definition of the set $\cP_{2n+2}(n+1,n+2)$ is slightly more involved. Note that the graph $Q_{2n+2}(n+1,n+2)$ can be decomposed into $Q_{2n+1}(n+1,n+2)\circ(0)$ and $Q_{2n+1}(n,n+1)\circ(1)$ plus the edges from $M_{2n+2}$ that connect the vertices in the set $B_{2n+1}(n+1)\circ(0)$ to the vertices in the set $B_{2n+1}(n+1)\circ(1)$. The first graph can be further decomposed into $Q_{2n}(n+1,n+2)\circ(0,0)$ and $Q_{2n}(n,n+1)\circ(1,0)$ plus some matching edges that are not relevant here. The second graph is the middle layer of $Q_{2n+1}\circ(1)$. Let $\cC_{2n+1}^-$ denote the graph obtained from the 2-factor $\cC_{2n+1}$  defined in \eqref{eq:2-factor} by removing every edge of the form $(F(P),S(P))\circ(0)$ for some $P\in\cP_{2n}(n,n+1)$ (those edges are crossed out in Figure~\ref{fig:construction}). As on every cycle of $\cC_{2n+1}$ every such edge is oriented the same way, $\cC_{2n+1}^-$ is a set of paths (visiting all vertices of the middle layer of $Q_{2n+1}$), with the property that each of those paths starts at a vertex of the form $S(P)\circ(0)$ and ends at a vertex of the form $F(P')\circ(0)$ for two paths $P,P'\in\cP_{2n}(n,n+1)$. Denoting by $M_{2n+2}^S$ the edges from $M_{2n+2}$ that have one end vertex in the set $S(\cP_{2n}(n,n+1))\circ(0,0)\seq B_{2n}(n+1)\circ(0,0)$ (and the other in the set $S(\cP_{2n}(n,n+1))\circ(0,1)\seq B_{2n}(n+1)\circ(0,1)$), it follows that
\begin{equation} \label{eq:new-paths}
  \cP_{2n+2}^*:=M_{2n+2}^S\cup \cC_{2n+1}^-\circ(1)
\end{equation}
is a set of dangling oriented paths, where we choose the orientation of each path such that the edge from the set $M_{2n+2}^S$ is the first edge (see Figure~\ref{fig:construction}).
Note that we have
\begin{subequations} \label{eq:new-paths-FSL}
\begin{align}
  F(\cP_{2n+2}^*) &= S(\cP_{2n}(n,n+1))\circ(0,0) \enspace, \label{eq:new-paths-F} \\
  S(\cP_{2n+2}^*) &= S(\cP_{2n}(n,n+1))\circ(0,1) \enspace, \label{eq:new-paths-S} \\
  L(\cP_{2n+2}^*) &= F(\cP_{2n}(n,n+1))\circ(0,1) \enspace. \label{eq:new-paths-L}
\end{align}
\end{subequations}
We then define
\begin{equation} \label{eq:ind-step2-P}
  \cP_{2n+2}(n+1,n+2):=\cP_{2n}(n+1,n+2)\circ(0,0)\cup\cP_{2n}(n,n+1)\circ(1,0)\cup \cP_{2n+2}^* \enspace,
\end{equation}
where in the case $n=1$ we use the convention $\cP_2(2,3):=\emptyset$.

We now argue that the sets of paths $\cP_{2n+2}(k,k+1)$, $k=n+1,n+2,\ldots,2n+1$, defined in \eqref{eq:ind-step1-P} and \eqref{eq:ind-step2-P} satisfy the conditions~(i) and (ii). For every $k=n+3,\ldots,2n+1$, by the definition in \eqref{eq:ind-step1-P} and by induction, the paths in $\cP_{2n+2}(k,k+1)$ visit all vertices in the set
\begin{equation*}
  B_{2n}(k+1)\circ(0,0)\cup B_{2n}(k)\circ(1,0)\cup B_{2n}(k)\circ(0,1)\cup B_{2n}(k-1)\circ(1,1)=B_{2n+2}(k+1) \enspace,
\end{equation*}
and the only vertices not visited in the set $B_{2n+2}(k)$ are exactly the elements in the set
\begin{multline*}
  S(\cP_{2n}(k-1,k))\circ(0,0)\cup S(\cP_{2n}(k-2,k-1))\circ(1,0) \\
  \cup S(\cP_{2n}(k-2,k-1))\circ(0,1)\cup S(\cP_{2n}(k-3,k-2))\circ(1,1) \enspace.
\end{multline*}
As for those $k$ the set of paths $\cP_{2n+2}(k-1,k)$ in the layer below is also defined via \eqref{eq:ind-step1-P}, this set is equal to $S(\cP_{2n+2}(k-1,k))$, proving that $\cP_{2n+2}(k,k+1)$ indeed satisfies condition~(ii).

By the definition in \eqref{eq:ind-step1-P} and by induction, the paths in the set $\cP_{2n+2}(n+2,n+3)$ visit all vertices in the set $B_{2n+2}(n+3)$, and the only vertices not visited in the set $B_{2n+2}(n+2)$ are exactly the elements in the set
\begin{equation*}
  S(\cP_{2n}(n+1,n+2))\circ(0,0)\cup S(\cP_{2n}(n,n+1))\circ(1,0)\cup S(\cP_{2n}(n,n+1))\circ(0,1) \enspace.
\end{equation*}
By the definition in \eqref{eq:ind-step2-P} and by \eqref{eq:new-paths-S} this set is equal to $S(\cP_{2n+2}(n+1,n+2))$, proving that $\cP_{2n+2}(n+2,n+3)$ indeed satisfies condition~(ii).

It remains to show that the set $\cP_{2n+2}(n+1,n+2)$ satisfies condition~(i). This follows directly from the definitions in \eqref{eq:new-paths} and \eqref{eq:ind-step2-P} and by induction, using that the paths in $\cC_{2n+1}^-\circ(1)$ visit \emph{all} vertices in the middle layer of $Q_{2n+1}\circ(1)$ (recall that those paths were obtained from a 2-factor in this graph), and that the only vertices in $Q_{2n}(n+1,n+2)\circ(0,0)$ not visited by the paths in $\cP_{2n}(n+1,n+2)\circ(0,0)$ are exactly the first vertices of the paths $\cP_{2n+2}^*$ (see~\eqref{eq:new-paths-F}).

\subsection{Dependence on the parameter sequence}
\label{sec:dependence-alpha}

It follows inductively from our construction that for all $k=n,\ldots,2n-1$, the set of paths $\cP_{2n}(k,k+1)$ depends on all parameters $\alpha_2,\alpha_4,\ldots,\alpha_{2(2n-1-k)}$, and that the 2-factor $\cC_{2n+1}$ defined in \eqref{eq:2-factor} depends on all parameters $\alpha_2,\alpha_4,\ldots,\alpha_{2n}$.

Even though the paths in the sets $\cP_{2n}(k,k+1)$ depend on the parameter sequence $(\alpha_{2i})_{i\geq 1}$, it follows from Lemma~\ref{lemma:FL-invariant} that the sets of first, second and last vertices of those paths do not depend on the sequence $(\alpha_{2i})_{i\geq 1}$. In particular, the number of paths in the sets $\cP_{2n}(k,k+1)$ is independent of $(\alpha_{2i})_{i\geq 1}$ (those numbers are already fixed by the conditions~(i) and (ii) from Section~\ref{sec:construction} and the cardinalities of the sets $B_{2n}(k)$, $k=n,n+1,\ldots,2n$). Note moreover that the pairs $(F(P),S(P))$ for all paths $P\in\cP_{2n}(k,k+1)$ are the same regardless of the sequence $(\alpha_{2i})_{i\geq 1}$ (which last vertex $L(P)$ from the set of all last vertices belongs to this path does of course depend on the chosen parameter sequence).

As $\alpha_{2i}\in\{0,1\}^{i-1}$, our construction yields at most $\prod_{i=1}^{n} 2^{i-1}=2^{\binom{n}{2}}$ different 2-factors in the middle layer of $Q_{2n+1}$. It can be shown that all these 2-factors are indeed different subgraphs of $Q_{2n+1}(n,n+1)$ (see \cite[Theorem~6]{muetze-weber:12}).
However, as mentioned in Section~\ref{sec:outline-2factor}, only few of them seem to be amenable to theoretical analysis. Essentially only the parameter choices in \eqref{eq:2-factor-0} and \eqref{eq:2-factor-1} yield a 2-factor with a well-understood cycle structure. We shall see in Section~\ref{sec:structure-2-factor} that these two 2-factors are actually intimately related --- in particular, they have the same number of cycles.

\section{Construction of flippable pairs and proofs of Theorem~\texorpdfstring{\ref{thm:main1}}{1} and \texorpdfstring{\ref{thm:main2}}{2}}
\label{sec:flipp}

In this section we introduce the concept of flippable pairs outlined in Section~\ref{sec:outline-flipp}, show how flippable pairs can be constructed inductively along the lines of the construction of 2-factors presented in the previous section, and finally show how they can be used to prove Theorem~\ref{thm:main1} and \ref{thm:main2}. These proofs rely on two propositions (Proposition~\ref{prop:main1} and \ref{prop:main2} below) that will be proved in later sections of this paper.

\textit{Flippable pairs.}
Let $\cP$ be a set of disjoint oriented paths in a graph $G$, and let $P_1,P_1',\ldots,P_\ell,P_\ell'$ be pairwise different paths from $\cP$.
We call $\cX=\{(P_1,P_1'),\ldots,(P_\ell,P_\ell')\}$ \emph{a set of flippable pairs for $\cP$ in $G$}, if for every $i=1,\ldots,\ell$, there are two oriented paths $R_i$ and $R_i'$ in $G$ such that $V(P_i)\cup V(P_i')=V(R_i)\cup V(R_i')$ and such that $F(P_i)=F(R_i)$, $F(P_i')=F(R_i')$, $L(P_i)=L(R_i')$ and $L(P_i')=L(R_i)$.
The first condition states that the paths $R_i$ and $R_i'$ together visit the same vertices as $P_i$ and $P_i'$, and the second condition states that the matching between first and last vertices of the pairs $(P_i,P_i')$ and $(R_i,R_i')$ is interchanged/flipped. Note that $R_i$ and $R_i'$ are \emph{not} contained in $\cP$.
We call pairs of paths $(R_1,R_1'),\ldots,(R_\ell,R_\ell')$ satisfying these conditions \emph{flipped pairs} corresponding to $\cX$.

We state the following simple and general observation for further reference. For any graph $G$ and any subset $U\seq V(G)$ we denote by $G[U]$ the subgraph of $G$ induced by the vertices in $U$.

\begin{lemma}
\label{lemma:flipp-propagation}
Let $G$ be a graph, and let $U$ and $U'$ be disjoint subsets of $V(G)$. Furthermore, let $\cP$ and $\cP'$ be sets of disjoint oriented paths in $G[U]$ and $G[U']$, and let $\cX$ and $\cX'$ be sets of flippable pairs for $\cP$ and $\cP'$, respectively. Then $\cX\cup\cX'$ is a set of flippable pairs for $\cP\cup\cP'$ in $G$.
\end{lemma}

Consider the set of paths $\cP_{2n}(n,n+1)$ and the 2-factor $\cC_{2n+1}$ defined in Section~\ref{sec:construction} for an arbitrary parameter sequence $(\alpha_{2i})_{i\geq 1}$, $\alpha_{2i}\in\{0,1\}^{i-1}$. Suppose we are given a set of flippable pairs $\cX=\{(P_1,P_1'),\ldots,(P_\ell,P_\ell')\}$ for $\cP_{2n}(n,n+1)$ in the graph $Q_{2n}(n,n+1)$.
We can think of replacing the paths $P_i\circ(0)$ and $P_i'\circ(0)$ in $\cC_{2n+1}$ (recall \eqref{eq:2-factor}) by corresponding flipped paths $R_i\circ(0)$ and $R_i'\circ(0)$ as a flipping operation (connecting the first and last vertices of these paths the other way, see the bottom left of Figure~\ref{fig:idea}). As this flipping operation can be performed independently for each flippable pair, we obtain in total $2^\ell$ different 2-factors from the basic 2-factor $\cC_{2n+1}$ and very precise \emph{local} control over them.
As mentioned before, this is in stark contrast to what happens when varying the parameter $\alpha_{2n}$ of the isomorphism $f_{\alpha_{2n}}$ in our basic construction: Varying $\alpha_{2n}$ in \eqref{eq:2-factor} affects many of the paths from $f_{\alpha_{2n}}(\cP_{2n}(n,n+1))$ simultaneously und therefore affects the resulting 2-factor $\cC_{2n+1}$ \emph{globally} (in a way that is hard to control).

\subsection{The graph \texorpdfstring{$\cG(\cC_{2n+1},\cX)$}{G(C2n+1,X)}}
\label{sec:gCX}

These insights motivate us to define a directed multigraph $\cG(\cC_{2n+1},\cX)$ as follows: The nodes of $\cG(\cC_{2n+1},\cX)$ are the cycles of $\cC_{2n+1}$. For each flippable pair $(P,P')\in\cX$, we consider the cycles $C,C'\in\cC_{2n+1}$ that contain the paths $P\circ (0)$ and $P'\circ(0)$, respectively (recall \eqref{eq:2-factor}), and add a directed edge from $C$ to $C'$ to the graph $\cG(\cC_{2n+1},\cX)$. Note that this graph may have multiple edges and/or loops.
The definition of the graph $\cG(\cC_{2n+1},\cX)$ is illustrated at the bottom of Figure~\ref{fig:idea}.

The following two crucial lemmas reduce the problem of proving that the middle layer graph has a Hamilton cycle (or many Hamilton cycles) to the problem of proving that $\cG(\cC_{2n+1},\cX)$ is connected (or has many spanning trees), which is considerably easier.

As customary, we call a directed (multi)graph \emph{weakly} connected, if replacing all directed edges by undirected edges yields a connected (multi)graph.

\begin{lemma}
\label{lemma:GCS-connected}
For any $n\geq 1$, if the graph $\cG(\cC_{2n+1},\cX)$ is weakly connected, then the middle layer graph $Q_{2n+1}(n,n+1)$ has a Hamilton cycle.
\end{lemma}

\begin{proof}
Let $\cP_{2n}(n,n+1)$ be the set of paths defined in Section~\ref{sec:construction}. Furthermore, let $(P_1,P_1'),\ldots,(P_\ell,P_\ell')$ be the flippable pairs in $\cX$ for $\cP_{2n}(n,n+1)$, and let $(R_1,R_1'),\ldots,(R_\ell,R_\ell')$ be corresponding flipped pairs (these are paths in $Q_{2n}(n,n+1)$).
We fix any (not necessarily spanning) subtree $T$ of the graph $\cG(\cC_{2n+1},\cX)$, and let $I_T\seq\{1,\ldots,\ell\}$ be the set of all indices $i$ of flippable pairs $(P_i,P_i')$ that correspond to edges of $T$.
We define $\cP_T:=(\cP_{2n}(n,n+1)\setminus\{P_i,P_i'\mid i\in I_T\})\cup \{R_i,R_i'\mid i\in I_T\}$.
Observe that by the definition of flippable pairs and by the structure of the 2-factor $\cC_{2n+1}$ defined in \eqref{eq:2-factor}, the graph
\begin{equation}
\label{eq:2-factor-x}
  \cC_T := \cP_T\circ (0)\cup f_{\alpha_{2n}}(\cP_{2n}(n,n+1))\circ (1) \cup M_{2n+1}^{FL}
\end{equation}
is a 2-factor in the middle layer graph $Q_{2n+1}(n,n+1)$ whose cycle structure differs from the cycle structure of $\cC_{2n+1}$ only in that all cycles in $V(T)$ (the nodes of $T$ correspond to cycles of $\cC_{2n+1}$) are joined to a single cycle (all other cycles are exactly the same).
In particular, if $T$ is a spanning tree of $\cG(\cC_{2n+1},\cX)$, then $\cC_T$ is a Hamilton cycle of $Q_{2n+1}(n,n+1)$.
\end{proof}

\begin{lemma}
\label{lemma:GCS-spanning}
For any $n\geq 1$, if the graph $\cG(\cC_{2n+1},\cX)$ has $t$ different spanning trees, then the middle layer graph $Q_{2n+1}(n,n+1)$ has at least $t$ different Hamilton cycles.
\end{lemma}

As in the case of connectedness, for spanning trees the direction of edges is also irrelevant for us.

\begin{proof}
The proof is a straightforward extension of the proof of Lemma~\ref{lemma:GCS-connected} and follows by observing that if $T$ and $T'$ are different subtrees of $\cG(\cC_{2n+1},\cX)$, then the 2-factors $\cC_T$ and $\cC_{T'}$ defined in \eqref{eq:2-factor-x} are different subgraphs of $Q_{2n+1}(n,n+1)$ (in particular, different spanning trees yield different Hamilton cycles).
\end{proof}

In the following we show how to construct a set of flippable pairs $\cX$ such that the resulting graph $\cG(\cC_{2n+1},\cX)$ satisfies the preconditions of Lemma~\ref{lemma:GCS-connected} and \ref{lemma:GCS-spanning} (with $t=2^{2^{\Omega(n)}}$). So after all, flippable pairs are not just a void theoretical concept, but they are very useful for proving Hamiltonicity results about the middle layer graph.

\subsection{Construction of flippable pairs}
\label{sec:construction-flipp}

Let $\cP_{2n}(k,k+1)$ be the sets of paths defined in Section~\ref{sec:construction} for an arbitrary parameter sequence $(\alpha_{2i})_{i\geq 1}$, $\alpha_{2i}\in\{0,1\}^{i-1}$.
In the following we show how to inductively construct a set of flippable pairs $\cX_{2n}(k,k+1)$ for the set $\cP_{2n}(k,k+1)$ for all $n\geq 2$ and all $k=n,n+1,\ldots,2n-1$. This construction arises very naturally from the inductive construction of the sets $\cP_{2n}(k,k+1)$ and is based on Lemma~\ref{lemma:flipp-propagation}. 

\textbf{Induction basis $n=2$ ($Q_4$):}
Consider the set of paths $\cP_4(2,3)=\{P,P'\}$ in $Q_4(2,3)$ with
\begin{subequations}
\label{eq:ind-base-flipp}
\begin{align}
  P &:= ((1,1,0,0),(1,1,0,1),(0,1,0,1),(0,1,1,1),(0,0,1,1),(1,0,1,1),(1,0,0,1)) \enspace, \\
  P' &:= ((1,0,1,0),(1,1,1,0),(0,1,1,0)) \enspace,
\end{align}
arising from the basic inductive construction after one step (as $\alpha_2=()$ this step does not involve any parameter choices yet).
The set
\begin{equation}
  \cX_4(2,3):=\{(P,P')\}
\end{equation}
is a set of flippable pairs for $\cP_4(2,3)$, which can be seen by considering the flipped paths $(R,R')$ in $Q_4(2,3)$ defined by
\begin{align*}
  R &:= ((1,1,0,0),(1,1,1,0),(0,1,1,0)) \enspace, \\
  R' &:= ((1,0,1,0),(1,0,1,1),(0,0,1,1),(0,1,1,1),(0,1,0,1),(1,1,0,1),(1,0,0,1)) \enspace.
\end{align*}
For completeness we also define
\begin{equation}
  \cX_4(3,4):=\emptyset \enspace,
\end{equation}
\end{subequations}
which is trivially a set of flippable pairs for the set of paths $\cP_4(3,4)$ in $Q_4(3,4)$ arising from the basic construction.

\textbf{Induction step $n\rightarrow n+1$ ($Q_{2n}\rightarrow Q_{2n+2}$), $n\geq 2$:}
Consider the sets of flippable pairs $\cX_{2n}(k,k+1)$, $k=n,n+1,\ldots,2n-1$, for the sets of paths $\cP_{2n}(k,k+1)$. In the following we describe how to use them to construct sets of flippable pairs $\cX_{2n+2}(k,k+1)$, $k=n+1,n+2,\ldots,2n+1$, for the sets $\cP_{2n+2}(k,k+1)$ in $Q_{2n+2}(k,k+1)$.

For all $k=n+2,\ldots,2n+1$ we define, in analogy to \eqref{eq:ind-step1-P},
\begin{equation}
\label{eq:ind-step1-flipp}
\begin{split}
  \cX_{2n+2}(k,k+1) &:= \cX_{2n}(k,k+1)\circ(0,0)\cup\cX_{2n}(k-1,k)\circ(1,0) \\
                    &\qquad \cup \cX_{2n}(k-1,k)\circ(0,1)\cup\cX_{2n}(k-2,k-1)\circ(1,1) \enspace,
\end{split}
\end{equation}
where we use the convention $\cX_{2n}(2n,2n+1):=\emptyset$ and $\cX_{2n}(2n+1,2n+2):=\emptyset$ to unify treatment of the sets of flippable pairs $\cX_{2n+2}(2n,2n+1)$ and $\cX_{2n+2}(2n+1,2n+2)$ in the two uppermost layers of $Q_{2n+2}$.
The sets of flippable pairs on the right hand side of \eqref{eq:ind-step1-flipp} clearly lie in four disjoint subgraphs of $Q_{2n+2}(k,k+1)$, so by Lemma~\ref{lemma:flipp-propagation} and by induction $\cX_{2n+2}(k,k+1)$ is indeed a set of flippable pairs for $\cP_{2n+2}(k,k+1)$.

To define the set $\cX_{2n+2}(n+1,n+2)$, we consider the oriented paths $\cP_{2n+2}^*$ defined in \eqref{eq:new-paths} (recall that these paths originate from splitting up the 2-factor $\cC_{2n+1}$ defined in \eqref{eq:2-factor}).
By \eqref{eq:2-factor} and \eqref{eq:new-paths}, every path $P^+\in\cP_{2n+2}^*$ has the following structure (see the right hand side of Figure~\ref{fig:construction}): There are two paths $P,\Phat\in\cP_{2n}(n,n+1)$ with $f_{\alpha_{2n}}(L(\Phat))=L(P)$ such that $P^+$ contains all edges except the first one from $P\circ(0,1)$ and all edges from $f_{\alpha_{2n}}(\Phat)\circ(1,1)$ ($P^+$ has three more edges, two from the matching $M_{2n+1}^{FL}\circ(1)$ and one from the matching $M_{2n+2}^S$).
By this structural property of paths from $\cP_{2n+2}^*$ and the fact that $f_{\alpha_{2n}}$ is an isomorphism between the graphs $Q_{2n}(n,n+1)$ and $Q_{2n}(n-1,n)$, the set
\begin{equation}
\label{eq:new-flipp-paths}
\begin{split}
  \cX_{2n+2}^* &:= \big\{(P^+,P^{+'}) \mid P^+,P^{+'}\in\cP_{2n+2}^* \text{ and there is a flippable pair } (\Phat,\Phat')\in\cX_{2n}(n,n+1) \\
               &\hspace{3cm} \text{ with } f_{\alpha_{2n}}(\Phat)\circ(1,1) \seq P^+ \;\wedge\; f_{\alpha_{2n}}(\Phat')\circ(1,1) \seq P^{+'} \big\} \enspace.
\end{split}
\end{equation}
is a set of flippable pairs for $\cP_{2n+2}^*$.
We now define, in analogy to \eqref{eq:ind-step2-P},
\begin{equation}
\label{eq:ind-step2-flipp}
  \cX_{2n+2}(n+1,n+2):=\cX_{2n}(n+1,n+2)\circ(0,0)\cup\cX_{2n}(n,n+1)\circ(1,0)\cup \cX_{2n+2}^* \enspace.
\end{equation}
The sets of flippable pairs on the right hand side of \eqref{eq:ind-step2-flipp} lie in three disjoint subgraphs of $Q_{2n+2}(n+1,n+2)$, so by Lemma~\ref{lemma:flipp-propagation} and by induction $\cX_{2n+2}(n+1,n+2)$ is indeed a set of flippable pairs for $\cP_{2n+2}(n+1,n+2)$.

In analogy to \eqref{eq:2-factor-0} and \eqref{eq:2-factor-1}, we define two sets of flippable pairs that are obtained from the general definitions above for two particular parameter sequences by
\begin{align}
  \cX_{2n}^0(k,k+1) &:= \cX_{2n}(k,k+1) \quad \text{for} \enspace \alpha_{2i}:=(0,0,\ldots,0)\seq \{0,1\}^{i-1}, \enspace i=1,\ldots,n-1 \enspace, \label{eq:flipp-0} \\
  \cX_{2n}^1(k,k+1) &:= \cX_{2n}(k,k+1) \quad \text{for} \enspace \alpha_{2i}:=(1,1,\ldots,1)\seq \{0,1\}^{i-1}, \enspace i=1,\ldots,n-1 \label{eq:flipp-1}
\end{align}
(the all-zero vector or the all-one vector is used as a parameter in each construction step, respectively).

\subsection{Proofs of Theorem~\texorpdfstring{\ref{thm:main1}}{1} and \texorpdfstring{\ref{thm:main2}}{2}}

The rest of this paper is devoted to proving the following two propositions:

\begin{proposition}
\label{prop:main1}
Let $\cC_{2n+1}^1$ be the 2-factor defined in \eqref{eq:2-factor-1} and $\cX_{2n}^1(n,n+1)$ the set of flippable pairs defined in \eqref{eq:flipp-1}. For any $n\geq 1$, the graph $\cG(\cC_{2n+1}^1,\cX_{2n}^1(n,n+1))$ defined in Section~\ref{sec:gCX} is weakly connected.
\end{proposition}

\begin{proposition}
\label{prop:main2}
Let $\cC_{2n+1}^1$ be the 2-factor defined in \eqref{eq:2-factor-1} and $\cX_{2n}^1(n,n+1)$ the set of flippable pairs defined in \eqref{eq:flipp-1}. For any $n\geq 1$, the graph $\cG(\cC_{2n+1}^1,\cX_{2n}^1(n,n+1))$ defined in Section~\ref{sec:gCX} has at least $\frac{1}{4}2^{2^{\lfloor (n+1)/4\rfloor}}$ different spanning trees.
\end{proposition}

With these propositions at hand, proving Theorem~\ref{thm:main1} and \ref{thm:main2} is straightforward.

\begin{proof}[Proof of Theorem~\ref{thm:main1}]
Combine Lemma~\ref{lemma:GCS-connected} and Proposition~\ref{prop:main1}.
\end{proof}

\begin{proof}[Proof of Theorem~\ref{thm:main2}]
Combine Lemma~\ref{lemma:GCS-spanning} and Proposition~\ref{prop:main2}.
\end{proof}

\begin{remark}
\label{remark:C0C1}
We remark that in contrast to $\cG(\cC_{2n+1}^1,\cX_{2n}^1(n,n+1))$, the graph $\cG(\cC_{2n+1}^0,\cX_{2n}^0(n,n+1))$ is not connected (so it is not useful for proving Hamiltonicity of the middle layer graph). However, the 2-factor $\cC_{2n+1}^0$ will be crucial in understanding the 2-factor $\cC_{2n+1}^1$ (as indicated before, both 2-factors are intimately related and have the same number of cycles).
\end{remark}

\section{Correctness of the construction}
\label{sec:correctness}

In this section we prove Lemma~\ref{lemma:FL-invariant}, thus showing that the construction of 2-factors described in Section~\ref{sec:construction} indeed works as claimed. Our proof strategy is as follows: After setting up some machinery that relates bitstrings to another set of combinatorial objects, namely lattice paths, we consider an abstract recursion over sets of bitstrings and show that the solutions of this recursion correspond to certain sets of lattice paths. It will then be easy to convince ourselves that the sets of first, second and last vertices of the oriented paths in the sets $\cP_{2n}(k,k+1)$ arising in our basic construction satisfy exactly this abstract recursion, which allows us to apply our knowledge from the world of lattice paths and to derive Lemma~\ref{lemma:FL-invariant}.

\subsection{Bitstrings and lattice paths}
\label{sec:lattice-paths}

We begin by introducing some terminology related to lattice paths in $\mathbb{Z}^2$, explain the relation of those combinatorial objects to bitstrings (these are the vertex labels of $Q_n$ and thus the objects our basic construction works with), and establish an invariance property of certain sets of lattice paths (Lemma~\ref{lemma:dyck-paths-invariant} below).

\textit{Lattice paths, Dyck paths.}
For any $n\geq 0$ we denote by $P_n$ the set of lattice paths in $\mathbb{Z}^2$ that start at $(0,0)$ and move $n$ steps, each of which changes the current coordinate by either $(+1,+1)$ or $(+1,-1)$. We refer to such a step as an upstep or downstep, respectively.
For any $n\geq 0$ and $k\geq 0$ we denote by $D_n(k)$ the set of lattice paths from $P_n$ that never move below the line $y=0$ and that have exactly $k$ upsteps.
Note that such a path has $n-k$ downsteps and therefore ends at $(n,2k-n)$.
For $n\geq 1$ we define $D_n^{>0}(k)\seq D_n(k)$ as the set of lattice paths that have no point of the form $(x,0)$, $1\leq x\leq n$, and $D_n^{=0}(k)\seq D_n(k)$ as the set of lattice paths that have at least one point of the form $(x,0)$, $1\leq x\leq n$. For $n=0$ we define $D_0^{=0}(0):=\{()\}$, where $()$ denotes the empty lattice path, and $D_0^{>0}(0):=\emptyset$.
We clearly have $D_n(k)=D_n^{=0}(k)\cup D_n^{>0}(k)$.
Furthermore, for $n\geq 1$ we let $D_n^-(k)$ denote the set of lattice paths from $P_n$ that move below the line $y=0$ exactly once and that have exactly $k$ upsteps (such a path has exactly one point of the form $(x,-1)$, $1\leq x\leq n$).
Depending on the values of $n$ and $k$ the sets of lattice paths we just defined might be empty. E.g., we have $D_{2n}^{>0}(n)=\emptyset$ and therefore $D_{2n}(n)=D_{2n}^{=0}(n)$.

Given two lattice paths $p$ and $q$, we denote by $p\circ q$ the lattice path obtained by gluing the first point of $q$ onto the last point of $p$ (the first point of $p\circ q$ is the same as the first point of $p$).
We sometimes identify a lattice path $p\in P_n$ with its step sequence $p=(p_1,\ldots,p_n)$, $p_i\in\{\upstep,\downstep\}$, where $p_i=\upstep$ if the $i$-th step of $p$ is an upstep and $p_i=\downstep$ if the $i$-th step of $p$ is a downstep.
Using these notations we clearly have for $*\in\{=0,-\}$ ($*$ is a symbolic placeholder, denoting the superscripts $=0$ and $-$), all $n\geq 1$ and all $k=n+2,\ldots,2n+1$ that
\begin{subequations} \label{eq:D-partitions}
\begin{align}
  D_{2n+2}^*(k) &= D_{2n}^*(k)\circ (\downstep,\downstep)
                 \cup D_{2n}^*(k-1)\circ (\upstep,\downstep) \notag \\
                 &\qquad \cup D_{2n}^*(k-1)\circ (\downstep,\upstep)
                               \cup D_{2n}^*(k-2)\circ (\upstep,\upstep) \enspace, \label{eq:D2np2-*-u-partition} \\
  D_{2n+2}^{>0}(k+1) &= D_{2n}^{>0}(k+1)\circ (\downstep,\downstep)
                 \cup D_{2n}^{>0}(k)\circ (\upstep,\downstep) \notag \\
                 &\qquad \cup D_{2n}^{>0}(k)\circ (\downstep,\upstep)
                               \cup D_{2n}^{>0}(k-1)\circ (\upstep,\upstep) \enspace. \label{eq:D2np2-g0-u-partition}
\end{align}
Similarly, for all $n\geq 1$ we have
\begin{align}
  D_{2n+2}^{=0}(n+1) &= \big(D_{2n}^{=0}(n+1)\cup D_{2n}^{>0}(n+1)\big) \circ (\downstep,\downstep) \cup
                             D_{2n}^{=0}(n)\circ (\upstep,\downstep) \enspace, \label{eq:D2np2-eq0-m-partition} \\
  D_{2n+2}^{>0}(n+2) &= D_{2n}^{>0}(n+2)\circ (\downstep,\downstep) \cup
                        D_{2n}^{>0}(n+1)\circ (\upstep,\downstep) \cup
                        D_{2n}^{>0}(n+1)\circ (\downstep,\upstep) \enspace, \label{eq:D2np2-g0-m-partition} \\
  D_{2n+2}^-(n+1) &= D_{2n}^-(n+1)\circ (\downstep,\downstep) \cup
                     D_{2n}^-(n)\circ (\upstep,\downstep) \cup
                     D_{2n}^{=0}(n)\circ (\downstep,\upstep) \enspace. \label{eq:D2np2-m-m-partition}
\end{align}
\end{subequations}
Note that all the unions in \eqref{eq:D-partitions} are disjoint and that some of the sets participating in the unions might be empty.

\textit{Bijection $\varphi$ between bitstrings and lattice paths.}
For any $x\in B_n=\{0,1\}^n$, $x=(x_1,\ldots,x_n)$, we define $\varphi(x)$ as the lattice path from $P_n$ whose $i$-th step is an upstep if $x_i=1$ and a downstep if $x_i=0$. Note that the step sequence of $\varphi(x)$ is obtained from $(x_1,\ldots,x_n)$ by replacing every $1$ by $\upstep$ and every $0$ by $\downstep$. This mapping is clearly a bijection between $B_n$ and $P_n$.

\textit{The mappings $\ol{\rev}$, $\pi_{\alpha_{2n}}$ and $f_{\alpha_{2n}}$ on lattice paths.}
Via the bijection $\varphi$, the operation $\ol{\rev}$ of reversing and inverting a bitstring can be extended naturally to lattice paths (recall \eqref{eq:biject-map}).
Note that $\ol{\rev}$ simply mirrors every lattice path from the set $P_{2n}$ with endpoint $(2n,0)$ along the axis $x=n$.
In a similar fashion we also extend the mappings $\pi_{\alpha_{2n}}$ and $f_{\alpha_{2n}}$, defined around \eqref{eq:f-alpha} as mappings on the set $B_{2n}$, to mappings on the set $P_{2n}$.
Note that $\pi_{\alpha_{2n}}$ swaps the order of any two adjacent steps $2i$ and $2i+1$, $i=1,\ldots,n-1$, of a given lattice path from $P_{2n}$, if and only if $\alpha_{2n}(i)=1$.

\begin{lemma} \label{lemma:dyck-paths-invariant}
For any $n\geq 1$ and any $\alpha_{2n}\in\{0,1\}^{n-1}$ the mapping $f_{\alpha_{2n}}$ defined in \eqref{eq:f-alpha}
 (viewed as a mapping $P_{2n}\rightarrow P_{2n}$) satisfies $f_{\alpha_{2n}}(D_{2n}^{=0}(n))=D_{2n}^{=0}(n)$ and $f_{\alpha_{2n}}(D_{2n}^-(n))=D_{2n}^-(n)$.
\end{lemma}

Note that even though the sets $D_{2n}^{=0}(n)$ and $D_{2n}^-(n)$ are invariant under the mapping $f_{\alpha_{2n}}$, changing the parameter $\alpha_{2n}$ will of course change the images of certain lattice paths from those sets.

\begin{proof}
The lemma follows from \eqref{eq:f-alpha} by showing that each of the mappings $\overline{\rev}$ and $\pi_{\alpha_{2n}}$ (viewed as mappings $P_{2n}\rightarrow P_{2n}$) maps each of the sets $D_{2n}^{=0}(n)$ and $D_{2n}^-(n)$ onto itself.
For the mapping $\overline{\rev}$ this is trivial, as $\overline{\rev}$ simply mirrors every lattice path from the set $P_{2n}$ with endpoint $(2n,0)$ along the axis $x=n$.

Note that the permutation $\pi_{\alpha_{2n}}$ leaves the $y$-coordinates of a given lattice path from $P_{2n}$ at all \emph{odd} abscissas $x=1,3,\ldots,2n-1$ invariant, and decreases the $y$-coordinates at all \emph{even} abscissas $x=2i$, $i=1,\ldots,n-1$, by $-2$ if and only if $\alpha_{2n}(i)=1$ and the steps $2i$ and $2i+1$ of the path are an upstep and a downstep, respectively. This mapping clearly leaves the $y$-coordinates at the abscissas $x=0$ and $x=2n$ invariant as well.

Observe that for every lattice path from $D_{2n}^{=0}(n)$ or from $D_{2n}^-(n)$, the $y$-coordinates at all odd abscissas $x=1,3,\ldots,2n-1$ are odd, and the $y$-coordinates at all even abscissas $x=0,2,\ldots,2n$ are even (in particular, the abscissa where a lattice path from the set $D_{2n}^-(n)$ touches the line $y=-1$ is odd). This property implies that for any pair $2i$ and $2i+1$, $1\leq i\leq n-1$, of an upstep and a downstep on such a path, the point $(2i,y')$ on the path satisfies $y'\geq 2$ ($y'$ must be even, and if it were 0 or less, then this path would have at least two points with a negative $y$-coordinate). Using these observations and the above-mentioned properties how the permutation $\pi_{\alpha_{2n}}$ affects the $y$-coordinates at the odd and even abscissas, it follows that $\pi_{\alpha_{2n}}$ indeed maps each of the sets $D_{2n}^{=0}(n)$ and $D_{2n}^-(n)$ onto itself. This proves the lemma.
\end{proof}

\subsection{An abstract recursion}

In this section we define an abstract recursion over sets of bitstrings and show that the solutions of this recursion correspond to certain sets of lattice paths (Lemma~\ref{lemma:FSL-D-isomorphic} below).

For all $n\geq 1$ and all $k=n,n+1,\ldots,2n-1$ we define sets of bitstrings $F_{2n}(k,k+1)\seq B_{2n}(k)$, $S_{2n}(k,k+1)\seq B_{2n}(k+1)$ and $L_{2n}(k,k+1)\seq B_{2n}(k)$ recursively as follows:

For $n=1$ we define
\begin{equation} \label{eq:ind-base-FSL}
  F_2(1,2) := \{(1,0)\} \enspace, \quad
  S_2(1,2) := \{(1,1)\} \enspace, \quad
  L_2(1,2) := \{(0,1)\} \enspace.
\end{equation}

For any $n\geq 1$ and all $k=n+2,\ldots,2n+1$ we define
\begin{subequations} \label{eq:ind-step1-FSL}
\begin{align}
  F_{2n+2}(k,k+1) &:= F_{2n}(k,k+1)\circ(0,0)\cup F_{2n}(k-1,k)\circ(1,0) \notag \\
                  &\qquad \cup F_{2n}(k-1,k)\circ(0,1)\cup F_{2n}(k-2,k-1)\circ(1,1) \enspace, \label{eq:ind-step1-F} \\
  S_{2n+2}(k,k+1) &:= S_{2n}(k,k+1)\circ(0,0)\cup S_{2n}(k-1,k)\circ(1,0) \notag \\
                  &\qquad \cup S_{2n}(k-1,k)\circ(0,1)\cup S_{2n}(k-2,k-1)\circ(1,1) \enspace, \label{eq:ind-step1-S} \\
  L_{2n+2}(k,k+1) &:= L_{2n}(k,k+1)\circ(0,0)\cup L_{2n}(k-1,k)\circ(1,0) \notag \\
                  &\qquad \cup L_{2n}(k-1,k)\circ(0,1)\cup L_{2n}(k-2,k-1)\circ(1,1) \enspace, \label{eq:ind-step1-L}
\end{align}
\end{subequations}
where we use the convention $F_{2n}(2n,2n+1):=\emptyset$, $S_{2n}(2n,2n+1):=\emptyset$, $L_{2n}(2n,2n+1):=\emptyset$ and $F_{2n}(2n+1,2n+2):=\emptyset$, $S_{2n}(2n+1,2n+2):=\emptyset$, $L_{2n}(2n+1,2n+2):=\emptyset$.

Furthermore, for any $n\geq 1$ we define
\begin{subequations} \label{eq:ind-step2-FSL}
\begin{align}
  F_{2n+2}(n+1,n+2) &:= F_{2n}(n+1,n+2)\circ(0,0) \notag \\
                    &\qquad \cup F_{2n}(n,n+1)\circ(1,0)\cup S_{2n}(n,n+1)\circ(0,0) \enspace, \label{eq:ind-step2-F} \\
  S_{2n+2}(n+1,n+2) &:= S_{2n}(n+1,n+2)\circ(0,0) \notag \\
                    &\qquad \cup S_{2n}(n,n+1)\circ(1,0)\cup S_{2n}(n,n+1)\circ(0,1) \enspace, \label{eq:ind-step2-S} \\
  L_{2n+2}(n+1,n+2) &:= L_{2n}(n+1,n+2)\circ(0,0) \notag \\
                    &\qquad \cup L_{2n}(n,n+1)\circ(1,0)\cup F_{2n}(n,n+1)\circ(0,1) \enspace, \label{eq:ind-step2-L}
\end{align}
\end{subequations}
where in the case $n=1$ we use the convention $F_2(2,3):=\emptyset$, $S_2(2,3):=\emptyset$, $L_2(2,3):=\emptyset$.

\begin{lemma} \label{lemma:FSL-D-isomorphic}
For any $n\geq 1$ and all $k=n,n+1,\ldots,2n-1$ we have
\begin{subequations}
\begin{align}
  \varphi(F_{2n}(k,k+1)) &= D_{2n}^{=0}(k) \enspace, \label{eq:F-D-isomorphic} \\
  \varphi(S_{2n}(k,k+1)) &= D_{2n}^{>0}(k+1) \enspace, \label{eq:S-D-isomorphic} \\
  \varphi(L_{2n}(k,k+1)) &= D_{2n}^-(k) \enspace, \label{eq:L-D-isomorphic}
\end{align}
\end{subequations}
where the sets $F_{2n}(k,k+1)$, $S_{2n}(k,k+1)$ and $L_{2n}(k,k+1)$ are defined in \eqref{eq:ind-base-FSL}, \eqref{eq:ind-step1-FSL} and \eqref{eq:ind-step2-FSL}.
\end{lemma}

Note that all unions in \eqref{eq:ind-step1-FSL} and \eqref{eq:ind-step2-FSL} are disjoint: This is obvious for the definitions in \eqref{eq:ind-step1-FSL}, \eqref{eq:ind-step2-S} and \eqref{eq:ind-step2-L}, as the two-bit strings attached to the sets participating in each of the unions are distinct. For the definition in \eqref{eq:ind-step2-F} this follows from Lemma~\ref{lemma:FSL-D-isomorphic}, as by \eqref{eq:F-D-isomorphic} and \eqref{eq:S-D-isomorphic} the sets $F_{2n}(n+1,n+2)$ and $S_{2n}(n,n+1)$ participating in the union correspond to the sets $D_{2n}^{=0}(n+1)$ and $D_{2n}^{>0}(n+1)$ and are therefore disjoint.

\begin{proof}
We argue by induction over $n$. The fact that all three claimed relations hold for $n=1$ follows immediately from \eqref{eq:ind-base-FSL}.
For the induction step let $n\geq 1$ be fixed. We prove that the statement of the lemma holds for $n+1$ assuming that it holds for $n$. We distinguish the cases $k=n+2,\ldots,2n+1$ and $k=n+1$.

For $k=n+2,\ldots,2n+1$ we have
\begin{equation*}
\begin{split}
  \varphi(F_{2n+2}(k,k+1)) &\eqBy{eq:ind-step1-F}
    \varphi(F_{2n}(k,k+1)\circ(0,0)) \cup
    \varphi(F_{2n}(k-1,k)\circ(1,0)) \\
    &\qquad \cup
    \varphi(F_{2n}(k-1,k)\circ(0,1)) \cup
    \varphi(F_{2n}(k-2,k-1)\circ(1,1)) \\
  &\eqBy{eq:F-D-isomorphic}
    D_{2n}^{=0}(k)\circ(\downstep,\downstep) \cup
    D_{2n}^{=0}(k-1)\circ(\upstep,\downstep) \\
    &\qquad \cup
    D_{2n}^{=0}(k-1)\circ(\downstep,\upstep) \cup
    D_{2n}^{=0}(k-2)\circ(\upstep,\upstep)
  \eqBy{eq:D2np2-*-u-partition} D_{2n+2}^{=0}(k) \enspace,
\end{split}
\end{equation*}
where we used the induction hypothesis in the second step.
The proof that also the last two relations stated in the lemma hold in this case goes along very similar lines, using \eqref{eq:ind-step1-S}, \eqref{eq:S-D-isomorphic} and \eqref{eq:D2np2-g0-u-partition} in the first, second and third step, or \eqref{eq:ind-step1-L}, \eqref{eq:L-D-isomorphic} and \eqref{eq:D2np2-*-u-partition}, respectively. We omit the details here.

For the case $k=n+1$ we obtain
\begin{equation*}
\begin{split}
  \varphi(& F_{2n+2}(n+1,n+2)) \\
  &\eqBy{eq:ind-step2-F} 
    \varphi(F_{2n}(n+1,n+2)\circ(0,0)) \cup
    \varphi(F_{2n}(n,n+1)\circ(1,0)) \cup
    \varphi(S_{2n}(n,n+1)\circ(0,0)) \\
  &\eqByM{\eqref{eq:F-D-isomorphic},\eqref{eq:S-D-isomorphic}}
    \big(D_{2n}^{=0}(n+1) \cup D_{2n}^{>0}(n+1)\big)\circ(\downstep,\downstep) \cup
    D_{2n}^{=0}(n)\circ(\upstep,\downstep)
  \eqBy{eq:D2np2-eq0-m-partition} D_{2n+2}^{=0}(n+1) \enspace,
\end{split}
\end{equation*}
where we used the induction hypothesis in the second step.
In a similar fashion we obtain
\begin{equation*}
\begin{split}
  \varphi(& S_{2n+2}(n+1,n+2)) \\
  &\eqBy{eq:ind-step2-S}
    \varphi(S_{2n}(n+1,n+2)\circ(0,0)) \cup
    \varphi(S_{2n}(n,n+1)\circ(1,0)) \cup
    \varphi(S_{2n}(n,n+1)\circ(0,1)) \\
  &\eqBy{eq:S-D-isomorphic}
    D_{2n}^{>0}(n+2)\circ(\downstep,\downstep) \cup
    D_{2n}^{>0}(n+1) \circ(\upstep,\downstep) \cup
    D_{2n}^{>0}(n+1)\circ(\downstep,\upstep)
  \eqBy{eq:D2np2-g0-m-partition} D_{2n+2}^{>0}(n+2)
\end{split}
\end{equation*}
and
\begin{equation*}
\begin{split}
  \varphi(& L_{2n+2}(n+1,n+2)) \\
  &\eqBy{eq:ind-step2-L}
    \varphi(L_{2n}(n+1,n+2)\circ(0,0)) \cup
    \varphi(L_{2n}(n,n+1)\circ(1,0)) \cup
    \varphi(F_{2n}(n,n+1)\circ(0,1)) \\
  &\eqByM{\eqref{eq:F-D-isomorphic},\eqref{eq:L-D-isomorphic}}
    D_{2n}^-(n+1)\circ(\downstep,\downstep) \cup
    D_{2n}^-(n)\circ(\upstep,\downstep) \cup
    D_{2n}^{=0}(n)\circ(\downstep,\upstep)
  \eqBy{eq:D2np2-m-m-partition} D_{2n+2}^-(n+1) \enspace.
\end{split}
\end{equation*}
This completes the proof.
\end{proof}

\subsection{Proof of Lemma~\texorpdfstring{\ref{lemma:FL-invariant}}{3}}

We are now ready to complete the proof of Lemma~\ref{lemma:FL-invariant}, thus showing that the construction of 2-factors described in Section~\ref{sec:construction} works as claimed.

We introduce the abbreviations
\begin{subequations} \label{eq:FSL-abbreviation}
\begin{align}
  F_{2n}(k,k+1) &:= F(\cP_{2n}(k,k+1)) \enspace, \\
  S_{2n}(k,k+1) &:= S(\cP_{2n}(k,k+1)) \enspace, \\
  L_{2n}(k,k+1) &:= L(\cP_{2n}(k,k+1))
\end{align}
\end{subequations}
for the sets of first, second and last vertices of the oriented paths in the sets $\cP_{2n}(k,k+1)$ arising in our construction.

\begin{proof}[Proof of Lemma~\ref{lemma:FL-invariant}]
Observe that the sets $F_{2n}(k,k+1)$, $S_{2n}(k,k+1)$ and $L_{2n}(k,k+1)$ defined in \eqref{eq:FSL-abbreviation} satisfy exactly the recursive relations in \eqref{eq:ind-base-FSL}, \eqref{eq:ind-step1-FSL} and \eqref{eq:ind-step2-FSL}: This can be seen by comparing \eqref{eq:ind-base-P} with \eqref{eq:ind-base-FSL}, \eqref{eq:ind-step1-P} with \eqref{eq:ind-step1-FSL} and finally \eqref{eq:ind-step2-P} with \eqref{eq:ind-step2-FSL}, in the last step also using \eqref{eq:new-paths-FSL}.
We may thus apply Lemma~\ref{lemma:FSL-D-isomorphic}, and using the relations \eqref{eq:F-D-isomorphic} and \eqref{eq:L-D-isomorphic} for $k=n$, we obtain that proving \eqref{eq:FL-invariant} is equivalent to proving that the mapping $f_{\alpha_{2n}}$ defined in \eqref{eq:f-alpha} satisfies $f_{\alpha_{2n}}(D_{2n}^{=0}(n))=D_{2n}^{=0}(n)$ and $f_{\alpha_{2n}}(D_{2n}^-(n))=D_{2n}^-(n)$, which is exactly the assertion of Lemma~\ref{lemma:dyck-paths-invariant}.
\end{proof}

\begin{remark} \label{remark:use-lemma-FSL-D}
Using the abbreviations defined in \eqref{eq:FSL-abbreviation} we may and will from now on use Lemma~\ref{lemma:FSL-D-isomorphic} as a statement about the sets of first, second and last vertices of the oriented paths in the sets $\cP_{2n}(k,k+1)$ arising in our construction (rather than as a statement about abstractly defined sets of bitstrings).
\end{remark}

\section{Structure of the 2-factors \texorpdfstring{$\cC_{2n+1}^0$}{C2n+10} and \texorpdfstring{$\cC_{2n+1}^1$}{C2n+11}}
\label{sec:structure-2-factor}

This section constitutes the first building block of our analysis of the graph $\cG(\cC_{2n+1}^1,\cX_{2n}^1(n,n+1))$ that is required for proving Proposition~\ref{prop:main1} and \ref{prop:main2}. Specifically, we analyze in detail the cycle structure of the 2-factors $\cC_{2n+1}^0$ and $\cC_{2n+1}^1$ defined in \eqref{eq:2-factor-0} and \eqref{eq:2-factor-1}, and show that the cycles of each of those 2-factors are in one-to-one correspondence with all plane trees with $n$ edges (these trees will be defined shortly). It follows that the nodes of $\cG(\cC_{2n+1}^1,\cX_{2n}^1(n,n+1))$ correspond to plane trees (recall that the nodes of $\cG(\cC_{2n+1}^1,\cX_{2n}^1(n,n+1))$ are the cycles of $\cC_{2n+1}^1$, see the bottom of Figure~\ref{fig:idea}).
This correspondence between the cycles of $\cC_{2n+1}^0$ and $\cC_{2n+1}^1$ and plane trees is stated in Lemma~\ref{lemma:2f-C0} and \ref{lemma:2f-C1} below, respectively.
Recall from Remark~\ref{remark:C0C1} that we are ultimately only interested in the 2-factor $\cC_{2n+1}^1$, but that the simpler 2-factor $\cC_{2n+1}^0$ is the key to understanding $\cC_{2n+1}^1$.

As we have seen in Section~\ref{sec:correctness}, to prove that our construction of 2-factors works as claimed, we only needed to consider the sets of first, second and last vertices of the paths in the sets $\cP_{2n}(k,k+1)$ (and could neglect all other vertices on these paths). So far we did not use any information about which of those vertices actually lie on the same paths, e.g., which first vertices are connected to which last vertices. This information will become crucial in the following.

\subsection{Subpaths of lattice paths}

We begin by extending some of the notation introduced in Section~\ref{sec:lattice-paths}.

\textit{The subpaths $p[x,x']$, $\ell(p)$, $r(p)$.}
For any lattice path $p\in P_n$ and any two abscissas $0\leq x\leq x'\leq n$ we define $p[x,x']$ as the subpath of $p$ between (and including) the abscissas $x$ and $x'$.

For any $n\geq 1$ and any lattice path $p$ in one of the sets $D_{2n}^{=0}(k)$, $D_{2n}^{>0}(k+1)$ and $D_{2n}^-(k)$, $k=n,n+1,\ldots,2n-1$, we define disjoint subpaths $\ell(p)$ and $r(p)$ of $p$ that cover all but two steps of $p$ as follows (see Figure~\ref{fig:subpaths}):
\begin{subequations} \label{eq:def-ell-r}
\begin{itemize}
\item
If $p\in D_{2n}^{=0}(k)$ we define
\begin{equation} \label{eq:def-ell-r-F}
  \ell(p):=p[1,x-1] \quad \text{and} \quad r(p):=p[x,2n] \enspace,
\end{equation}
where $x$ is the smallest strictly positive abscissa where $p$ touches the $y$-axis.

\item
If $p\in D_{2n}^{>0}(k+1)$ we define
\begin{equation} \label{eq:def-ell-r-S}
  \ell(p):=p[1,x] \quad \text{and} \quad r(p):=p[x+1,2n] \enspace,
\end{equation}
where $x$ is the largest abscissa where $p$ touches the line $y=1$ (the first time $p$ touches it is at $(1,1)$).

\item
If $p\in D_{2n}^-(k)$ we define
\begin{equation} \label{eq:def-ell-r-L}
  \ell(p):=p[0,x-1] \quad \text{and} \quad r(p):=p[x+1,2n] \enspace,
\end{equation}
where $x$ is the abscissa where $p$ touches the line $y=-1$.
\end{itemize}
\end{subequations}

With those definitions, depending on whether $p$ is contained in $D_{2n}^{=0}(k)$, $D_{2n}^{>0}(k+1)$ or $D_{2n}^-(k)$, we have
\begin{subequations}
\begin{align}
  p &= (\upstep)\circ\ell(p)\circ(\downstep)\circ r(p) \enspace, \label{eq:ell-r-F-partition} \\
  p &= (\upstep)\circ\ell(p)\circ(\upstep)\circ r(p) \enspace, \label{eq:ell-r-S-partition} \\
  p &= \ell(p)\circ(\downstep,\upstep)\circ r(p) \enspace, \notag 
\end{align}
\end{subequations}
respectively. In all cases, the subpath $\ell(p)$ starts and ends at the same ordinate and never moves below this ordinate in between. Furthermore, the ordinate of the endpoint of the subpath $r(p)$ is by $2(k-n)$ higher than the ordinate of its starting point and also this subpath never moves below the ordinate of its starting point.

\begin{figure}
\centering
\PSforPDF{
 \psfrag{z}{0}
 \psfrag{x}{$x$}
 \psfrag{2n}{$2n$}
 \psfrag{2k}[c][c][1][90]{\small $2(k-n)$}
 \psfrag{2kp}[c][c][1][90]{\small $2(k+1-n)$}
 \psfrag{Deq0}{$p\in D_{2n}^{=0}(k)$}
 \psfrag{Dgt0}{$p\in D_{2n}^{>0}(k+1)$}
 \psfrag{Dm}{$p\in D_{2n}^-(k)$}
 \psfrag{ellp}{$\ell(p)$}
 \psfrag{rp}{$r(p)$}
 \includegraphics{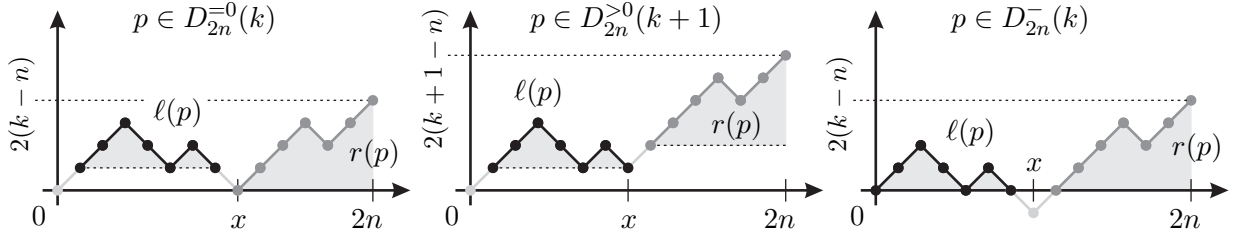}
}
\caption{Illustration of the definitions \eqref{eq:def-ell-r}.}
\label{fig:subpaths}
\end{figure}

\subsection{Matching between first and last vertices}

The next lemma relates the lattice paths $\varphi(F(P))$, $\varphi(S(P))$ and $\varphi(L(P))$ corresponding to the first, second and last vertex on each of the paths $P\in\cP_{2n}(k,k+1)$ arising in the construction of Section~\ref{sec:construction} for the all-zero parameter sequence. In the following we will repeatedly use that by Lemma~\ref{lemma:FSL-D-isomorphic} those lattice paths satisfy $\varphi(F(P))\in D_{2n}^{=0}(k)$, $\varphi(S(P))\in D_{2n}^{>0}(k+1)$ and $\varphi(L(P))\in D_{2n}^-(k)$ (recall Remark~\ref{remark:use-lemma-FSL-D}).

\begin{lemma}
\label{lemma:FL-matching-0}
For any $n\geq 1$, the sets of paths $\cP_{2n}(k,k+1)$, $k=n,n+1,\ldots,2n-1$, defined in Section~\ref{sec:construction} for the parameter sequence $(\alpha_{2i})_{1\leq i\leq n-1}$ with $\alpha_{2i}=(0,0,\ldots,0)\in\{0,1\}^{i-1}$ for all $i=1,\ldots,n-1$ have the following properties: For any path $P\in\cP_{2n}(k,k+1)$, defining $p_F:=\varphi(F(P))\in D_{2n}^{=0}(k)$, $p_S:=\varphi(S(P))\in D_{2n}^{>0}(k+1)$ and $p_L:=\varphi(L(P))\in D_{2n}^-(k)$, we have
\begin{align}
  (\ell(p_F),r(p_F)) &= (\ell(p_S),r(p_S)) \enspace, \label{eq:FS-relation-0} \\
  (\ell(p_F),r(p_F)) &= (\ell(p_L),r(p_L)) \enspace, \label{eq:FL-relation-0}
\end{align}
where $\ell(p_F)$ and $r(p_F)$ are defined in \eqref{eq:def-ell-r-F}, $\ell(p_S)$ and $r(p_S)$ in \eqref{eq:def-ell-r-S}, and $\ell(p_L)$ and $r(p_L)$ in \eqref{eq:def-ell-r-L}.
\end{lemma}

With the equalities in \eqref{eq:FS-relation-0} and \eqref{eq:FL-relation-0} we mean that the step sequences of the lattice paths $\ell(p_F)$, $\ell(p_S)$ and $\ell(p_L)$, and the step sequences of the lattice paths $r(p_F)$, $r(p_S)$ and $r(p_L)$ are the same. The absolute coordinates of those subpaths of $p_F$, $p_S$ and $p_L$ might be different.

The proof of Lemma~\ref{lemma:FL-matching-0} shows that \eqref{eq:FS-relation-0} holds for \emph{any} choice of the parameter sequence, not just for the all-zero parameter sequence (recall the remarks from Section~\ref{sec:dependence-alpha}).

\begin{proof}
We argue by induction over $n$. By the definition in \eqref{eq:ind-base-P}, for $n=1$ the sets of paths $\cP_{2n}(k,k+1)$ consist only of a single set $\cP_2(1,2)$, which contains only a single path $P:=((1,0),(1,1),(0,1))$ ($P$ has two edges). We clearly have $p_F:=\varphi(F(P))=(\upstep,\downstep)\in D_2^{=0}(1)$, $p_S:=\varphi(S(P))=(\upstep,\upstep)\in D_2^{>0}(2)$ and $p_L:=\varphi(L(P))=(\downstep,\upstep)\in D_2^-(1)$, and by the definitions in \eqref{eq:def-ell-r} the subpaths $\ell(p_F)$, $r(p_F)$, $\ell(p_S)$, $r(p_S)$, $\ell(p_L)$ and $r(p_L)$ of those lattice paths all consist only of a single point (and zero steps), showing that both claims of the lemma hold. This settles the induction basis.

For the induction step $n\rightarrow n+1$ let $n\geq 1$ be fixed. We consider a fixed path $P^+$ from one of the sets $\cP_{2n+2}(k,k+1)$, $k=n+1,n+2,\ldots,2n+1$, and define the lattice paths $p_F^+:=\varphi(F(P^+))\in D_{2n+2}^{=0}(k)$, $p_S^+:=\varphi(S(P^+))\in D_{2n+2}^{>0}(k+1)$ and $p_L^+:=\varphi(L(P^+))\in D_{2n+2}^-(k)$.
By the definitions in \eqref{eq:ind-step1-P} and \eqref{eq:ind-step2-P}, $P^+$ is either contained in the set
\begin{equation} \label{eq:upper-paths}
  \cP_{2n}(n,n+1)\circ \{(1,0),(1,1)\} \cup \bigcup_{k'=n+1}^{2n-1} \cP_{2n}(k',k'+1)\circ\{(0,0),(1,0),(0,1),(1,1)\} 
\end{equation}
or in the set $\cP_{2n+2}^*$ defined in \eqref{eq:new-paths} (in the latter case we have $k=n+1$).

We first consider the case that $P^+$ is contained in \eqref{eq:upper-paths}, i.e., $P^+$ is obtained from some path $P\in\cP_{2n}(k',k'+1)$, $n\leq k'\leq 2n-1$, by extending each vertex label of $P$ by two bits $x_1,x_2\in\{0,1\}$.
We know by induction that the lattice paths $p_F:=\varphi(F(P))\in D_{2n}^{=0}(k')$, $p_S:=\varphi(S(P))\in D_{2n}^{>0}(k'+1)$ and $p_L:=\varphi(L(P))\in D_{2n}^-(k')$ satisfy the relations
\begin{align}
  (\ell(p_F),r(p_F)) &= (\ell(p_S),r(p_S)) \enspace, \label{eq:ind-FS-relation} \\
  (\ell(p_F),r(p_F)) &= (\ell(p_L),r(p_L)) \enspace. \label{eq:ind-FL-relation}
\end{align}
Moreover, we clearly have
\begin{subequations}
\begin{align}
  p_F^+ &= p_F\circ\varphi((x_1,x_2)) \enspace, \label{eq:pFp-pF} \\
  p_S^+ &= p_S\circ\varphi((x_1,x_2)) \enspace, \label{eq:pSp-pS} \\
  p_L^+ &= p_L\circ\varphi((x_1,x_2)) \enspace. \label{eq:pLp-pL}
\end{align}
\end{subequations}
Using \eqref{eq:pFp-pF} and the fact that $p_F$ is contained in the set $D_{2n}^{=0}(k')$, the definition in \eqref{eq:def-ell-r-F} yields
\begin{equation} \label{eq:ell-r-pFp-pF}
  (\ell(p_F^+),r(p_F^+)) = \big(\ell(p_F),r(p_F)\circ \varphi((x_1,x_2))\big) \enspace.
\end{equation}
Similarly, using \eqref{eq:pLp-pL} and the fact that $p_L$ is contained in the set $D_{2n}^-(k')$, the definition in \eqref{eq:def-ell-r-L} yields
\begin{equation} \label{eq:ell-r-pLp-pL}
  (\ell(p_L^+),r(p_L^+)) = \big(\ell(p_L),r(p_L)\circ \varphi((x_1,x_2))\big) \enspace.
\end{equation}
Using that $p_S\in D_{2n}^{>0}(k'+1)$, it follows that if $k'=n$, then the $y$-coordinate of the last point of $p_S$ is 2, whereas if $k'\geq n+1$, then the $y$-coordinate of the last point of $p_S$ is at least 4. Combined with \eqref{eq:upper-paths} and \eqref{eq:pSp-pS} it follows that the last two steps of $p_S^+$ do not move below the line $y=2$. By the definition in \eqref{eq:def-ell-r-S} and by \eqref{eq:pSp-pS} we therefore have
\begin{equation} \label{eq:ell-r-pSp-pS}
  (\ell(p_S^+),r(p_S^+)) = \big(\ell(p_S),r(p_S)\circ \varphi((x_1,x_2))\big) \enspace.
\end{equation}
Combining \eqref{eq:ind-FS-relation}, \eqref{eq:ell-r-pFp-pF} and \eqref{eq:ell-r-pSp-pS} yields $(\ell(p_F^+),r(p_F^+))=(\ell(p_S^+),r(p_S^+))$ and thus proves \eqref{eq:FS-relation-0}. Combining \eqref{eq:ind-FL-relation}, \eqref{eq:ell-r-pFp-pF} and \eqref{eq:ell-r-pLp-pL} yields $(\ell(p_F^+),r(p_F^+))=(\ell(p_L^+),r(p_L^+))$ and thus proves \eqref{eq:FL-relation-0}.
(Note that so far we did not use that $\alpha_{2n}=(0,0,\ldots,0)\in\{0,1\}^{n-1}$, but only that all other elements of the parameter sequence used in previous construction steps are zero vectors as well, so that the induction hypothesis can be used.)

We now consider the case that $P^+$ is contained in the set $\cP_{2n+2}^*$. For the reader's convenience, Figure~\ref{fig:new-paths} illustrates the notations used in this part of the proof. By the definition in \eqref{eq:new-paths}, there are two paths $P,P'\in\cP_{2n}(n,n+1)$ with
\begin{subequations} \label{eq:FS-Pp-P}
\begin{align}
  F(P^+) &= S(P)\circ(0,0) \enspace, \\
  S(P^+) &= S(P)\circ(0,1) \enspace, \\
  L(P^+) &= F(P')\circ(0,1)
\end{align}
\end{subequations}
(see~\eqref{eq:new-paths-FSL}).
Defining $p_S:=\varphi(S(P))\in D_{2n}^{>0}(n+1)$ and $p_F':=\varphi(F(P'))\in D_{2n}^{=0}(n)$ we obtain from \eqref{eq:FS-Pp-P} that
\begin{subequations}
\begin{align}
  p_F^+ &= p_S\circ (\downstep,\downstep) \enspace, \label{eq:pFp-pS} \\
  p_S^+ &= p_S\circ (\downstep,\upstep) \enspace, \label{eq:pSp-pS2} \\
  p_L^+ &= p_F'\circ (\downstep,\upstep) \enspace. \label{eq:pLp-pFpr}
\end{align}
\end{subequations}
The lattice path $p_S\in D_{2n}^{>0}(n+1)$ clearly ends at $(2n,2)$.
From \eqref{eq:pFp-pS} it follows that $p_F^+\in D_{2n+2}^{=0}(n+1)$ and that the smallest strictly positive abscissa where this lattice path touches the $y$-axis is $2n+2$ (see Figure~\ref{fig:new-paths}). By the definition in \eqref{eq:def-ell-r-F} and by \eqref{eq:pFp-pS} we therefore have
\begin{equation} \label{eq:ell-r-pF+}
  (\ell(p_F^+),r(p_F^+))=\big(p_S[1,2n]\circ (\downstep),()\big)
\end{equation}
($r(p_F^+)$ consists only of a single point).
From \eqref{eq:pSp-pS2} it follows that $p_S^+\in D_{2n+2}^{>0}(n+2)$ and that the largest abscissa where this lattice path touches the line $y=1$ is $2n+1$. By the definition in \eqref{eq:def-ell-r-S} and by \eqref{eq:pSp-pS2} we therefore have
\begin{equation*}
  (\ell(p_S^+),r(p_S^+))=\big(p_S[1,2n]\circ (\downstep),()\big) \enspace,
\end{equation*}
which together with \eqref{eq:ell-r-pF+} shows that \eqref{eq:FS-relation-0} also holds in this case. (Note that the inductive proof of \eqref{eq:FS-relation-0} goes through for any choice of $\alpha_{2n}\in\{0,1\}^{n-1}$.)

It remains to prove \eqref{eq:FL-relation-0} in this case (this part of the proof uses that $\alpha_{2n}=(0,0,\ldots,0)\in\{0,1\}^{n-1}$).
From \eqref{eq:pLp-pFpr} it follows that $p_L^+\in D_{2n+2}^-(n+1)$ and that the only abscissa where this lattice path touches the line $y=-1$ is $2n+1$. By the definition in \eqref{eq:def-ell-r-L} and by \eqref{eq:pLp-pFpr} we therefore have
\begin{equation} \label{eq:ell-r-pL+}
  (\ell(p_L^+),r(p_L^+))=\big(p_F',()\big) \enspace.
\end{equation}
By \eqref{eq:ell-r-pF+} and \eqref{eq:ell-r-pL+}, to complete the proof of the lemma we need to show that $p_F'=p_S[1,2n]\circ(\downstep)$.

By induction we know that the lattice paths $p_F:=\varphi(F(P))\in D_{2n}^{=0}(n)$ and $p_L:=\varphi(L(P))\in D_{2n}^-(n)$ satisfy the relations
\begin{align}
  (\ell(p_F),r(p_F)) &= (\ell(p_S),r(p_S)) \enspace, \label{eq:ell-r-pF-pS} \\
  (\ell(p_F),r(p_F)) &= (\ell(p_L),r(p_L)) \enspace. \label{eq:ell-r-pF-pL}
\end{align}
By Lemma~\ref{lemma:FL-invariant} there is a path $\Phat\in\cP_{2n}(n,n+1)$ satisfying $f_{\alpha_{2n}}(L(\Phat))=L(P)$ and $f_{\alpha_{2n}}(F(\Phat))=F(P')$. By the definition in \eqref{eq:f-alpha}, for $\alpha_{2n}=(0,0,\ldots,0)\in\{0,1\}^{n-1}$ we have $f_{\alpha_{2n}}=\ol{\rev}$, so these relations imply that the corresponding lattice paths $\phat_L:=\varphi(L(\Phat))\in D_{2n}^-(n)$ and $\phat_F:=\varphi(F(\Phat))\in D_{2n}^{=0}(n)$ satisfy
\begin{align}
  \big(\ol{\rev}(r(\phat_L)),\ol{\rev}(\ell(\phat_L))\big) &= (\ell(p_L),r(p_L)) \enspace, \label{eq:ell-r-olpL-pL-flipp} \\
  \ol{\rev}(\phat_F) &= p_F' \label{eq:rev-phF}
\end{align}
(recall that $\ol{\rev}$ simply mirrors a lattice path with $2n$ steps that ends at $(2n,0)$ along the axis $x=n$).
By induction we also know that
\begin{equation}
\label{eq:ell-r-phF-phL}
  (\ell(\phat_F),r(\phat_F)) = (\ell(\phat_L),r(\phat_L)) \enspace.
\end{equation}
Combining our previous observations we obtain
\begin{equation} \label{eq:pF'}
\begin{split}
  p_F' &\eqByM{\eqref{eq:rev-phF}, \eqref{eq:ell-r-F-partition}} \ol{\rev}\big((\upstep)\circ \ell(\phat_F) \circ(\downstep)\circ r(\phat_F)\big) \\
    &\eqBy{eq:ell-r-phF-phL} \ol{\rev}(r(\phat_L))\circ (\upstep)\circ \ol{\rev}(\ell(\phat_L))\circ (\downstep) \\
    &\eqByM{\eqref{eq:ell-r-pF-pL},\eqref{eq:ell-r-olpL-pL-flipp}} \ell(p_F) \circ (\upstep) \circ r(p_F) \circ (\downstep) \\
    &\eqByM{\eqref{eq:ell-r-pF-pS},\eqref{eq:ell-r-S-partition}} p_S[1,2n] \circ (\downstep) \enspace,
\end{split}
\end{equation}
completing the proof.
\end{proof}

\begin{figure}
\centering
\PSforPDF{
 \psfrag{bnp2}{$B_{2n+2}(n+2)$}
 \psfrag{bnp1}{$B_{2n+2}(n+1)$}
 \psfrag{q2n4}{$Q_{2n+1}\circ(1)$}
 \psfrag{q2n1}{$Q_{2n}\circ(0,0)$}
 \psfrag{q2n2}{$Q_{2n}\circ(0,1)$}
 \psfrag{q2n3}{$Q_{2n}\circ(1,1)$}
 \psfrag{m2}{\scriptsize $M_{2n+1}^{FL}\circ(1)$} 
 \psfrag{m1}{\scriptsize $M_{2n+2}^{S}$}
 \psfrag{pp}[c][c][1][13]{$P^+\in\cP_{2n+2}^*$}
 \psfrag{pm0}{\scriptsize $P\circ(0,0)$}
 \psfrag{pm1}{\scriptsize $P\circ(0,1)$}
 \psfrag{pb}{\scriptsize $\Phat\circ(0,1)$}
 \psfrag{ppr}{\scriptsize $P'\circ(0,1)$}
 \psfrag{fpp}{\scriptsize $F(P^+)$}
 \psfrag{spp}{\scriptsize $S(P^+)$}
 \psfrag{lpp}{\scriptsize $L(P^+)$}
 \psfrag{fpb1}{\scriptsize $f_{\alpha_{2n}}(\Phat)\circ(0,1)$}
 \psfrag{fpb2}{\scriptsize $f_{\alpha_{2n}}(\Phat)\circ(1,1)$}
 \psfrag{p2n1}{\scriptsize $\cP_{2n}(n,n+1)\circ(0,0)$}
 \psfrag{p2n2}{\scriptsize $\cP_{2n}(n,n+1)\circ(0,1)$}
 \psfrag{z}{\scriptsize $0$}
 \psfrag{2n}{\scriptsize $2n$}
 \psfrag{ps}{\scriptsize $p_S=\varphi(S(P))$}
 \psfrag{pfp}{\scriptsize $p_F^+=\varphi(F(P^+))$}
 \psfrag{psp}{\scriptsize $p_S^+=\varphi(S(P^+))$}
 \psfrag{plp}{\scriptsize $p_L^+=\varphi(L(P^+))$}
 \psfrag{pfpr}{\scriptsize $p_F'=\varphi(F(P'))$}
 \psfrag{pphi}{\scriptsize $=\varphi(f_{\alpha_{2n}}(F(\Phat)))$}
 \psfrag{pa1}{\scriptsize $p_F=\varphi(F(P))$}
 \psfrag{pa2}{\scriptsize $p_L=\varphi(L(P))$}
 \psfrag{pa4}{\scriptsize $\phat_F=\varphi(F(\Phat))$}
 \psfrag{pa5}{\scriptsize $\phat_L=\varphi(L(\Phat))$}
 \psfrag{f1}{\scriptsize $p_L^+=\ol{\rev}(\phat_F)\circ(\downstep,\upstep)$}
 \psfrag{f2}{\scriptsize $p_L=\ol{\rev}(\phat_L)$}
 \psfrag{psi1}{\scriptsize $\psi(p_F)$}
 \psfrag{psi2}{\scriptsize $\psi(p_F')$}
 \psfrag{psiell}{\scriptsize $\psi(\ell(p_F))$}
 \psfrag{psir}{\scriptsize $\psi(r(p_F))$}
 \includegraphics{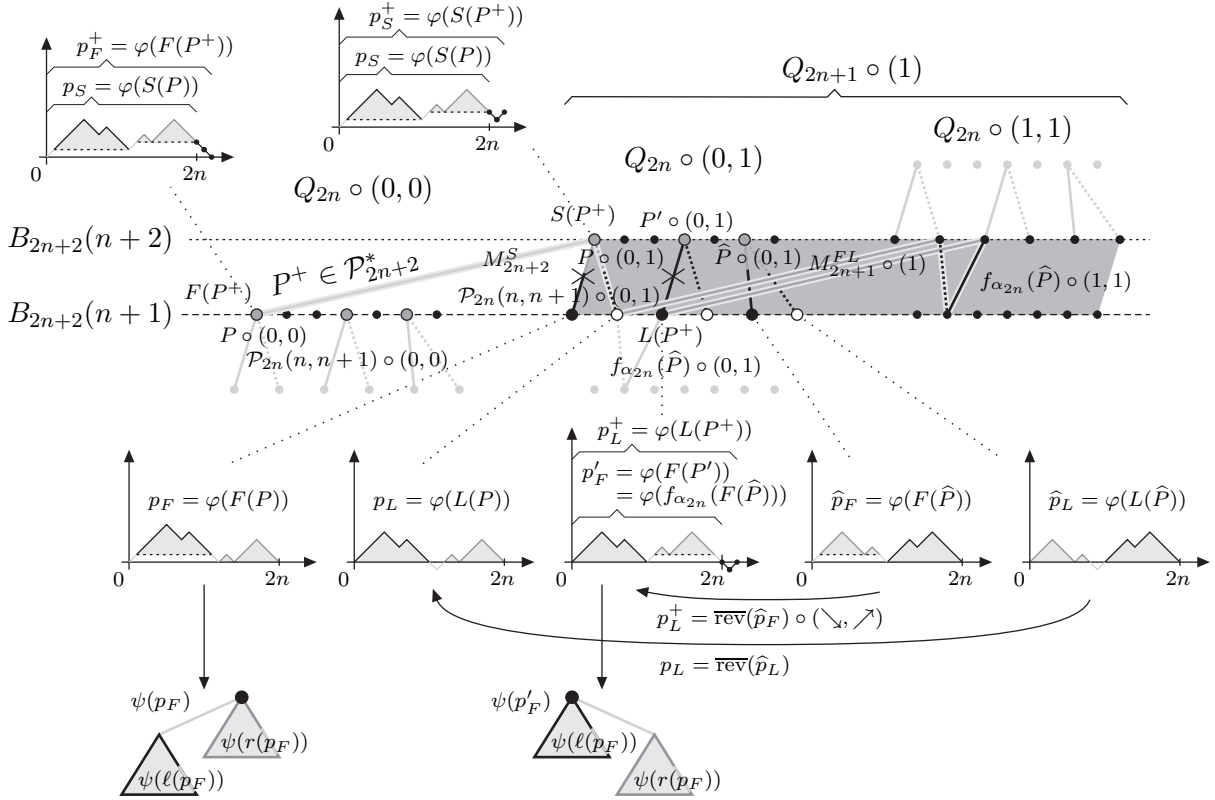}
}
\caption{Notations used in the proof of Lemma~\ref{lemma:FL-matching-0}. The figure illustrates the relations between various lattice paths corresponding to first, second and last vertices of oriented paths in the construction described in Section~\ref{sec:construction} in the induction step $n\rightarrow n+1$ ($Q_{2n}\rightarrow Q_{2n+2}$) when the all-zero parameter sequence is used.}
\label{fig:new-paths}
\end{figure}

\textit{The mapping $\pi_1$ on bitstrings and lattice paths.}
The next lemma is the analogue of Lemma~\ref{lemma:FL-matching-0} for the all-one parameter sequence.
To state it we introduce another definition.
We define
\begin{equation}
\label{eq:def-pi1}
  \pi_1:=\pi_{\alpha_{2n}} \quad \text{for} \enspace \alpha_{2n}:=(1,1,\ldots,1)\in\{0,1\}^{n-1} \enspace,
\end{equation}
where $\pi_{\alpha_{2n}}$ is defined before \eqref{eq:f-alpha}, and for the empty bitstring $()$ we define $\pi_1(()):=()$.
Via the bijection $\varphi$ between bitstrings and lattice paths, we can naturally extend $\pi_1$ to lattice paths (recall \eqref{eq:biject-map}).
For any lattice path with $2n$ steps, the mapping $\pi_1$ swaps any two adjacent steps at positions $2i$ and $2i+1$ for all $i=1,\ldots,n-1$.

\begin{lemma}
\label{lemma:FL-matching-1}
For any $n\geq 1$, the sets of paths $\cP_{2n}(k,k+1)$, $k=n,n+1,\ldots,2n-1$, defined in Section~\ref{sec:construction} for the parameter sequence $(\alpha_{2i})_{1\leq i\leq n-1}$ with $\alpha_{2i}=(1,1,\ldots,1)\in\{0,1\}^{i-1}$ for all $i=1,\ldots,n-1$ have the following properties: For any path $P\in\cP_{2n}(k,k+1)$, defining $p_F:=\varphi(F(P))\in D_{2n}^{=0}(k)$, $p_S:=\varphi(S(P))\in D_{2n}^{>0}(k+1)$ and $p_L:=\varphi(L(P))\in D_{2n}^-(k)$, we have
\begin{align}
  (\ell(p_F),r(p_F)) &= (\ell(p_S),r(p_S)) \enspace, \label{eq:FS-relation-1} \\
  (\ell(p_F),r(p_F)) &= (\pi_1(\ell(p_L)),r(p_L)) \enspace, \label{eq:FL-relation-1}
\end{align}
where $\ell(p_F)$ and $r(p_F)$ are defined in \eqref{eq:def-ell-r-F}, $\ell(p_S)$ and $r(p_S)$ in \eqref{eq:def-ell-r-S}, $\ell(p_L)$ and $r(p_L)$ in \eqref{eq:def-ell-r-L} and $\pi_1$ in \eqref{eq:def-pi1}.
\end{lemma}

\begin{proof}
The proof is completely analogous to the proof of Lemma~\ref{lemma:FL-matching-0}. We simply substitute \eqref{eq:ind-FL-relation}, \eqref{eq:ell-r-pF-pL}, \eqref{eq:ell-r-olpL-pL-flipp}, \eqref{eq:rev-phF}, \eqref{eq:ell-r-phF-phL} and \eqref{eq:pF'} by the following primed versions of these equations:
\begin{equation}
\labelp{eqp:ind-FL-relation}{eq:ind-FL-relation}
  (\ell(p_F),r(p_F))=(\pi_1(\ell(p_L)),r(p_L)) \enspace,
\end{equation}
\begin{equation}
\labelp{eqp:ell-r-pF-pL}{eq:ell-r-pF-pL}
  (\ell(p_F),r(p_F))=(\pi_1(\ell(p_L)),r(p_L)) \enspace,
\end{equation}
\begin{subequations}
\begin{align}
  \big(\ol{\rev}(\pi_1(r(\phat_L))),\ol{\rev}(\pi_1(\ell(\phat_L)))\big) &= (\ell(p_L),r(p_L)) \enspace, \labelp{eqp:ell-r-olpL-pL-flipp}{eq:ell-r-olpL-pL-flipp} \\
  \ol{\rev}(\pi_1(\phat_F)) &= p_F' \enspace, \labelp{eqp:rev-phF}{eq:rev-phF}
\end{align}
\end{subequations}
\begin{equation}
\labelp{eqp:ell-r-phF-phL}{eq:ell-r-phF-phL}
  (\ell(\phat_F),r(\phat_F))=(\pi_1(\ell(\phat_L)),r(\phat_L)) \enspace,
\end{equation}
\begin{equation}
\labelp{eqp:pF'}{eq:pF'}
\begin{split}
  \pi_1(p_F') \eqBy{eqp:rev-phF} \pi_1(\ol{\rev}(\pi_1(\phat_F)))=\ol{\rev}(\phat_F)
    &\eqBy{eq:ell-r-F-partition} \ol{\rev}\big((\upstep)\circ \ell(\phat_F) \circ(\downstep)\circ r(\phat_F)\big) \\
    &\eqBy{eqp:ell-r-phF-phL} \ol{\rev}(r(\phat_L))\circ (\upstep)\circ \ol{\rev}(\pi_1(\ell(\phat_L)))\circ (\downstep) \\
    &\eqByM{\eqref{eqp:ell-r-pF-pL},\eqref{eqp:ell-r-olpL-pL-flipp}} \ell(p_F) \circ (\upstep) \circ r(p_F) \circ (\downstep) \\
    &\eqByM{\eqref{eq:FS-relation-1},\eqref{eq:ell-r-S-partition}} p_S[1,2n] \circ (\downstep) \enspace.
\end{split}
\end{equation}
In the derivation of \eqref{eqp:ell-r-olpL-pL-flipp} and \eqref{eqp:rev-phF} we use that for $\alpha_{2n}=(1,1,\ldots,1)\in\{0,1\}^{n-1}$ we have $f_{\alpha_{2n}}=\ol{\rev}\bullet \pi_1$. Moreover, to derive \eqref{eqp:ell-r-olpL-pL-flipp} we use that the abscissas where the lattice paths $\phat_L$ and $\pi_1(\phat_L)$ touch the line $y=-1$ are the same.
To derive \eqref{eqp:pF'} we use the relations $\ol{\rev}\bullet\pi_1=\pi_1\bullet \ol{\rev}$ and $\pi_1\bullet\pi_1=\id$ in the second and the second to last step.
\end{proof}

\subsection{Ordered rooted trees and plane trees}
\label{sec:trees}

In this section we introduce ordered rooted trees and plane trees, two well-known sets of combinatorial objects. Moreover, we introduce natural transformations between those trees and a bijection between ordered rooted trees and the set of lattice paths $D_{2n}^{=0}(n)$ defined in Section~\ref{sec:lattice-paths}.
We shall see in the next section that the cycle structure of the 2-factors $\cC_{2n+1}^0$ and $\cC_{2n+1}^1$ can be characterized very precisely with the help of those trees.
For the reader's convenience, the notions introduced in the following are illustrated in Figure~\ref{fig:or-trees} and Figure~\ref{fig:p-trees}.

\textit{Ordered rooted trees.}
An \emph{ordered rooted tree} is a tree with a distinguished vertex $r$, called the \emph{root}, and in addition, for each vertex $v$ of $T$, a left-to-right ordering of neighbors of $v$. For the root $r$ this left-to-right ordering contains all neighbors of $r$, for every other vertex $v$ it contains all neighbors except the one on the path from $v$ to $r$. The neighbors of $v$ contained in the left-to-right ordering are referred to as \emph{children of $v$}, and the vertex not contained in this ordering is referred to as the \emph{parent of $v$}.

We think of an ordered rooted tree as a tree drawn in the plane with the root on top, and with downward edges leading from any vertex to its children, where the order in which the children appear from left to right is precisely the specified left-to-right ordering (see the right hand side of Figure~\ref{fig:or-trees}).

We denote by $\cT_n^*$ the set of all ordered rooted trees with $n$ edges.
It is well known that $|\cT_n^*|=C_n$, the $n$-th Catalan number, so we have $(|\cT_n^*|)_{n\geq 1}=(1,2,5,14,42,132,429,1430,4862,16796,\ldots)$ and $|\cT_n^*|=\Theta(4^n\cdot n^{-3/2})$ (see \cite{catalan-seq}).

\textit{Rotation $\trot()$ of ordered rooted trees.}
For any tree $T\in\cT_n^*$ we define the tree $T'=\trot(T)\in\cT_n^*$ as follows: Let $r$ denote the root of $T$ and $(u_1,\ldots,u_k)$ the left-to-right ordering of the children of $r$. The underlying (abstract) tree of $T$ and $T'$ is the same. The tree $T'$ is obtained from $T$ by making $u_1$ the new root vertex, and by making the subtree rooted at $r$ without $u_1$ and its descendants (this subtree only contains $r$ and the subtrees rooted at $u_2,\ldots,u_k$) a new rightmost child of $u_1$.

Intuitively, $T'=\trot(T)$ is obtained from $T$ by shifting the root vertex to the left. Note that we obtain the same tree again after $2n/k$ such rotations, where the value of $k$ depends on the symmetries of the tree (see the right hand side of Figure~\ref{fig:or-trees}).

\textit{Bijection $\psi$ between lattice paths and ordered rooted trees.}
For any lattice path $p\in D_{2n}^{=0}(n)$ we inductively define an ordered rooted tree $\psi(p)\in\cT_n^*$ as follows:
If $p=()$, then $\psi(p)$ is the tree consisting only of a single (root) vertex.
If $p$ has the form $p=(\upstep)\circ p'\circ (\downstep)$ with $p'\in D_{2n-2}^{=0}(n-1)$, then $\psi(p)$ is the tree whose root has exactly one child, and the subtree rooted at this child is the tree $\psi(p')$.
Otherwise $p$ has the form $p=p_1\circ p_2$ with $p_i\in D_{2n_i}^{=0}(n_i)$ and $n_i\geq 1$, $i\in\{1,2\}$. In this case $\psi(p)$ is obtained by gluing together the roots of $\psi(p_1)$ and $\psi(p_2)$ such that all subtrees of the new tree originating from $\psi(p_1)$ appear to the left of all subtrees originating from $\psi(p_2)$ (and the order within subtrees from $\psi(p_1)$ and within subtrees from $\psi(p_2)$ is preserved). Note that the definition of $\psi(p)$ in the last case is independent of the partition of $p$ into subpaths $p_1$ and $p_2$, even if there are several different such partitions.

Intuitively, $\psi(p)$ is obtained from $p$ by drawing a downward tree edge for every upstep $\upstep$ of $p$, and by moving up to the parent vertex (without drawing an edge) for every downstep $\downstep$ of $p$ (see the left hand side of Figure~\ref{fig:or-trees}).\footnote{This slightly counterintuitive correspondence arises as we defined $D_{2n}^{=0}(n)$ as the set of lattice paths that stay \emph{above} the line $y=0$, and as we draw rooted trees with the root on top with edges going \emph{downward}.}
It is easy to see that $\psi$ is a bijection between $D_{2n}^{=0}(n)$ and $\cT_n^*$.

Using the bijection $\psi$, all operations on lattice paths $D_{2n}^{=0}(n)$ as e.g.\ the permutation $\pi_1$ defined in \eqref{eq:def-pi1} can be extended naturally to ordered rooted trees $\cT_n^*$ (recall \eqref{eq:biject-map}).

\begin{figure}
\centering
\PSforPDF{
 \psfrag{p}{$p\in D_{10}^{=0}(5)$}
 \psfrag{t}{$T\in \cT_5^*$}
 \psfrag{psi}{$T=\psi(p)$}
 \psfrag{rot}{$\trot$}
 \psfrag{z}{0}
 \psfrag{ten}{10}
 \includegraphics{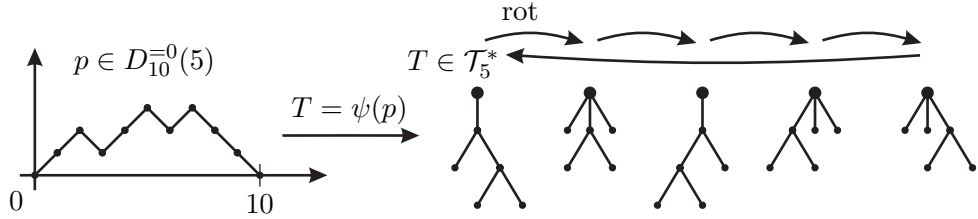}
}
\caption{Rotation of an ordered rooted tree $T\in\cT_5^*$ (right) and the corresponding lattice path $p\in D_{10}^{=0}(5)$ (left). The roots of the trees are drawn bold.}
\label{fig:or-trees}
\end{figure}

\textit{Plane trees.}
A \emph{plane tree} is a tree with a cyclic ordering of all neighbors of each vertex.

We think of a plane tree as a tree embedded in the plane such that for each vertex $v$ the order in which its neighbors are encountered when walking around $v$ in counterclockwise direction is precisely the specified cyclic ordering (see the middle of Figure~\ref{fig:p-trees}).

We denote by $\cT_n$ the set of all plane trees with $n$ edges.
The number of plane trees is $(|\cT_n|)_{n\geq 1}=(1,1,2,3,6,14,34,95,280,854,\ldots)$ and $|\cT_n|=\Theta(4^n\cdot n^{-5/2})$ (see \cite{plane-tree-seq} and references therein for closed formulas for this sequence).

\textit{Transformations $\plane()$ and $\troot()$ between ordered rooted trees and plane trees.}
For any ordered rooted tree $T^*\in\cT_n^*$, we define a plane tree $T=\plane(T^*)\in\cT_n$ as follows: The underlying (abstract) tree of $T^*$ and $T$ is the same. For the root $r$ of $T^*$, the cyclic ordering of neighbors of $r$ in $T$ is given by the left-to-right ordering of the children of $r$ in $T^*$. For any other vertex $v$ of $T^*$, if $(u_1,\ldots,u_k)$ is the left-to-right ordering of the children of $v$ in $T^*$ and $w$ is the parent of $v$, then we define $(w,u_1,\ldots,u_k)$ as the cyclic ordering of neighbors of $v$ in $T$.

For any plane tree $T\in\cT_n$ and any edge $(r,u)$ of $T$, we define an ordered rooted tree $T^*=\troot(T,(r,u))\in\cT_n^*$ as follows: The underlying (abstract) tree of $T$ and $T^*$ is the same. The vertex $r$ is the root of $T^*$, and if $(u_1,\ldots,u_k)$ with $u_1=u$ is the cyclic ordering of neighbors of $r$ in $T$, then $(u_1,\ldots,u_k)$ is the left-to-right ordering of the children of $r$ in $T^*$. For any other vertex $v$ of $T$, if $w$ is the vertex on the path from $v$ to $r$ and $(u_0,u_1,\ldots,u_k)$ with $u_0=w$ is the cyclic ordering of neighbors of $v$ in $T$, then $(u_1,\ldots,u_k)$ is the left-to-right ordering of the children of $v$ in $T^*$.

Informally speaking, $\plane(T^*)$ is obtained from $T^*$ by `forgetting' the root vertex, and $\troot(T,(r,u))$ is obtained from $T$ by `pulling out' the vertex $r$ as root such that $u$ is the leftmost child of the root (see Figure~\ref{fig:p-trees}).

\begin{figure}
\centering
\PSforPDF{
 \psfrag{t1s}{$T_1^*$}
 \psfrag{t2s}{$T_2^*$}
 \psfrag{tp}{$T$}
 \psfrag{u1}{$u_1$}
 \psfrag{u2}{$u_2$}
 \psfrag{r1}{$r_1$}
 \psfrag{r2}{$r_2$}
 \psfrag{plane1}{$\plane(T_1^*)$}
 \psfrag{plane2}{$\plane(T_2^*)$}
 \psfrag{root1}{$\troot(T,(r_1,u_1))$}
 \psfrag{root2}{$\troot(T,(r_2,u_2))$}
 \includegraphics{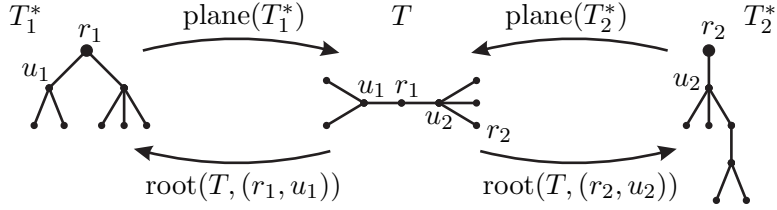}
}
\caption{Two ordered rooted trees $T_1^*,T_2^*\in\cT_7^*$ (left and right), a plane tree $T\in\cT_7$ (middle), and the transformations between them.}
\label{fig:p-trees}
\end{figure}

\subsection{Structure of the 2-factors \texorpdfstring{$\cC_{2n+1}^0$}{C2n+10} and \texorpdfstring{$\cC_{2n+1}^1$}{C2n+11}}

In this section we characterize the cycle structure of the 2-factors $\cC_{2n+1}^0$ and $\cC_{2n+1}^1$ defined in \eqref{eq:2-factor-0} and \eqref{eq:2-factor-1} using the sets of trees introduced in the previous section. The main results of this section for later use are Lemma~\ref{lemma:2f-C0} and \ref{lemma:2f-C1} below.
It turns out that the cycles of the 2-factors $\cC_{2n+1}^0$ and $\cC_{2n+1}^1$ can be characterized by considering the cyclic sequences of first vertices from the paths $\cP_{2n}(n,n+1)$ that are encountered along each cycle. In particular, it will be enlightening to consider the corresponding sequences of ordered rooted trees obtained via the bijection $\psi$.

\textit{The cyclic sequences $F(C)$ and $\varphi(F(C))$.}
Consider the set of paths $\cP_{2n}(n,n+1)$ and the 2-factor $\cC_{2n+1}$ defined in Section~\ref{sec:construction}.
For any cycle $C$ of $\cC_{2n+1}$ we let $F(C)$ denote the cyclic sequence of all vertices of the form $F(P)\circ(0)$, $P\in\cP_{2n}(n,n+1)$, when walking along this cycle (recall \eqref{eq:2-factor}). Formally, we define
\begin{equation}
\label{eq:def-FC}
  F(C):=(F(P_1),\ldots,F(P_k)) \enspace,
\end{equation}
where for all $i=1,\ldots,k$ we have $P_i\in\cP_{2n}(n,n+1)$, $F(P_i)\circ(0)\in C$ and there is a $\Phat\in\cP_{2n}(n,n+1)$ with $f_{\alpha_{2n}}(L(\Phat))=L(P_i)$ and $f_{\alpha_{2n}}(F(\Phat))=F(P_{i+1})$ (recall Lemma~\ref{lemma:FL-invariant}). All indices in this definition are treated modulo $k$, i.e., $P_{k+1}=P_1$.
In the following it will be useful to consider the corresponding cyclic sequence of lattice paths $\varphi(F(C))$. Note that by Lemma~\ref{lemma:FSL-D-isomorphic}, all lattice paths in this sequence are from the set $D_{2n}^{=0}(n)$.

The following lemma describes the relation of any two consecutive elements of the sequence $F(C)$ for any cycle $C$ of the 2-factor $\cC_{2n+1}^0$.

\begin{lemma}
\label{lemma:rotation-0}
For any $n\geq 1$ and any cycle $C$ of the 2-factor $\cC_{2n+1}^0$ defined in \eqref{eq:2-factor-0}, let $p_F,p_F'\in D_{2n}^{=0}(n)$ be any two consecutive elements of the cyclic sequence $\varphi(F(C))$ with $F(C)$ as defined in \eqref{eq:def-FC}, i.e., $\varphi(F(C))=(\ldots,p_F,p_F',\ldots)$. Then we have
\begin{subequations}
\label{eq:rotation-0}
\begin{align}
  p_F  &= (\upstep)\circ \ell(p_F) \circ (\downstep) \circ r(p_F) \enspace, \quad \label{eq:rotation-0-pF} \\
  p_F' &= \ell(p_F)\circ (\upstep)\circ r(p_F)\circ (\downstep) \enspace, \label{eq:rotation-0-pFp}
\end{align}
\end{subequations}
where $\ell(p_F)$ and $r(p_F)$ are defined in \eqref{eq:def-ell-r-F}.
\end{lemma}

Note that \eqref{eq:rotation-0-pF} is the known partition relation \eqref{eq:ell-r-F-partition}, so the actual contribution of Lemma~\ref{lemma:rotation-0} is the relation \eqref{eq:rotation-0-pFp}, describing how $p_F'$ is obtained from $p_F$.
The point is that the ordered rooted trees $\psi(p_F)$ and $\psi(p_F')$ for the lattice paths $p_F$ and $p_F'$ as in \eqref{eq:rotation-0} differ by one rotation operation, i.e., $\psi(p_F')=\trot(\psi(p_F))$ (see the bottom of Figure~\ref{fig:new-paths}).

\begin{proof}
Let $\cP_{2n}(n,n+1)$ be the set of paths defined in Section~\ref{sec:construction} for the parameter sequence $(\alpha_{2i})_{1\leq i\leq n-1}$ with $\alpha_{2i}=(0,0,\ldots,0)\seq \{0,1\}^{i-1}$ for all $i=1,\ldots,n-1$.
Fix a cycle $C$ in $\cC_{2n+1}^0$ and two consecutive elements $F(P)$ and $F(P')$ of the sequence $F(C)$, and let $\Phat\in\cP_{2n}(n,n+1)$ be such that $f_{\alpha_{2n}}(L(\Phat))=L(P)$ and $f_{\alpha_{2n}}(F(\Phat))=F(P')$.
As in the proof of Lemma~\ref{lemma:FL-matching-0}, for $\alpha_{2n}=(0,0,\ldots,0)\in\{0,1\}^{n-1}$ those relations can be simplified to show that the lattice paths $p_F:=\varphi(F(P))\in D_{2n}^{=0}(n)$ and $p_F':=\varphi(F(P'))\in D_{2n}^{=0}(n)$ satisfy the relation \eqref{eq:rotation-0-pFp} (see \eqref{eq:pF'}), proving the lemma.
\end{proof}

The following lemma characterizes the cycle structure of the 2-factor $\cC_{2n+1}^0$ in terms of ordered rooted trees and plane trees.

\begin{lemma}
\label{lemma:2f-C0}
For any $n\geq 1$ and any cycle $C$ of the 2-factor $\cC_{2n+1}^0$ defined in \eqref{eq:2-factor-0}, consider the cyclic sequence of ordered rooted trees $(T_1,\ldots,T_k):=\psi(\varphi(F(C)))$ with $F(C)$ as defined in \eqref{eq:def-FC}. Then for $i=1,\ldots,k$ we have $T_{i+1}=\trot(T_i)$, i.e., we can associate $C$ with the plane tree
\begin{equation}
\label{eq:def-TC-0}
  T^0(C):=\plane(T_1)=\cdots=\plane(T_k)\in\cT_n \enspace.
\end{equation}
Moreover, for any plane tree $T\in\cT_n$ there is a cycle $C\in\cC_{2n+1}^0$ with $T^0(C)=T$. In particular, we have $|\cC_{2n+1}^0|=|\cT_n|$.
\end{lemma}

An illustration of the rotating ordered rooted trees in the sequences $\psi(\varphi(F(C)))$, $C\in\cC_{2n+1}^0$, and the corresponding plane trees $T^0(C)$ is shown at the top of Figure~\ref{fig:rot-trees}.

\begin{proof}
The first part of the lemma is an immediate consequence of Lemma~\ref{lemma:rotation-0}: Observe that the ordered rooted trees $\psi(p_F)$ and $\psi(p_F')$ for the lattice paths $p_F$ and $p_F'$ as in \eqref{eq:rotation-0} differ by one rotation operation, i.e., $\psi(p_F')=\trot(\psi(p_F))$.
The second part of the lemma follows from Lemma~\ref{lemma:FSL-D-isomorphic} and the fact that $\psi$ is a bijection between the lattice paths $D_{2n}^{=0}(n)$ and the trees $\cT_n^*$.
\end{proof}

The following lemma is the analogue of Lemma~\ref{lemma:rotation-0} for the 2-factor $\cC_{2n+1}^1$.

\begin{lemma}
\label{lemma:rotation-1}
For any $n\geq 1$ and any cycle $C$ of the 2-factor $\cC_{2n+1}^1$ defined in \eqref{eq:2-factor-1}, let $p_F,p_F'\in D_{2n}^{=0}(n)$ be any two consecutive elements of the cyclic sequence $\varphi(F(C))$ with $F(C)$ as defined in \eqref{eq:def-FC}, i.e., $\varphi(F(C))=(\ldots,p_F,p_F',\ldots)$. Then we have
\begin{subequations}
\label{eq:rotation-1}
\begin{align}
  p_F  &= (\upstep)\circ \ell(p_F) \circ (\downstep) \circ r(p_F) \enspace, \quad \label{eq:rotation-1-pF} \\
  p_F' &= \pi_1(\ell(p_F))\circ (\upstep)\circ \pi_1(r(p_F))\circ (\downstep) \enspace, \label{eq:rotation-1-pFp}
\end{align}
\end{subequations}
where $\ell(p_F)$ and $r(p_F)$ are defined in \eqref{eq:def-ell-r-F} and $\pi_1$ in \eqref{eq:def-pi1}.
\end{lemma}

\begin{proof}
Let $\cP_{2n}(n,n+1)$ be the set of paths defined in Section~\ref{sec:construction} for the parameter sequence $(\alpha_{2i})_{1\leq i\leq n-1}$ with $\alpha_{2i}=(1,1,\ldots,1)\seq \{0,1\}^{i-1}$ for all $i=1,\ldots,n-1$.
Note that according to \eqref{eq:2-factor-1}, to construct $\cC_{2n+1}^1$ the parameter $\alpha_{2n}=(0,0,\ldots,0)\in\{0,1\}^{n-1}$ is used in the last construction step, but the paths $\cP_{2n}(n,n+1)$ do not depend on $\alpha_{2n}$, i.e., Lemma~\ref{lemma:FL-matching-1} applies.
Fix a cycle $C$ in $\cC_{2n+1}^1$ and two consecutive elements $F(P)$ and $F(P')$ of the sequence $F(C)$, and let $\Phat\in\cP_{2n}(n,n+1)$ be such that
\begin{subequations}
\label{eq:f-alpha-FLolP-FLPp}
\begin{align}
  f_{\alpha_{2n}}(L(\Phat)) &= L(P) \enspace, \\
  f_{\alpha_{2n}}(F(\Phat)) &= F(P') \enspace.
\end{align}
\end{subequations}
We define the lattice paths $p_F:=\varphi(F(P))\in D_{2n}^{=0}(n)$, $p_L:=\varphi(L(P))\in D_{2n}^-(n)$, $\phat_F:=\varphi(F(\Phat))\in D_{2n}^{=0}(n)$, $\phat_L:=\varphi(L(\Phat))\in D_{2n}^-(n)$, and $p_F':=\varphi(F(P'))\in D_{2n}^{=0}(n)$.
By the definition in \eqref{eq:f-alpha}, for $\alpha_{2n}=(0,0,\ldots,0)\in\{0,1\}^{n-1}$ we have $f_{\alpha_{2n}}=\ol{\rev}$, so the relations \eqref{eq:f-alpha-FLolP-FLPp} show that the corresponding lattice paths satisfy
\begin{align}
  \big(\ol{\rev}(r(\phat_L)),\ol{\rev}(\ell(\phat_L))\big) &= (\ell(p_L),r(p_L)) \enspace, \label{eq:ell-r-olpL-pL-flippp} \\
  \ol{\rev}(\phat_F) &= p_F' \label{eq:rev-phFp} \enspace.
\end{align}
Furthermore, by Lemma~\ref{lemma:FL-matching-1} we know that
\begin{subequations}
\begin{align}
  (\ell(p_F),r(p_F)) &= (\pi_1(\ell(p_L)),r(p_L)) \enspace, \label{eq:FL} \\
  (\ell(\phat_F),r(\phat_F)) &= (\pi_1(\ell(\phat_L)),r(\phat_L)) \enspace. \label{eq:FHh}
\end{align}
\end{subequations}
Combining these observations we obtain
\begin{equation*}
\begin{split}
  p_F' &\eqByM{\eqref{eq:rev-phFp}, \eqref{eq:ell-r-F-partition}} \ol{\rev}\big((\upstep)\circ \ell(\phat_F) \circ(\downstep)\circ r(\phat_F)\big) \\
    &\eqBy{eq:FHh} \ol{\rev}(r(\phat_L))\circ (\upstep)\circ \ol{\rev}(\pi_1(\ell(\phat_L)))\circ (\downstep) \\
    &\eqByM{\eqref{eq:ell-r-olpL-pL-flippp},\eqref{eq:FL}} \pi_1(\ell(p_F)) \circ (\upstep) \circ \pi_1(r(p_F)) \circ (\downstep) \enspace,
\end{split}
\end{equation*}
where we used $\ol{\rev}\bullet\pi_1=\pi_1\bullet \ol{\rev}$ and $\pi_1\bullet\pi_1=\id$ in the last step.
This completes the proof.
\end{proof}

\begin{figure}
\centering
\PSforPDF{
 \psfrag{2f0}{$\cC_{2n+1}^0=\{C_1,C_2,C_3\}$, $n=4$}
 \psfrag{2f1}{$\cC_{2n+1}^1=\{\Chat_1,\Chat_2,\Chat_3\}$, $n=4$}
 \psfrag{c1}{$C_1$}
 \psfrag{c2}{$C_2$}
 \psfrag{c3}{$C_3$}
 \psfrag{ch1}{$\Chat_1$}
 \psfrag{ch2}{$\Chat_2$}
 \psfrag{ch3}{$\Chat_3$} 
 \psfrag{tc1}{$T^0(C_1)$}
 \psfrag{tc2}{$T^0(C_2)$}
 \psfrag{tc3}{$T^0(C_3)$}
 \psfrag{tch1}{$T^1(\Chat_1)$}
 \psfrag{tch2}{$T^1(\Chat_2)$}
 \psfrag{tch3}{$T^1(\Chat_3)$}
 \psfrag{g0}{$g_0=\trot$}
 \psfrag{g1}{$g_1$}
 \psfrag{h}{\Large $h$}
 \psfrag{t01}{$T_1$}
 \psfrag{t02}{$T_2$}
 \psfrag{t03}{$T_3$}
 \psfrag{t04}{$T_4$}
 \psfrag{t05}{$T_5$}
 \psfrag{t06}{$T_6$}
 \psfrag{t07}{$T_7$}
 \psfrag{t08}{$T_8$}
 \psfrag{t09}{$T_9$}
 \psfrag{t10}{$T_{10}$}
 \psfrag{tel}{$T_{11}$}
 \psfrag{t12}{$T_{12}$}
 \psfrag{t13}{$T_{13}$}
 \psfrag{t14}{$T_{14}$}
 \psfrag{th01}{$\That_1$}
 \psfrag{th02}{$\That_2$}
 \psfrag{th03}{$\That_3$}
 \psfrag{th04}{$\That_4$}
 \psfrag{th05}{$\That_5$}
 \psfrag{th06}{$\That_6$}
 \psfrag{th07}{$\That_7$}
 \psfrag{th08}{$\That_8$}
 \psfrag{th09}{$\That_9$}
 \psfrag{th10}{$\That_{10}$}
 \psfrag{thel}{$\That_{11}$}
 \psfrag{th12}{$\That_{12}$}
 \psfrag{th13}{$\That_{13}$}
 \psfrag{th14}{$\That_{14}$}
 \psfrag{fp01}{\footnotesize $F(P_1)$}
 \psfrag{fp02}{\footnotesize $F(P_2)$}
 \psfrag{fp03}{\footnotesize $F(P_3)$}
 \psfrag{fp04}{\footnotesize $F(P_4)$}
 \psfrag{fp05}{\footnotesize $F(P_5)$}
 \psfrag{fp06}{\footnotesize $F(P_6)$}
 \psfrag{fp07}{\footnotesize $F(P_7)$}
 \psfrag{fp08}{\footnotesize $F(P_8)$}
 \psfrag{fp09}{\footnotesize $F(P_9)$}
 \psfrag{fp10}{\footnotesize $F(P_{10})$}
 \psfrag{fpel}{\footnotesize $F(P_{11})$}
 \psfrag{fp12}{\footnotesize $F(P_{12})$}
 \psfrag{fp13}{\footnotesize $F(P_{13})$}
 \psfrag{fp14}{\footnotesize $F(P_{14})$}
 \psfrag{fph01}{\footnotesize $F(\Phat_1)$}
 \psfrag{fph02}{\footnotesize $F(\Phat_2)$}
 \psfrag{fph03}{\footnotesize $F(\Phat_3)$}
 \psfrag{fph04}{\footnotesize $F(\Phat_4)$}
 \psfrag{fph05}{\footnotesize $F(\Phat_5)$}
 \psfrag{fph06}{\footnotesize $F(\Phat_6)$}
 \psfrag{fph07}{\footnotesize $F(\Phat_7)$}
 \psfrag{fph08}{\footnotesize $F(\Phat_8)$}
 \psfrag{fph09}{\footnotesize $F(\Phat_9)$}
 \psfrag{fph10}{\footnotesize $F(\Phat_{10})$}
 \psfrag{fphel}{\footnotesize $F(\Phat_{11})$}
 \psfrag{fph12}{\footnotesize $F(\Phat_{12})$}
 \psfrag{fph13}{\footnotesize $F(\Phat_{13})$}
 \psfrag{fph14}{\footnotesize $F(\Phat_{14})$}
 \psfrag{pu}[c][c][1][90]{$\cP_{2n}(n,n+1)$}
 \psfrag{pl}[c][c][1][90]{$f_{\alpha_{2n}}(\cP_{2n}(n,n+1))$, $\alpha_{2n}=(0,0,0)$}
 \includegraphics{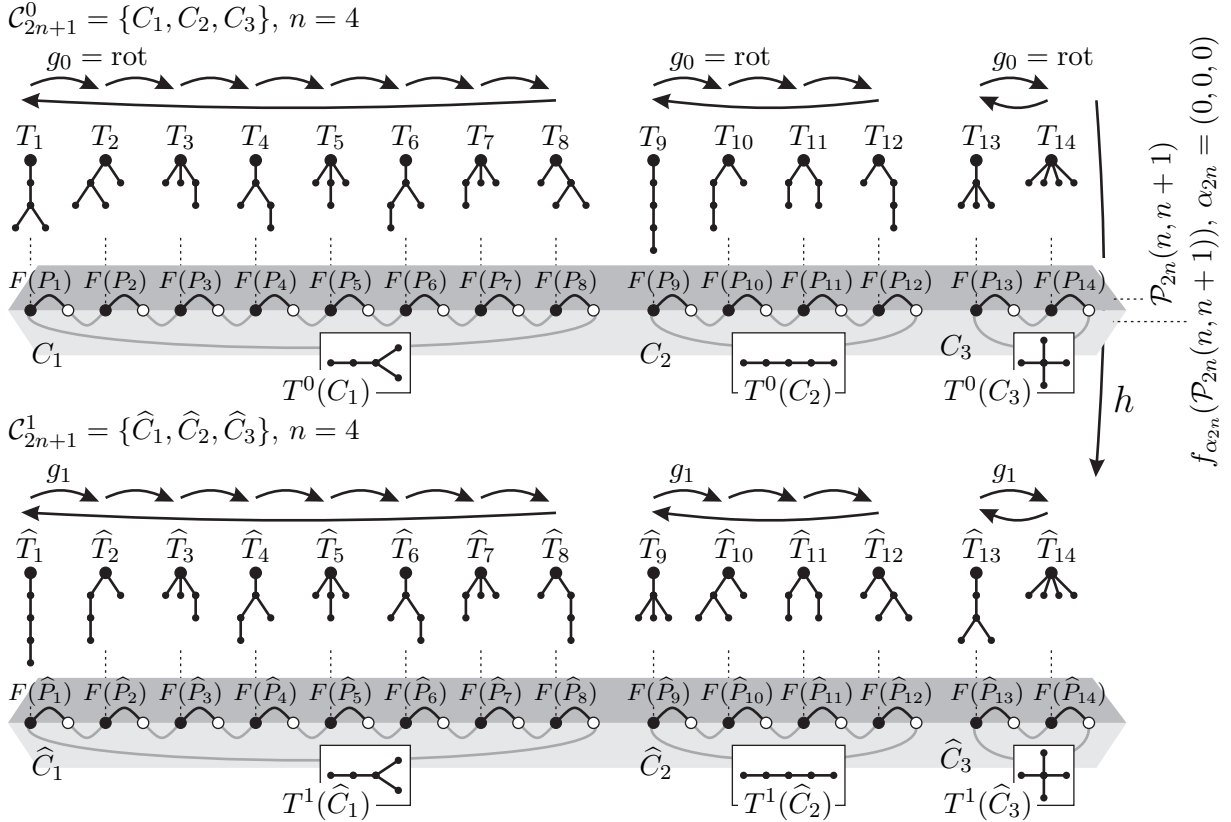}
}
\caption{Illustration of Lemmas~\ref{lemma:rotation-0}--\ref{lemma:2f-C1}. The figure shows the cyclic sequences $F(C)$, $C\in\cC_{2n+1}^0$, (top) and $F(\Chat)$, $\Chat\in\cC_{2n+1}^1$, (bottom) for $n=4$. For each element $F(P_i)$ or $F(\Phat_i)$ of the sequences $F(C)$ and $F(\Chat)$, respectively, the corresponding ordered rooted tree $T_i:=\psi(\varphi(F(P_i)))\in\cT_n^*$ and $\That_i:=\psi(\varphi(F(\Phat_i)))\in\cT_n^*$ is shown above it. The plane trees $T^0(C)\in\cT_n$, $C\in\cC_{2n+1}^0$, and $T^1(\Chat)\in\cT_n$, $\Chat\in\cC_{2n+1}^1$, as defined in \eqref{eq:def-TC-0} and \eqref{eq:def-TC-1}, respectively, are shown below each cycle.
Note how in $\cC_{2n+1}^0$ the ordered rooted trees rotate along each cycle, i.e., $g_0=\trot$.
}
\label{fig:rot-trees}
\end{figure}

While the relations \eqref{eq:rotation-0} in Lemma~\ref{lemma:rotation-0} can be interpreted straightforwardly as the rotation operation of ordered rooted trees, the relations \eqref{eq:rotation-1} in Lemma~\ref{lemma:rotation-1} are more difficult to understand: They correspond to a tree rotation plus an application of $\pi_1$ to the left and right subtree (compare \eqref{eq:rotation-0-pFp} with \eqref{eq:rotation-1-pFp}). This additional application of $\pi_1$ changes the subtrees in a way that seems rather chaotic, but that (maybe surprisingly) does not change the overall cycle structure (see the bottom of Figure~\ref{fig:rot-trees}). In the following we establish a bijection between the sequences $F(C)$, $C\in\cC_{2n+1}^0$, and the sequences $F(C)$, $C\in\cC_{2n+1}^1$. In particular, we will see that both 2-factors have the same number of cycles, namely $|\cC_{2n+1}^0|=|\cC_{2n+1}^1|=|\cT_n|$.

\textit{The mappings $g_0$, $g_1$ and $h$.}
For $n\geq 1$ and any lattice path $p\in D_{2n}^{=0}(n)$, $p=(\upstep)\circ \ell(p)\circ (\downstep)\circ r(p)$ (recall \eqref{eq:ell-r-F-partition}), we define lattice paths $g_0(p),g_1(p)\in D_{2n}^{=0}(n)$ by
\begin{align}
  g_0(p) &:= \ell(p)\circ (\upstep)\circ r(p)\circ (\downstep) \enspace, \label{eq:def-g0} \\
  g_1(p) &:= \pi_1(\ell(p))\circ (\upstep)\circ \pi_1(r(p))\circ (\downstep) \enspace, \label{eq:def-g1}
\end{align}
where $\pi_1$ is defined in \eqref{eq:def-pi1} (these definitions are motivated by \eqref{eq:rotation-0-pFp} and \eqref{eq:rotation-1-pFp}).

Furthermore, for $n\geq 0$ and any lattice path $p\in D_{2n}^{=0}(n)$ we inductively define a lattice path $h(p)\in D_{2n}^{=0}(n)$ as follows:
If $p=()$ then we define
\begin{subequations}
\label{eq:def-h}
\begin{equation}
\label{eq:def-h-indbase}
  h(p):=() \enspace.
\end{equation}
If $p$ has the form $p=(\upstep)\circ p'\circ (\downstep)$ with $p'\in D_{2n-2}^{=0}(n-1)$, then we set
\begin{equation}
\label{eq:def-h-indstep1}
  h(p):= (\upstep)\circ \pi_1(h(p'))\circ (\downstep) \enspace,
\end{equation}
where $\pi_1$ is defined in \eqref{eq:def-pi1}.
Otherwise $p$ has the form $p=p_1\circ p_2$ with $p_i\in D_{2n_i}^{=0}(n_i)$ and $n_i\geq 1$, $i\in\{1,2\}$. In this case we define
\begin{equation}
\label{eq:def-h-indstep2}
  h(p):=h(p_1)\circ h(p_2) \enspace.
\end{equation}
\end{subequations}
Note that the definition of $h(p)$ in the last case is independent of the partition of $p$ into subpaths $p_1$ and $p_2$, even if there are several different such partitions.

It is easy to check that the mappings $g_0$, $g_1$ and $h$ are bijections on $D_{2n}^{=0}(n)$.
Via the bijection $\psi$ between lattice paths and ordered rooted trees, these mappings can be extended naturally to ordered rooted trees (recall \eqref{eq:biject-map}). Figure~\ref{fig:rot-trees} shows how $g_0$, $g_1$ and $h$ operate on all 14 ordered rooted trees with 4 edges.
Note that by Lemma~\ref{lemma:rotation-0}, we have $g_0=\trot$ (so $g_0$ simply rotates a tree).

\begin{lemma}
\label{lemma:g0-via-g1}
The mappings $g_0$, $g_1$ and $h$ defined in \eqref{eq:def-g0}, \eqref{eq:def-g1} and \eqref{eq:def-h} satisfy $h\bullet g_0=g_1\bullet h$.
\end{lemma}

\begin{proof}
We show that for any $n\geq 1$ and any lattice path $p\in D_{2n}^{=0}(n)$ we have $h(g_0(p))=g_1(h(p))$.
Using the representation $p=(\upstep)\circ \ell(p)\circ(\downstep)\circ r(p)$ from \eqref{eq:ell-r-F-partition} we obtain
\begin{align*}
  h(g_0(p)) &\eqBy{eq:def-g0}  h\big(\ell(p)\circ(\upstep)\circ r(p)\circ (\downstep)\big) \\
            &\eqBy{eq:def-h-indstep2} h(\ell(p))\circ h((\upstep)\circ r(p)\circ (\downstep)) \\
            &\eqBy{eq:def-h-indstep1} h(\ell(p))\circ(\upstep)\circ \pi_1(h(r(p)))\circ (\downstep)
\end{align*}
and
\begin{align*}
  g_1(h(p)) &\eqBy{eq:def-h-indstep2} g_1\big(h((\upstep)\circ \ell(p)\circ (\downstep))\circ h(r(p))\big) \\
            &\eqBy{eq:def-h-indstep1} g_1\big((\upstep)\circ \pi_1(h(\ell(p))) \circ (\downstep) \circ h(r(p))\big) \\
            &\eqBy{eq:def-g1}  h(\ell(p))\circ (\upstep)\circ \pi_1(h(r(p)))\circ (\downstep) \enspace,
\end{align*}
where we used $\pi_1\bullet\pi_1=\id$ in the last step. This proves the claim.
\end{proof}

The next lemma shows that, intuitively speaking, the 2-factor $\cC_{2n+1}^1$ is the image of the 2-factor $\cC_{2n+1}^0$ under the mapping $h$ (see Figure~\ref{fig:rot-trees}), i.e., the mapping $h$ is the desired link between these 2-factors. This allows a precise understanding of the cycle structure of $\cC_{2n+1}^1$ via the cycle structure of $\cC_{2n+1}^0$. Recall from Lemma~\ref{lemma:2f-C0} that the cycle structure of $\cC_{2n+1}^0$ can be described by the rotation of ordered rooted trees and the corresponding plane trees.

\begin{lemma}
\label{lemma:2f01-bijection}
Let $\cC_{2n+1}^0$ and $\cC_{2n+1}^1$ be the 2-factors defined in \eqref{eq:2-factor-0} and \eqref{eq:2-factor-1}. The mapping $h$ defined in \eqref{eq:def-h} is a bijection between the cyclic sequences in the sets $\{\varphi(F(C))\mid C\in\cC_{2n+1}^0\}$ and $\{\varphi(F(C))\mid C\in\cC_{2n+1}^1\}$, where $F(C)$ is defined in \eqref{eq:def-FC}.
\end{lemma}

\begin{proof}
By Lemma~\ref{lemma:FSL-D-isomorphic}, Lemma~\ref{lemma:rotation-0} and the definition in \eqref{eq:def-g0} we have
\begin{equation}
\label{eq:FC-g0}
  \{\varphi(F(C))\mid C\in\cC_{2n+1}^0\} = \big\{(p,g_0^1(p),g_0^2(p),\ldots) \mid p\in D_{2n}^{=0}(n) \big\} \enspace.
\end{equation}
Similarly, by Lemma~\ref{lemma:FSL-D-isomorphic}, Lemma~\ref{lemma:rotation-1} and the definition in \eqref{eq:def-g1} we have
\begin{equation}
\label{eq:FC-g1}
  \{\varphi(F(C))\mid C\in\cC_{2n+1}^0\} = \big\{(p,g_1^1(p),g_1^2(p),\ldots) \mid p\in D_{2n}^{=0}(n) \big\} \enspace.
\end{equation}
By Lemma~\ref{lemma:g0-via-g1} we have $g_1=h\bullet g_0\bullet h^{-1}$, i.e., the mapping $h$ is a bijection between the cyclic sequences in the sets on the right hand side of \eqref{eq:FC-g0} and \eqref{eq:FC-g1}.
\end{proof}

Combining Lemma~\ref{lemma:2f-C0} and Lemma~\ref{lemma:2f01-bijection} immediately yields the following characterization of the cycle structure of the 2-factor $\cC_{2n+1}^1$. This lemma is the analogue of Lemma~\ref{lemma:2f-C0} for the 2-factor $\cC_{2n+1}^1$.

\begin{lemma}
\label{lemma:2f-C1}
For any $n\geq 1$ and any cycle $C$ of the 2-factor $\cC_{2n+1}^1$ defined in \eqref{eq:2-factor-1}, consider the cyclic sequence of ordered rooted trees $(T_1,\ldots,T_k):=\psi(\varphi(F(C)))$ with $F(C)$ as defined in \eqref{eq:def-FC}. Then for $i=1,\ldots,k$ we have $h^{-1}(T_{i+1})=\trot(h^{-1}(T_i))$, where $h$ is defined in \eqref{eq:def-h}. I.e., we can associate $C$ with the plane tree
\begin{equation}
\label{eq:def-TC-1}
  T^1(C):=\plane(h^{-1}(T_1))=\cdots=\plane(h^{-1}(T_k))\in\cT_n \enspace.
\end{equation}
Moreover, for any plane tree $T\in\cT_n$ there is a cycle $C\in\cC_{2n+1}^1$ with $T^1(C)=T$. In particular, we have $|\cC_{2n+1}^1|=|\cT_n|$.
\end{lemma}

An illustration of the ordered rooted trees in the sequences $\psi(\varphi(F(C)))$, $C\in\cC_{2n+1}^1$, and the corresponding plane trees $T^1(C)$ is shown at the bottom of Figure~\ref{fig:rot-trees}.

\subsection{Edge-shifting properties of the mapping \texorpdfstring{$h$}{h}}
\label{sec:edge-shift}

In this section we formulate two lemmas that describe how the mapping $h$ defined in the previous section operates on certain pairs of ordered rooted trees. Lemma~\ref{lemma:hinv-is-tau1} and \ref{lemma:hinv-is-tau2} below describe how $h$ operates on pairs of trees as shown at the top left and top right of Figure~\ref{fig:hinv-tau}, respectively (the images of these pairs of trees under $h$ are shown below). Note that each of these pairs of trees differs only in a single edge.
In Section~\ref{sec:structure-flipp} below we will see that the pairs of trees in these lemmas that arise as images under $h$ correspond to pairs of first vertices of flippable pairs in the set $\cX_{2n}^1(n,n+1)$ defined in Section~\ref{sec:flipp}. I.e., these lemmas will become crucial in proving Proposition~\ref{prop:main1} and \ref{prop:main2} when we analyze which edges are present in the graph $\cG(\cC_{2n+1}^1,\cX_{2n}^1(n,n+1))$.

For technical reasons we formulate and prove these two lemmas in terms of lattice paths, rather than ordered rooted trees (ordered rootes trees are easier to visualize, but harder to deal with formally).
All lattice paths we work with in this section are from the set $D_{2n}^{=0}(n)$.

\begin{figure}
\centering
\PSforPDF{
 \psfrag{phat}{$\psi(\phat)$}
 \psfrag{pcheck}{$\psi(\pcheck)$}
 \psfrag{hphat}{$\psi(h(\phat))$}
 \psfrag{hpcheck}{$\psi(h(\pcheck))$}
 \psfrag{k}{$k$}
 \psfrag{h}{$h$} 
 \psfrag{hq}{$\psi(h(q))$}
 \psfrag{hq0}{$\psi(h(q_0))$}
 \psfrag{qp}{$\psi(q')$}
 \psfrag{q}{$\psi(q)$}
 \psfrag{q0}{$\psi(q_0)$}
 \psfrag{q1}{$\psi(q_1)$}
 \psfrag{q2}{$\psi(q_2)$}
 \psfrag{qkp1}{$\psi(q_{k+1})$}
 \psfrag{phat1}{$\psi(\phat_1)$}
 \psfrag{phatk}{$\psi(\phat_k)$}
 \psfrag{phatkp1}{$\psi(\phat_{k+1})$}
 \psfrag{pcheck1}{$\psi(\pcheck_1)$}
 \psfrag{pcheckk}{$\psi(\pcheck_k)$}
 \psfrag{pcheckkp1}{$\psi(\pcheck_{k+1})$}
 \psfrag{vdots}{$\vdots$}
 \includegraphics{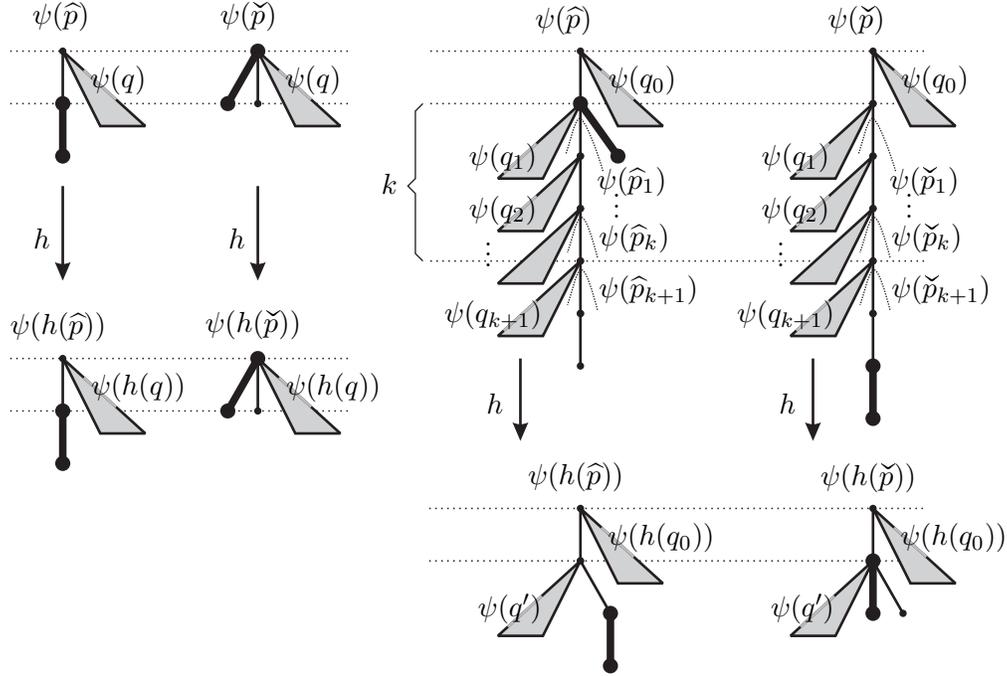}
}
\caption{Ordered rooted tree representation of the lattice paths in Lemma~\ref{lemma:hinv-is-tau1} (left) and Lemma~\ref{lemma:hinv-is-tau2} (right). The edges in which the trees differ are drawn bold, and the grey areas represent particular subtrees.}
\label{fig:hinv-tau}
\end{figure}

\begin{lemma}
\label{lemma:hinv-is-tau1}
Let $h$ be as defined in \eqref{eq:def-h}, and let $\phat,\pcheck\in D_{2n}^{=0}(n)$ be two lattice paths of the form $\phat=(\upstep,\upstep,\downstep,\downstep)\circ q$ and $\pcheck=(\upstep,\downstep,\upstep,\downstep)\circ q$ for some $q\in D_{2n-4}^{=0}(n-2)$.
Then we have
\begin{subequations}
\label{eq:hphat-hpcheck-hq}
\begin{align}
  h(\phat)   &= (\upstep,\upstep,\downstep,\downstep)\circ h(q) \enspace, \label{eq:hphat-hq} \\
  h(\pcheck) &= (\upstep,\downstep,\upstep,\downstep)\circ h(q) \enspace. \label{eq:hpcheck-hq}
\end{align}
\end{subequations}
\end{lemma}

\begin{proof}
This follows directly from the definition \eqref{eq:def-h}, using that $h(\upstep,\upstep,\downstep,\downstep)=(\upstep,\upstep,\downstep,\downstep)$ and $h(\upstep,\downstep,\upstep,\downstep)=(\upstep,\downstep,\upstep,\downstep)$.
\end{proof}

\begin{lemma}
\label{lemma:hinv-is-tau2}
Let $h$ be as defined in \eqref{eq:def-h}, and let $\phat,\pcheck\in D_{2n}^{=0}(n)$ be two lattice paths of the form
\begin{subequations}
\label{eq:phat-pcheck}
\begin{align}
  \phat &= (\upstep)\circ q_1\circ (\upstep)\circ q_2\circ (\upstep)\circ \cdots\circ q_{k+1}\circ (\upstep,\upstep,\downstep,\downstep)\circ(\downstep)^k \circ (\upstep,\downstep)\circ (\downstep)\circ q_0 \enspace, \\
  \pcheck &= (\upstep)\circ q_1\circ (\upstep)\circ q_2\circ (\upstep)\circ\cdots\circ q_{k+1}\circ (\upstep,\upstep,\upstep,\downstep,\downstep,\downstep) \;\circ\;\, (\downstep)^k \;\circ\; (\downstep)\circ q_0 \enspace,
\end{align}
\end{subequations}
where $k\geq 0$ and $q_i\in D_{2n_i}^{=0}(n_i)$ for all $i=0,\ldots,k+1$. Then we have
\begin{subequations}
\label{eq:hphat-hpcheck}
\begin{align}
  h(\phat)   &= (\upstep)\circ q'\circ (\upstep,\upstep,\downstep,\downstep)\circ (\downstep)\circ h(q_0) \enspace, \label{eq:hhphat} \\
  h(\pcheck) &= (\upstep)\circ q'\circ (\upstep,\downstep,\upstep,\downstep)\circ (\downstep)\circ h(q_0) \label{eq:hhpcheck}
\end{align}
\end{subequations}
for some $q'\in D_{2n'}^{=0}(n')$.
\end{lemma}

\begin{proof}
For the reader's convenience, the notations used in this proof are illustrated on the right hand side of Figure~\ref{fig:hinv-tau}.
We inductively define subpaths/subsequences $\phat_i$ and $\pcheck_i$, $i=k+1,k,\ldots,1$, of $\phat$ and $\pcheck$ by
\begin{subequations}
\begin{align}
  \phat_{k+1}   &:= (\upstep,\upstep,\downstep,\downstep) \enspace, \label{eq:phat-kp1} \\
  \pcheck_{k+1} &:= (\upstep,\upstep,\upstep,\downstep,\downstep,\downstep) \enspace, \label{eq:pcheckkp1} \\
  \phat_i       &:= (\upstep)\circ q_{i+1}\circ \phat_{i+1}\circ (\downstep) \enspace, \quad i=k,k-1,\ldots,1 \enspace, \label{eq:phati} \\
  \pcheck_i     &:= (\upstep)\circ q_{i+1}\circ \pcheck_{i+1}\circ (\downstep) \enspace, \quad i=k,k-1,\ldots,1 \enspace. \label{eq:pchecki}
\end{align}
\end{subequations}
Note that each $\phat_i$, $i=k+1,k,\ldots,1$, is contained in $D_{2\nhat_i}^{=0}(\nhat_i)$ for some $\nhat_i$ and that $\pcheck_i\in D_{2\nhat_i+2}^{=0}(\nhat_i+1)$.
With these definitions the relations \eqref{eq:phat-pcheck} can be written as
\begin{subequations}
\label{eq:phat-pcheck-rewr}
\begin{align}
  \phat   &= (\upstep)\circ q_1\circ \phat_1\circ (\upstep,\downstep) \circ (\downstep) \circ q_0 \enspace, \label{eq:phat-rewr} \\
  \pcheck &= (\upstep)\circ q_1\circ \pcheck_1\circ (\downstep)\circ q_0 \enspace. \label{eq:pcheck-rewr}
\end{align}
\end{subequations}
We proceed by successively computing $h(\phat)$ and $h(\pcheck)$.
Using the definition \eqref{eq:def-h}, we obtain from \eqref{eq:phat-kp1} and \eqref{eq:pcheckkp1} that
\begin{subequations}
\label{eq:hphat-ind-base}
\begin{align}
  h(\phat_{k+1})   &= (\upstep,\upstep,\downstep,\downstep) \enspace, \label{eq:hphatkp1} \\
  h(\pcheck_{k+1}) &= (\upstep,\upstep,\downstep,\upstep,\downstep,\downstep) \enspace. \label{eq:hpcheckkp1}
\end{align}
\end{subequations}
We now prove inductively that for $i=k+1,k,\ldots,1$ the lattice paths $h(\phat_i)$ and $h(\pcheck_i)$ have the form
\begin{subequations}
\label{eq:ps}
\begin{align}
  h(\phat_i)   &= (\upstep)\circ q_{i+1}'\circ (\upstep,\downstep,\downstep) \enspace, \label{eq:hphati} \\
  h(\pcheck_i) &= (\upstep)\circ q_{i+1}'\circ (\upstep,\downstep,\upstep,\downstep,\downstep) \label{eq:hpchecki}
\end{align}
\end{subequations}
for some $q_{i+1}'\in D_{2n_{i+1}'}^{=0}(n_{i+1}')$. The induction basis $i=k+1$ is given by \eqref{eq:hphat-ind-base}.
For the induction step $i+1\rightarrow i$ we compute
\begin{subequations}
\begin{align}
  h(\phat_i) &\eqByM{\eqref{eq:def-h}, \eqref{eq:phati}}
              (\upstep)\circ \pi_1\big(h(q_{i+1})\circ h(\phat_{i+1})\big)\circ (\downstep) \notag \\
             &\eqBy{eq:hphati}
              (\upstep)\circ \pi_1\big(h(q_{i+1})\circ (\upstep)\circ q_{i+2}'\circ (\upstep,\downstep,\downstep)\big)\circ (\downstep) \label{eq:hphat-ind-step}
\end{align}
and
\begin{align}
  h(\pcheck_i) &\eqByM{\eqref{eq:def-h}, \eqref{eq:pchecki}}
                (\upstep)\circ \pi_1\big(h(q_{i+1})\circ h(\pcheck_{i+1})\big)\circ (\downstep) \notag \\
               &\eqBy{eq:hpchecki}
                (\upstep)\circ \pi_1\big(h(q_{i+1})\circ (\upstep)\circ q_{i+2}'\circ (\upstep,\downstep,\upstep,\downstep,\downstep)\big)\circ (\downstep) \enspace, \label{eq:hpcheck-ind-step}
\end{align}
\end{subequations}
where we used the induction hypothesis in the last step of the two preceding equations. Using the definition of $\pi_1$ from \eqref{eq:def-pi1} we can write \eqref{eq:hphat-ind-step} and \eqref{eq:hpcheck-ind-step} as $(\upstep)\circ q_{i+1}'\circ (\upstep,\downstep,\downstep)$ and $(\upstep)\circ q_{i+1}'\circ (\upstep,\downstep,\upstep,\downstep,\downstep)$, respectively, for some $q_{i+1}'\in D_{2n_{i+1}'}^{=0}(n_{i+1}')$, proving \eqref{eq:ps}.

Combining these observations we obtain
\begin{subequations}
\begin{align}
  h(\phat)
  &\eqByM{\eqref{eq:def-h}, \eqref{eq:phat-rewr}}
  (\upstep)\circ \pi_1\big(h(q_1)\circ h(\phat_1)\circ (\upstep,\downstep)\big)\circ (\downstep)\circ h(q_0) \notag \\
  &\eqBy{eq:hphati}
  (\upstep)\circ \pi_1\big(h(q_1)\circ (\upstep)\circ q_2'\circ (\upstep,\downstep,\downstep,\upstep,\downstep)\big)\circ (\downstep)\circ h(q_0) \label{eq:hphat}
\end{align}
and
\begin{align}
  h(\pcheck)
  &\eqByM{\eqref{eq:def-h}, \eqref{eq:pcheck-rewr}}
  (\upstep)\circ \pi_1\big(h(q_1)\circ h(\pcheck_1)\big)\circ (\downstep)\circ h(q_0) \notag \\
  &\eqBy{eq:hpchecki}
  (\upstep)\circ \pi_1\big(h(q_1)\circ (\upstep)\circ q_2'\circ (\upstep,\downstep,\upstep,\downstep,\downstep)\big)\circ (\downstep)\circ h(q_0) \enspace. \label{eq:hpcheck}
\end{align}
\end{subequations}
Comparing \eqref{eq:hphat} and \eqref{eq:hpcheck} and using the definition of $\pi_1$ from \eqref{eq:def-pi1}, we conclude that $h(\phat)$ and $h(\pcheck)$ can indeed be written in the form \eqref{eq:hphat-hpcheck} for some $q'\in D_{2n'}^{=0}(n')$. This completes the proof.
\end{proof}

\section{Structure of the flippable pairs \texorpdfstring{$\cX_{2n}^0(n,n+1)$}{X2n0} and \texorpdfstring{$\cX_{2n}^1(n,n+1)$}{X2n1}}
\label{sec:structure-flipp}

This section constitutes the second building block of our analysis of the graph $\cG(\cC_{2n+1}^1,\cX_{2n}^1(n,n+1))$ that is required for proving Proposition~\ref{prop:main1} and \ref{prop:main2}. Specifically, we analyze in detail the sets of flippable pairs $\cX_{2n}^0(n,n+1)$ and $\cX_{2n}^1(n,n+1)$ defined in \eqref{eq:flipp-0} and \eqref{eq:flipp-1}, and show that each flippable pair corresponds to a pair of ordered rooted trees that differ in exactly one edge (it turns out that the resulting pairs of ordered rooted trees are the same for $\cX_{2n}^0(n,n+1)$ and $\cX_{2n}^1(n,n+1)$). These pairs of trees are shown at the bottom left and bottom right of Figure~\ref{fig:hinv-tau}. It follows that each edge of $\cG(\cC_{2n+1}^1,\cX_{2n}^1(n,n+1))$ can be interpreted as an elementary transformation between a pair of ordered rooted trees, or between the corresponding plane trees, where an elementary transformation consists of removing a leaf of the tree and attaching it to a different vertex.
This correspondence between flippable pairs from $\cX_{2n}^1(n,n+1)$ and pairs of ordered rooted trees is stated in Lemma~\ref{lemma:flipp-pairs-special} below.
As in Section~\ref{sec:edge-shift}, for technical reasons all results and proofs presented in this section are formulated in terms of lattice paths rather than ordered rooted trees (in the comments we will emphasize the connection to ordered rooted trees provided by the bijection $\psi$).

\subsection{Structure of the flippable pairs \texorpdfstring{$\cX_{2n}^0(n,n+1)$}{X2n0} and \texorpdfstring{$\cX_{2n}^1(n,n+1)$}{X2n1}}

To keep track which pairs of paths from $\cP_{2n}(k,k+1)$ are contained in a flippable pair in $\cX_{2n}(k,k+1)$ through our inductive construction, it suffices to consider the corresponding pairs of first vertices $F(\cX_{2n}(k,k+1))$ (the actual paths are already tracked via the sets $\cP_{2n}(k,k+1)$, so no information is lost by restricting our attention to only one vertex on each path in a flippable pair).
Note that by Lemma~\ref{lemma:FSL-D-isomorphic} we have $\varphi(F(\cX_{2n}(k,k+1)))\seq D_{2n}^{=0}(k)\times D_{2n}^{=0}(k)$.
The next lemma, Lemma~\ref{lemma:flipp-pairs}, provides a closed formula for the set $F(\cX_{2n}(k,k+1))$ when the all-zero or the all-one parameter sequence is used in our construction (in both cases, the resulting pairs of first vertices are the same). In particular, Lemma~\ref{lemma:flipp-pairs} describes all flippable pairs from $\cX_{2n}^1(n,n+1)$, i.e., \emph{all} edges of the graph $\cG(\cC_{2n+1}^1,\cX_{2n}^1(n,n+1))$. For proving that this graph is connected and has many spanning trees, a \emph{subset} of those edges is sufficient. This subset of flippable pairs needed for our later arguments is made explicit in Lemma~\ref{lemma:flipp-pairs-special} below, which is an immediate corollary of Lemma~\ref{lemma:flipp-pairs}.

To state Lemma~\ref{lemma:flipp-pairs} we need to introduce several definitions. These definitions are illustrated in Figure~\ref{fig:add12}.

\textit{The mappings $\lr()$, $\add_1()$, $\add_2()$.}
For any lattice path $p\in D_{2n}^{=0}(n)$ that starts at $(x,y)$ and ends at $(x+2n,y)$, we recursively define abscissas  $\lr(p),\rl(p)\in \{x,x+1,\ldots,x+2n\}$ as follows: If $n=0$, i.e., $p=()$, then we define $\lr(p)=\rl(p):=x$. If $n\geq 1$, then we consider the partition $p=p_1\circ\cdots\circ p_k$ of $p$ into the maximum number of lattice paths $p_i\in D_{2n_i}^{=0}(n_i)$ with $n_i\geq 1$, $i=1,\ldots,k$, and define $\lr(p):=\rl(p[x+1,x+2n_1-1])$ and $\rl(p):=\lr(p[x+2n-2n_k+1,x+2n-1])$. In the recursion step $\lr(p)$ descends into the subpath $p_1$ and $\rl(p)$ descends into the subpath $p_k$, but the recursion tracks the absolute coordinates of these subpaths.
We extend the mapping $\lr()$ to lattice paths $p\in D_{2n}^{=0}(k)$, $k\geq n$, as follows: Let $x$ denote the smallest strictly positive abscissa where $p$ touches the $y$-axis, then we define $\lr(p):=\lr(p[0,x])$.

Via the bijection $\psi$ between lattice paths and ordered rooted trees, the abscissas $\lr(p)$ and $\rl(p)$ for some $p\in D_{2n}^{=0}(n)$ can be interpreted as two particular leaves $\lr(T)$ and $\rl(T)$ of the tree $T=\psi(p)\in\cT_n^*$ (each lattice point on $p$ corresponds to a vertex of $\psi(p)$). The leaves $\lr(T)$ and $\rl(T)$ can be found by traversing $T$ starting at the root, and by descending, alternatingly in each level, towards the leftmost or the rightmost child (see Figure~\ref{fig:add12}). For the leaf $\lr(T)$ we start with following the leftmost child of the root, and for the leaf $\rl(T)$ we start with following the rightmost child of the root.
In the following we will not need the mapping $\rl()$ anymore, but only the mapping $\lr()$ (the former was only needed to define the latter).

Note that for $n\geq 1$ and any $p\in D_{2n}^{=0}(k)$, for $x:=\lr(p)$ we have
\begin{equation}
\label{eq:pxxp1}
  (p_x,p_{x+1})=(\upstep,\downstep) \enspace.
\end{equation}
For any $n\geq 2$ and any lattice path $p\in D_{2n-2}^{=0}(k)$, $k\geq n-1$, we define lattice paths $\add_1(p),\add_2(p)\in D_{2n}^{=0}(k+1)$ as follows: We set $x:=\lr(p)$ and define
\begin{subequations}
\label{eq:def-add12}
\begin{align}
  \add_1(p) &:= (p_1,...\,,p_x,\upstep,\downstep,p_{x+1},...\,,p_{2n-2})
            \eqBy{eq:pxxp1} (p_1,...\,,p_{x-1},\upstep,\upstep,\downstep,\downstep,p_{x+2},...\,,p_{2n-2}) \enspace, \\
  \add_2(p) &:= (p_1,...\,,\upstep,\downstep,p_x,p_{x+1},...\,,p_{2n-2})
            \eqBy{eq:pxxp1} (p_1,...\,,p_{x-1},\upstep,\downstep,\upstep,\downstep,p_{x+2},...\,,p_{2n-2}) \enspace.
\end{align}
\end{subequations}

In terms of ordered rooted trees, $\psi(\add_1(p))$ is obtained from the tree $\psi(p)$ by attaching an additional edge to the leaf $\lr(\psi(p))$. Similarly, $\psi(\add_2(p))$ is obtained from the tree $\psi(p)$ by attaching an additional edge to the parent $v$ of $\lr(\psi(p))$ to the left of the edge from $v$ to $\lr(\psi(p))$ (see Figure~\ref{fig:add12}).

For any lattice path $p\in D_{2n}^{=0}(k)$ of the form $p=p'\circ q$ with $p'\in D_{2n'}^{=0}(n')$, $n'\geq 1$, we clearly have $\lr(p)=\lr(p')$ and therefore 
\begin{equation}
\label{eq:add12-pq}
  \add_i(p)=\add_i(p')\circ q \enspace,\qquad i\in\{1,2\} \enspace.
\end{equation}
Furthermore, for any lattice path $p\in D_{2n}^{=0}(n)$ of the form $p=(\upstep)\circ p'\circ (\downstep)$ with $p'\in D_{2n-2}^{=0}(n-1)$ we have $\lr(p)=\rl(p[1,2n-1])=\rl(p')=2n-\lr(\ol{\rev}(p'))$ and therefore
\begin{equation}
\label{eq:add12-p-rev}
  \add_i(p)=(\upstep)\circ \ol{\rev}(\add_i(\ol{\rev}(p')))\circ(\downstep) \enspace,\qquad i\in\{1,2\} \enspace.
\end{equation}

\begin{figure}
\centering
\PSforPDF{
 \psfrag{p}{\footnotesize $p\in D_{18}^{=0}(9)$}
 \psfrag{t}{\footnotesize $T=\psi(p)\in\cT_9^*$}
 \psfrag{add1}{\footnotesize $\add_1(p)$}
 \psfrag{add2}{\footnotesize $\add_2(p)$}
 \psfrag{add1t}{\footnotesize $\psi(\add_1(p))\in\cT_{10}^*$}
 \psfrag{add2t}{\footnotesize $\psi(\add_2(p))\in\cT_{10}^*$}
 \psfrag{lrt}{\footnotesize $\lr(T)$}
 \psfrag{z}{\footnotesize 0}
 \psfrag{eig}{\footnotesize 18}
 \psfrag{lr0}{\footnotesize $\lr(p[0,18])$}
 \psfrag{lr1}{\footnotesize $\rl(p[1,11])$}
 \psfrag{lr2}{\footnotesize $\lr(p[4,10])$}
 \psfrag{lr3}{\footnotesize $\rl(p[5,9])$}
 \psfrag{lr4}{\footnotesize $\lr(p[8,8])=8$}
 \includegraphics{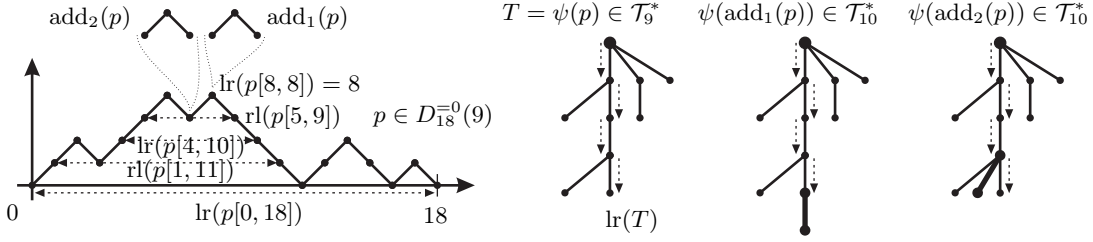}
}
\caption{Illustration of the definitions of the mappings $\lr()$, $\rl()$, $\add_1()$ and $\add_2()$. The edges added to the tree $T=\psi(p)$ by $\add_1()$ and $\add_2()$ are drawn bold.}
\label{fig:add12}
\end{figure}

\begin{lemma}
\label{lemma:flipp-pairs}
Let $n\geq 2$ and let $\cX_{2n}^0(k,k+1)$ and $\cX_{2n}^1(k,k+1)$ be the sets of flippable pairs defined in \eqref{eq:flipp-0} and \eqref{eq:flipp-1}, respectively.
For all $k=n,n+1,\ldots,2n-1$ we have
\begin{equation}
\label{eq:flipp-pairs-D}
  \varphi(F(\cX_{2n}^0(k,k+1)))=\{(\add_1(p),\add_2(p))\mid p\in D_{2n-2}^{=0}(k-1)\} \enspace,
\end{equation}
where $\add_1$ and $\add_2$ are defined in \eqref{eq:def-add12}.
Moreover, we have $F(\cX_{2n}^0(k,k+1))=F(\cX_{2n}^1(k,k+1))$.
\end{lemma}

Informally speaking, Lemma~\ref{lemma:flipp-pairs} states that all flippable pairs from $\cX_{2n}^0(k,k+1)$ or $\cX_{2n}^1(k,k+1)$ --- more precisely, the lattice paths corresponding to the first vertices of these paths --- are obtained by considering all lattice paths $p\in D_{2n-2}^{=0}(k-1)$ and by extending each of them by two steps, specifically, by applying $\add_1()$ and $\add_2()$ to $p$.

\begin{proof}
For the proof we analyze the inductive construction of flippable pairs described in Section~\ref{sec:construction-flipp}. We first consider the parameter sequence $\alpha_{2i}=(0,0,\ldots,0)\in\{0,1\}^{i-1}$ for all $i=1,\ldots,n-1$, i.e., the sets of flippable pairs $\cX_{2n}^0(k,k+1)$, $k=n,n+1,\ldots,2n-1$.

We first consider a fixed path $P^+\in\cP_{2n+2}^*$ as in the definition of the set $\cX_{2n+2}^*$ in \eqref{eq:new-flipp-paths}. By this definition, there is path $P\in\cP_{2n}(n,n+1)$ and a path $\Phat$ in one of the pairs from $\cX_{2n}^0(n,n+1)$ with
\begin{equation}
\label{eq:fLPhat-LP}
  f_{\alpha_{2n}}(L(\Phat))=L(P)
\end{equation}
such that $P^+$ contains all except the first edge of $P\circ(0,1)$ and all edges from $f_{\alpha_{2n}}(\Phat)\circ(1,1)$.
We define the lattice paths $p_F:=\varphi(F(P))\in D_{2n}^{=0}(n)$, $p_L:=\varphi(L(P))\in D_{2n}^-(n)$, $\phat_F:=\varphi(F(\Phat))\in D_{2n}^{=0}(n)$, $\phat_L:=\varphi(L(\Phat))\in D_{2n}^-(n)$, and $p_F^+:=\varphi(F(P^+))\in D_{2n+2}^{=0}(n+1)$.
By the definition in \eqref{eq:f-alpha}, for $\alpha_{2n}=(0,0,\ldots,0)\in\{0,1\}^{n-1}$ we have $f_{\alpha_{2n}}=\ol{\rev}$, so the relation \eqref{eq:fLPhat-LP} shows that the corresponding lattice paths satisfy
\begin{equation}
\label{eq:phatL-pL}
  \big(\ol{\rev}(r(\phat_L)),\ol{\rev}(\ell(\phat_L))\big) = (\ell(p_L),r(p_L)) \enspace.
\end{equation}
Furthermore, by Lemma~\ref{lemma:FL-matching-0} we know that
\begin{subequations}
\label{eq:pS-pL-phatF-phatL}
\begin{align}
  (\ell(p_S),r(p_S)) &= (\ell(p_L),r(p_L)) \enspace, \label{eq:pS-pL} \\
  (\ell(\phat_F),r(\phat_F)) &= (\ell(\phat_L),r(\phat_L)) \enspace. \label{eq:phatF-phatL}
\end{align}
\end{subequations}
By the definition in \eqref{eq:new-paths} we have $F(P^+)=S(P)\circ(0,0)$, implying that
\begin{equation}
\label{eq:pF-pS}
  p_F^+ = p_S\circ(\downstep,\downstep) \enspace.
\end{equation}
Combining our previous observations we obtain
\begin{align}
  p_F^+ &\eqByM{\eqref{eq:pF-pS},\eqref{eq:ell-r-S-partition}}
        (\upstep)\circ\ell(p_S)\circ(\upstep)\circ r(p_S)\circ (\downstep,\downstep) \notag \\
        &\eqByM{\eqref{eq:pS-pL},\eqref{eq:phatL-pL},\eqref{eq:phatF-phatL}}
        (\upstep)\circ \ol{\rev}(r(\phat_F))\circ(\upstep)\circ \ol{\rev}(\ell(\phat_F))\circ (\downstep,\downstep) \notag \\
        &\eqBy{eq:ell-r-F-partition}
        (\upstep)\circ \ol{\rev}(\phat_F)\circ (\downstep) \enspace. \label{eq:tpF}
\end{align}
Rephrasing \eqref{eq:tpF} in terms of bitstrings, we obtain that for any path $\Phat$ in a flippable pair in $\cX_{2n}^0(n,n+1)$ and the corresponding path $P^+$ in a flippable pair in $\cX_{2n+2}^*$ we have
\begin{equation}
\label{eq:FtP}
  F(P^+)=(1)\circ \ol{\rev}(F(\Phat))\circ (0) \enspace.
\end{equation}

To prove \eqref{eq:flipp-pairs-D} we argue by induction over $n$.
In this argument we will use that the sets of lattice paths introduced in Section~\ref{sec:lattice-paths} satisfy
\begin{equation*}
  D_{2n}^{>0}(n+1)\circ(\downstep,\downstep)
  =(\upstep)\circ D_{2n}^{=0}(n)\circ (\downstep)
\end{equation*}
for $n\geq 1$, which allows us to rewrite \eqref{eq:D2np2-eq0-m-partition} as
\begin{equation}
\label{eq:D2np2-eq0-m-partition-rewr}
  D_{2n+2}^{=0}(n+1) = D_{2n}^{=0}(n+1) \circ (\downstep,\downstep) \cup
                       D_{2n}^{=0}(n)\circ (\upstep,\downstep) \cup
                       (\upstep)\circ D_{2n}^{=0}(n)\circ (\downstep) \enspace.
\end{equation}
The validity of \eqref{eq:flipp-pairs-D} for the induction basis $n=2$ follows directly from \eqref{eq:ind-base-flipp}, the relations $D_2^{=0}(1)=\{(\upstep,\downstep)\}$, $D_2^{=0}(2)=\emptyset$ and the fact that $\add_1(\upstep,\downstep)=(\upstep,\upstep,\downstep,\downstep)$ and $\add_2(\upstep,\downstep)=(\upstep,\downstep,\upstep,\downstep)$.

For the induction step let $n\geq 2$. We prove \eqref{eq:flipp-pairs-D} for $n+1$ assuming that this relation holds for $n$. We distinguish the cases $k=n+2,\ldots,2n+1$ and $k=n+1$.

For $k=n+2,\ldots,2n+1$ we have
\begin{equation*}
\begin{split}
  \varphi(F(\cX_{2n+2}^0(k,k+1))) &\eqBy{eq:ind-step1-flipp}
    \varphi\big(F(\cX_{2n}^0(k,k+1))\circ(0,0)\big) \cup
    \varphi\big(F(\cX_{2n}^0(k-1,k))\circ(1,0)\big) \\
    &\qquad \cup
    \varphi\big(F(\cX_{2n}^0(k-1,k))\circ(0,1)\big) \cup
    \varphi\big(F(\cX_{2n}^0(k-2,k-1))\circ(1,1)\big) \\
    &\eqBy{eq:flipp-pairs-D}
                 \{(\add_1(p),\add_2(p))\mid p\in D_{2n-2}^{=0}(k-1)\}\circ(\downstep,\downstep) \\
    &\qquad \cup \{(\add_1(p),\add_2(p))\mid p\in D_{2n-2}^{=0}(k-2)\}\circ(\upstep,\downstep) \\
    &\qquad \cup \{(\add_1(p),\add_2(p))\mid p\in D_{2n-2}^{=0}(k-2)\}\circ(\downstep,\upstep) \\
    &\qquad \cup \{(\add_1(p),\add_2(p))\mid p\in D_{2n-2}^{=0}(k-3)\}\circ(\upstep,\upstep) \\
    &\eqByM{\eqref{eq:D2np2-*-u-partition},\eqref{eq:add12-pq}}
    \{(\add_1(p),\add_2(p))\mid p\in D_{2n}^{=0}(k-1)\} \enspace,
\end{split}
\end{equation*}
where we used the induction hypothesis in the second step.

For the case $k=n+1$ we obtain
\begin{equation*}
\begin{split}
  \varphi(F(\cX_{2n+2}^0(n+1,n+2))) 
    &\eqByM{\eqref{eq:ind-step2-flipp},\eqref{eq:FtP}}
                 \varphi\big(F(\cX_{2n}^0(n+1,n+2))\circ(0,0)\big) \\
    &\qquad \cup \varphi\big(F(\cX_{2n}^0(n,n+1))\circ(1,0)\big) \\
    &\qquad \cup \varphi\big((1)\circ \ol{\rev}(F(\cX_{2n}^0(n,n+1)))\circ(0)\big) \\
    &\eqBy{eq:flipp-pairs-D}
                 \{(\add_1(p),\add_2(p))\mid p\in D_{2n-2}^{=0}(n)\}\circ(\downstep,\downstep) \\
    &\qquad \cup \{(\add_1(p),\add_2(p))\mid p\in D_{2n-2}^{=0}(n-1)\}\circ(\upstep,\downstep) \\
    &\qquad \cup (\upstep)\circ \ol{\rev}\big(\{(\add_1(p),\add_2(p))\mid p\in D_{2n-2}^{=0}(n-1)\}\big)\circ(\downstep) \\
    &\eqByM{\eqref{eq:add12-pq},\eqref{eq:add12-p-rev},\eqref{eq:D2np2-eq0-m-partition-rewr}}
    \{(\add_1(p),\add_2(p))\mid p\in D_{2n}^{=0}(n)\} \enspace,
\end{split}
\end{equation*}
where we used the induction hypothesis in the second step, and in the last step also the relation $D_{2n-2}^{=0}(n-1)=\ol{\rev}(D_{2n-2}^{=0}(n-1))$. This completes the proof of \eqref{eq:flipp-pairs-D}.

To complete the proof of the lemma it remains to show that $F(\cX_{2n}^0(k,k+1))=F(\cX_{2n}^1(k,k+1))$.
To see this note that the first part of the above proof can be adapted for the parameter sequence $\alpha_{2i}=(1,1,\ldots,1)\in\{0,1\}^{i-1}$ for all $i=1,\ldots,n-1$ and the resulting set of flippable pairs $\cX_{2n}^1(k,k+1)$ simply by substituting \eqref{eq:phatL-pL} and \eqref{eq:pS-pL-phatF-phatL} by the following primed versions of these equations:
\begin{equation}
\labelp{eqp:phatL-pL}{eq:phatL-pL}
  \big(\ol{\rev}(\pi_1(r(\phat_L))),\ol{\rev}(\pi_1(\ell(\phat_L)))\big) = (\ell(p_L),r(p_L)) \enspace,
\end{equation}
\vspace{-5mm} 
\begin{subequations}
\begin{align}
  (\ell(p_S),r(p_S)) &= (\pi_1(\ell(p_L)),r(p_L)) \enspace, \labelp{eqp:pS-pL}{eq:pS-pL} \\
  (\ell(\phat_F),r(\phat_F)) &= (\pi_1(\ell(\phat_L)),r(\phat_L)) \enspace. \labelp{eqp:phatF-phatL}{eq:phatF-phatL}
\end{align}
\end{subequations}
In the derivation of \eqref{eqp:phatL-pL} we use that for $\alpha_{2n}=(1,1,\ldots,1)\in\{0,1\}^{n-1}$ we have $f_{\alpha_{2n}}=\ol{\rev}\bullet \pi_1$, and that the abscissas where the lattice paths $\phat_L$ and $\pi_1(\phat_L)$ touch the line $y=-1$ are the same. The relations \eqref{eqp:pS-pL} and \eqref{eqp:phatF-phatL} follow from an application of Lemma~\ref{lemma:FL-matching-1}.
Applying \eqref{eqp:phatL-pL}, \eqref{eqp:pS-pL} and \eqref{eqp:phatF-phatL} instead of \eqref{eq:phatL-pL}, \eqref{eq:pS-pL} and \eqref{eq:phatF-phatL} in the derivation of \eqref{eq:tpF} yields the \emph{same} resulting relation $p_F^+=(\upstep)\circ \ol{\rev}(\phat_F)\circ (\downstep)$ (here we use again that $\ol{\rev}\bullet\pi_1=\pi_1\bullet \ol{\rev}$ and $\pi_1\bullet\pi_1=\id$). From this point the proof continues as before, i.e., we have $F(\cX_{2n}^0(k,k+1))=F(\cX_{2n}^1(k,k+1))$, as desired.
\end{proof}

By Lemma~\ref{lemma:flipp-pairs}, the sets $F(\cX_{2n}^0(k,k+1))$ and $F(\cX_{2n}^1(k,k+1))$ are equal, so in the following we will only consider $F(\cX_{2n}^1(k,k+1))$ (everything we say also holds for $F(\cX_{2n}^0(k,k+1))$).

The following lemma is an immediate corollary of Lemma~\ref{lemma:flipp-pairs}.

\begin{lemma}
\label{lemma:flipp-pairs-special}
Let $n\geq 2$ and let $\cX_{2n}^1(n,n+1)$ be the set of flippable pairs defined in \eqref{eq:flipp-1}. We define sets of pairs of lattice paths $H_1,H_2\seq D_{2n}^{=0}(n)\times D_{2n}^{=0}(n)$ by
\begin{align}
  H_1 &:= \big\{\big((\upstep,\upstep,\downstep,\downstep)\circ q,
                     (\upstep,\downstep,\upstep,\downstep)\circ q\big) \mid q\in D_{2n-4}^{=0}(n-2)\big\} \enspace, \label{eq:def-H1} \\
  H_2 &:= \big\{\big((\upstep)\circ q_1\circ(\upstep,\upstep,\downstep,\downstep)\circ (\downstep)\circ q_2,
                     (\upstep)\circ q_1\circ(\upstep,\downstep,\upstep,\downstep)\circ (\downstep)\circ q_2\big) \mid {} \notag \\
              & \hspace{6cm} q_i\in D_{2n_i}^{=0}(n_i) \text{ for } i\in\{1,2\} \text{ and } n_1+n_2=n-3 \big\} \enspace. \label{eq:def-H2}
\end{align}
Then every pair $(p,p')\in H_1\cup H_2$ is contained in the set $\varphi(F(\cX_{2n}^1(n,n+1)))$.
\end{lemma}

Note that the pairs of lattice paths in $H_1$ and $H_2$ as defined above correspond exactly to the lattice paths that arise as images under the mapping $h$ in Lemma~\ref{lemma:hinv-is-tau1} and Lemma~\ref{lemma:hinv-is-tau2}, respectively (see \eqref{eq:hphat-hpcheck-hq} and \eqref{eq:hphat-hpcheck}). Recall that these pairs of lattice paths correspond to pairs of ordered rooted trees from $\cT_n^*$ shown at the bottom left and bottom right of Figure~\ref{fig:hinv-tau}. Consequently, for \emph{any} pair $(T,T')$ of ordered rooted trees from $\cT_n^*$ as shown at the bottom of Figure~\ref{fig:hinv-tau}, there is a flippable pair $(P,P')\in\cX_{2n}^1(n,n+1)$ such that $(T,T')=\psi(\varphi(F(P,P')))$.

\begin{proof}
Consider any lattice path $q\in D_{2n-4}^{=0}(n-2)$ and define $p:=(\upstep,\downstep)\circ q$. The lattice path $p$ is clearly contained in $D_{2n-2}^{=0}(n-1)$, and using the definitions in \eqref{eq:def-add12} it is readily checked that $\add_1(p)=(\upstep,\upstep,\downstep,\downstep)\circ q$ and $\add_2(p)=(\upstep,\downstep,\upstep,\downstep)\circ q$. By Lemma~\ref{lemma:flipp-pairs} we therefore have $H_1\seq \varphi(F(\cX_{2n}^1(n,n+1)))$.
To prove that $H_2\seq \varphi(F(\cX_{2n}^1(n,n+1)))$ we repeat the same calculation with the lattice path $p$ defined by $p:=(\upstep)\circ q_1\circ(\upstep,\downstep)\circ (\downstep)\circ q_2$ for any two lattice paths $q_1$ and $q_2$ with $q_i\in D_{2n_i}^{=0}(n_i)$ for $i\in\{1,2\}$ and $n_1+n_2=n-3$. This completes the proof.
\end{proof}

\section{Proofs of Proposition~\texorpdfstring{\ref{prop:main1}}{7} and \texorpdfstring{\ref{prop:main2}}{8}}
\label{sec:proofs-prop12}

In this section we complete our analysis of the graph $\cG(\cC_{2n+1}^1,\cX_{2n}^1(n,n+1))$ and present the proofs of Proposition~\ref{prop:main1} and \ref{prop:main2}. In the previous two sections we have shown that the nodes of $\cG(\cC_{2n+1}^1,\cX_{2n}^1(n,n+1))$ correspond to plane trees with $n$ edges, and that the edges of this graph correspond to elementary transformations between those trees (removing a leaf of the tree and attaching it to a different vertex).
This section is structured as follows: We start by defining another graph $\cG_n$ whose nodes are plane trees and whose edges correspond to elementary transformations between the trees, and we prove that this graph is connected and has many spanning trees (Lemma~\ref{lemma:Gn-connected} and \ref{lemma:Gn-spanning} below). The definition of the graph $\cG_n$ and its analysis are completely independent from anything mentioned before in this paper (we only need the concept of plane trees). Of course the definition is motivated by our knowledge about the 2-factor $\cC_{2n+1}^1$ and the flippable pairs $\cX_{2n}^1(n,n+1)$, so in the last part of this section we show that $\cG_n$ is a spanning subgraph of $\cG(\cC_{2n+1}^1,\cX_{2n}^1(n,n+1))$ (Lemma~\ref{lemma:gn-gCX} below). This implies that also $\cG(\cC_{2n+1}^1,\cX_{2n}^1(n,n+1))$ is connected and that it has the required number of spanning trees, and proves Proposition~\ref{prop:main1} and \ref{prop:main2}.

\subsection{The graph \texorpdfstring{$\cG_n$}{Gn}}
\label{sec:gn}

We begin by extending some of the notation about plane trees introduced in Section~\ref{sec:trees}.
These definitions are illustrated in Figure~\ref{fig:tau12} and Figure~\ref{fig:g6}.

\textit{Thin/thick leaves, clockwise/counterclockwise-next leaves.}
We call a leaf $u$ of a plane tree \emph{thin} or \emph{thick}, if the degree of the neighbor of $u$ is exactly 2 or at least 3, respectively. Clearly, for any tree with at least two edges, every leaf is either thin or thick.
Given two leaves $u$ and $v$ of a plane tree $T$, we say that $v$ is the \emph{clockwise-next} or \emph{counterclockwise-next leaf from $u$}, if all edges of $T$ not on the path $p$ from $u$ to $v$ lie to the right or left of $p$, respectively.

\textit{Tree operations $\tau_1$, $\tau_2$ and the graph $\cG_n$.}
Consider any plane tree $T$ with at least three edges and a thin leaf $u$. Let $u'$ be the neighbor of $u$ and $v$ the second neighbor of $u'$. We define $T'=\tau_1(T,u)$ as the plane tree obtained from $T$ by replacing the edge $(u,u')$ by the edge $(u,v)$, such that in $T'$ the leaf $u$ is the clockwise-next leaf from $u'$.

Consider any plane tree $T$ with a thick leaf $u$ whose clockwise-next leaf $v$ is thin. Let $u'$ be the neighbor of $u$. We define $\tau_2(T,u)$ as the plane tree obtained from $T$ by replacing the edge $(u,u')$ by the edge $(u,v)$.

The definitions of $\tau_1$ and $\tau_2$ are shown schematically in Figure~\ref{fig:tau12}.

\begin{figure}
\centering
\PSforPDF{
 \psfrag{t}{$T$}
 \psfrag{tp}{$T'$}
 \psfrag{tau1}{$\tau_1(T,u)$}
 \psfrag{tau2}{$\tau_2(T,u)$}
 \psfrag{r}{$r$}
 \psfrag{u}{$u$}
 \psfrag{up}{$u'$}
 \psfrag{v}{$v$}
 \psfrag{deg2}{$\deg=2$}
 \psfrag{deg3}{$\deg\geq 3$}
 \psfrag{cw}{$v$ is clockwise-next leaf from $u$}
 \psfrag{ccw}{$u$ is counterclockwise-next leaf from $v$}
 \includegraphics{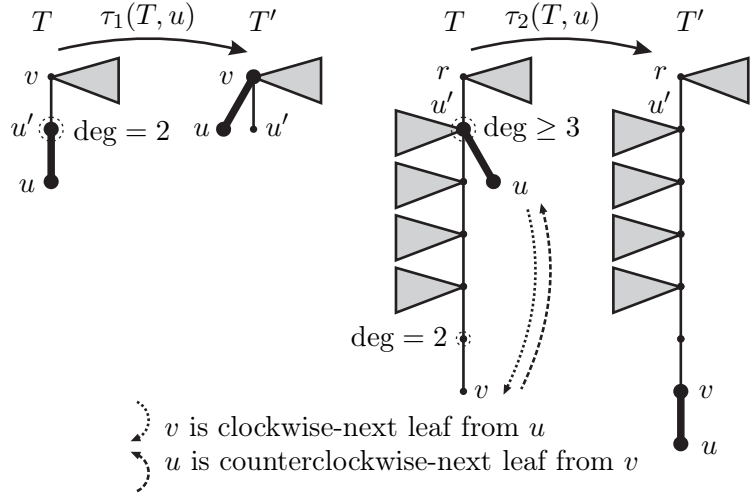}
}
\caption{Definition of the transformations $\tau_1$ (left) and $\tau_2$ (right). The edges in which the trees differ are drawn bold, and the grey areas represent arbitrary subtrees.}
\label{fig:tau12}
\end{figure}

For any $n\geq 1$ we define a directed graph $\cG_n$ whose nodes are all plane trees with $n$ edges, and whose edges capture the effects of the mappings $\tau_1$ and $\tau_2$. More formally, $\cG_n$ is a directed graph with node set $\cT_n$ and two types of edges, called $\tau_1$-edges and $\tau_2$-edges:
Any pair of trees $T,T'\in\cT_n$ is connected by a $\tau_1$-edge directed from $T$ to $T'$, if and only if $T$ has at least three edges and a thin leaf $u$ such that $T'=\tau_1(T,u)$.
Furthermore, any pair of trees $T,T'\in\cT_n$ is connected by a $\tau_2$-edge directed from $T$ to $T'$, if $T$ has a thick leaf $u$ whose clockwise-next leaf is thin such that $T'=\tau_2(T,u)$.

Figure~\ref{fig:g6} shows an example of this graph for $n=6$.

\begin{figure}
\centering
\PSforPDF{
 \psfrag{g6}{\LARGE $\cG_6$}
 \psfrag{6l}{6 leaves}
 \psfrag{5l}{5 leaves}
 \psfrag{4l}{4 leaves}
 \psfrag{3l}{3 leaves}
 \psfrag{2l}{2 leaves}
 \includegraphics{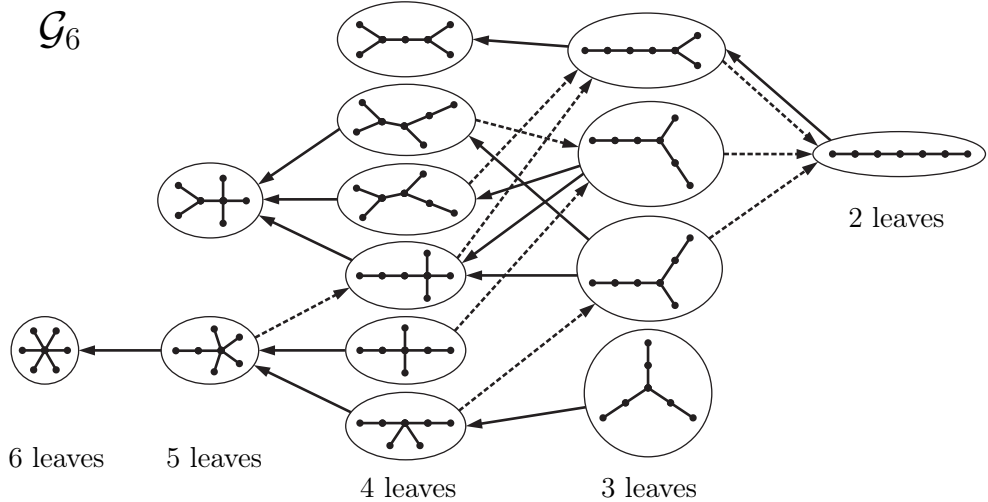}
}
\caption{The graph $\cG_6$ with nodes are arranged in layers according to the number of leaves of the corresponding plane trees. In the figure, $\tau_1$-edges are drawn as solid lines, $\tau_2$-edges as dashed lines.}
\label{fig:g6}
\end{figure}

\subsection{Properties of the graph \texorpdfstring{$\cG_n$}{Gn}}

Note that $\tau_1$ increases the number of leaves by one, and $\tau_2$ decreases the number of leaves by one (in Figure~\ref{fig:g6}, all $\tau_1$-edges go from right to left, and all $\tau_2$-edges from left to right). It follows that the subgraphs of $\cG_n$ induced either by all $\tau_1$-edges or by all $\tau_2$-edges, are acyclic. In particular, any pair of nodes of $\cG_n$ is either not connected, or connected by a single edge ($\tau_1$- or $\tau_2$-edge), or connected by a $\tau_1$-edge and a $\tau_2$-edge with opposite orientations. We conclude that $\cG_n$ does not have parallel edges with the same orientation or loops.

\begin{lemma}
\label{lemma:Gn-connected}
For any $n\geq 1$, the graph $\cG_n$ is weakly connected.
\end{lemma}

Lemma~\ref{lemma:Gn-connected} is equivalent to saying that any two plane trees $T,T'\in\cT_n$ can be transformed into each other by a sequence of applications of $\tau_1$, $\tau_2$ and the corresponding inverse mappings. In fact, for the proof we will explicitly construct such a sequence of transformations.

\begin{proof}
For the reader's convenience, the notations used in the proof are illustrated in Figure~\ref{fig:Gn-connected}.

\begin{figure}
\centering
\PSforPDF{
 \psfrag{t0}{$T_0$}
 \psfrag{t1}{$T_1$}
 \psfrag{t2}{$T_2$}
 \psfrag{t3}{$T_3$}
 \psfrag{t4}{$T_4$}
 \psfrag{t5}{$T_5$}
 \psfrag{t6}{$T_6$}
 \psfrag{t7}{$T_7$}
 \psfrag{t8}{$T_8$}
 \psfrag{v0}{$v_0$}
 \psfrag{v1}{$v_1$}
 \psfrag{v2}{$v_2$}
 \psfrag{v3}{$v_3$}
 \psfrag{v4}{$v_4$}
 \psfrag{v5}{$v_5$}
 \psfrag{v6}{$v_6$}
 \psfrag{v7}{$v_7$}
 \psfrag{v8}{$v_8$}
 \psfrag{taum1}{\small $\tau_1(T_0,v_0)$}
 \psfrag{tau0}{\small $\tau_2(T_0,u)$}
 \psfrag{tau1}{\small $\tau_1(T_1,u)$}
 \psfrag{tau2}{\small $\tau_2(T_2,u)$}
 \psfrag{tau3}{\small $\tau_1(T_3,u)$}
 \psfrag{tau4}{\small $\tau_2(T_4,u)$}
 \psfrag{tau5}{\small $\tau_2(T_5,u)$}
 \psfrag{tau6}{\small $\tau_2(T_6,u)$}
 \psfrag{tau7}{\small $\tau_2(T_7,u)$}
 \psfrag{u}{$u$}
 \psfrag{up}{$u'$}
 \psfrag{v}{$v$}
 \includegraphics{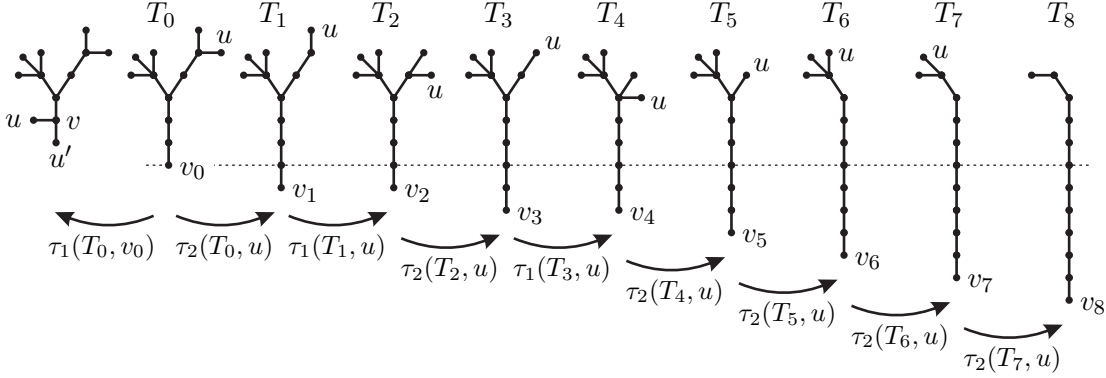}
}
\caption{Notations used in the proof of Lemma~\ref{lemma:Gn-connected}.}
\label{fig:Gn-connected}
\end{figure}

For $n=1$ and $n=2$ the claim is trivially true, as $\cG_1$ and $\cG_2$ only have a single node ($|\cT_1|=|\cT_2|=1$).
For the rest of the proof, we consider a fixed $n\geq 3$, so every tree $T\in\cT_n$ has at least three edges.
To prove the lemma, it suffices to show that any tree $T\in\cT_n$ can be transformed into $P_n\in\cT_n$, the path with $n$ edges, by a sequence of applications of $\tau_1,\tau_2$ and the corresponding inverse mappings.

To show this we fix a tree $T\in\cT_n$ that has a thin leaf.
We inductively define a sequence of trees $T_i$, $0\leq i\leq N$, and a sequence of thin leaves $v_i$ of $T_i$ as follows (see Figure~\ref{fig:Gn-connected}):
For the induction basis we define $T_0:=T$ and let $v_0$ be a thin leaf of $T$.
For the induction step $i\rightarrow i+1$, $i\geq 0$, we proceed as follows:
If $T_i=P_n$, then we are done and set $N:=i$.
Otherwise, we consider the counterclockwise-next leaf $u$ from $v_i$ in $T_i$ ($v_i$ is thin by induction).
If $u$ is thin (case~1), then we define $T_{i+1}:=\tau_1(T_i,u)$ and $v_{i+1}:=v_i$ ($v_{i+1}$ remains thin, as $v_i$ and $u$ have distance at least 3 in $T_i$).
If $u$ is thick (case~2), then we define $T_{i+1}:=\tau_2(T_i,u)$ and let $v_{i+1}:=u$ be the new leaf that has been attached to $v_i$ ($v_{i+1}$ is clearly thin). We refer to each step $i\rightarrow i+1$ as a $\tau_1$-step or $\tau_2$-step, respectively, depending on whether case~1 or case~2 applies (note that $\tau_1$- and $\tau_2$-steps correspond to $\tau_1$- and $\tau_2$-edges in $\cG_n$ on the path from $T_0$ to $T_N$).

Note that after the $k$-th $\tau_2$-step, the corresponding tree $T_{i+1}$ has a vertex $v$ with a path consisting of $k$ edges attached to it (this vertex $v$ corresponds to the original vertex $v_0$ in $T_0$).
Further observe that after a $\tau_1$-step, the counterclockwise-next leaf of $v_{i+1}$ in $T_{i+1}$ is thick, i.e., the subsequent step will be a $\tau_2$-step. In other words, a $\tau_1$-step is always followed by a $\tau_2$-step.
Combining these two observations shows that the above construction indeed yields a finite sequence of trees $T_i$, $0\leq i\leq N$, that ends with $T_N=P_n$, as desired.

It remains to consider the case that $T\in\cT_n$ has no thin leaves.
In this case, all leaves of $T$ are thick (see e.g.\ the tree on the very left of Figure~\ref{fig:Gn-connected}). It follows that $T$ has a vertex $v$ with two leaves $u$ and $u'$ as neighbors, such that $u$ is the clockwise-next leaf from $u'$ (such a vertex $v$ is given by any leaf in the tree that is obtained by removing \emph{all} leaves from $T$). Let $T'$ be the tree obtained from $T$ by replacing the edge $(u,v)$ by $(u,u')$, i.e., $T=\tau_1(T',u)$. The leaf $u$ of $T'$ is thin, and we proceed by constructing a sequence of trees leading to $P_n$ as above.
\end{proof}

The following is a considerable strengthening of Lemma~\ref{lemma:Gn-connected}.

\begin{lemma}
\label{lemma:Gn-spanning}
For any $n\geq 1$, the graph $\cG_n$ has at least $\frac{1}{4}2^{2^{\lfloor (n+1)/4\rfloor}}$ different spanning trees.
\end{lemma}

\begin{proof}
For any $n\geq 1$, we define $t_n:=\frac{1}{4}2^{2^{\lfloor (n+1)/4\rfloor}}$.
For the rest of the proof we assume that $n\geq 7$ as for $n=1,\ldots,6$, we have $t_n\leq 1$, and then the claim follows directly from Lemma~\ref{lemma:Gn-connected}.
In the following we describe a subgraph of $\cG_n$ that has at least $t_n$ different spanning trees. This proves the claim with the help of Lemma~\ref{lemma:Gn-connected}.
The construction of the subgraph of $\cG_n$ is illustrated in Figure~\ref{fig:Gn-spanning}.

\begin{figure}
\centering
\PSforPDF{
 \psfrag{tau1}{$\tau_1$}
 \psfrag{ellp}{$3+\ell$}
 \psfrag{t0}{$T_0$}
 \psfrag{t1}{$T_1$}
 \psfrag{tx}{$T(x)$}
 \psfrag{tx1}{$T_{x_1}$}
 \psfrag{tx2}{$T_{x_2}$}
 \psfrag{txk}{$T_{x_k}$}
 \psfrag{x}{\parbox{3cm}{$(x_1\ldots,x_k)$ \\ ~{} \hspace{3mm} $=(0,1,0,0)$}}
 \psfrag{txy}{$T(x,y)\in\cT_n$}
 \psfrag{u}{$u$}
 \psfrag{up}{$u'$}
 \psfrag{ttx}{$\plane(T(x))$}
 \psfrag{tty}{$\plane(T(y))$}
 \psfrag{t00}{$T(x_i,x_i)$}
 \psfrag{t10}{$T(x_{i+1},x_i)$}
 \psfrag{tr}{$T(x_{i+1},x_{i+1})$} 
 \psfrag{t01}{$T(x_i,x_{i+1})$}
 \psfrag{ts1}{$T(x_1,x_1)$}
 \psfrag{ts2}{$T(x_2,x_2)$}
 \psfrag{ts3}{$T(x_3,x_3)$}
 \psfrag{tsi}{$T(x_i,x_i)$}
 \psfrag{tsip1}{$T(x_{i+1},x_{i+1})$}
 \psfrag{ts2km1}{$T(x_{2^k-1},x_{2^k-1})$}
 \psfrag{ts2k}{$T(x_{2^k},x_{2^k})$}
 \psfrag{ggray}{\Large $\cG_n^{\gray}\seq \cG_n$}
 \psfrag{01flip}{\parbox{3cm}{bitflip $0\rightarrow 1$ from $x_i$ to $x_{i+1}$}}
 \psfrag{10flip}{\parbox{3cm}{bitflip $1\rightarrow 0$ from $x_i$ to $x_{i+1}$}}
 \psfrag{cdots}{$\cdots$}
 \includegraphics{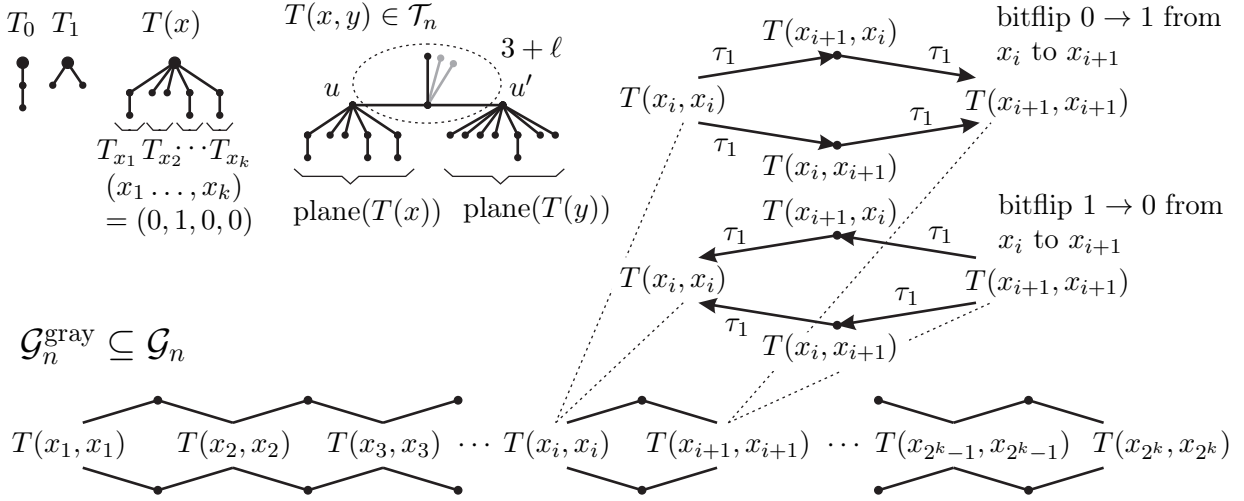}
}
\caption{Construction of trees and the subgraph of $\cG_n$ in the proof of Lemma~\ref{lemma:Gn-spanning}.}
\label{fig:Gn-spanning}
\end{figure}

Defining $k:=\lfloor (n-3)/4\rfloor$ and $\ell:=n-3-4k$, we clearly have $k\geq 1$ (recall the assumption $n\geq 7$) and $0\leq\ell<4$.
Let $T_0$ and $T_1$ be the two ordered rooted trees with two edges shown at the top left of Figure~\ref{fig:Gn-spanning} (the roots are drawn bold).
For any bitstring $x\in\{0,1\}^k$ we define an ordered rooted tree $T(x)$ with $2k$ edges that encodes the bitstring $x$ as follows: $T(x)$ is obtained by gluing together the trees $T_{x_i}$ for $i=1,\ldots,k$ at their roots from left to right (the gluing vertex becomes the root of $T(x)$). For any two bitstrings $x,y\in\{0,1\}^k$ we define the plane tree $T(x,y)$ that encodes the bitstrings $x$ and $y$ as follows: We take a star with $3+\ell$ edges and consider two of its leaves $u$ and $u'$ such that $u$ is the clockwise-next leaf from $u'$. $T(x,y)$ is obtained by gluing the trees $\plane(T(x))$ and $\plane(T(y))$ onto the leaves $u$ and $u'$ of the star, respectively, where the trees $\plane(T(x))$ and $\plane(T(y))$ are glued onto the star with the vertices that correspond to the roots of $T(x)$ and $T(y)$.
Note that $T(x,y)$ has $3+\ell+2\cdot 2k=n$ edges, so $T(x,y)\in\cT_n$.
We state the following two simple observations for further reference:

\emph{Observation 1:} For any $x,x',y,y'\in\{0,1\}^k$, if $(x,y)$ is different from $(x',y')$, then $T(x,y)$ and $T(x',y')$ are different plane trees (for this we need the star with $3+\ell$ edges in the definition of $T(x,y)$ which prevents that $T(x,y)$ has any rotational symmetries).

\emph{Observation 2:} For any $x,x'\in\{0,1\}^k$, if $x$ and $x'$ differ in exactly single bit, then for any $y\in\{0,1\}^k$, the nodes $T(x,y)$ and $T(x',y)$ are connected in $\cG_n$ (one of these trees is mapped onto the other by $\tau_1$). Similarly, if $y,y'\in\{0,1\}^k$ differ in exactly one bit, then for any $x\in\{0,1\}^k$, the nodes $T(x,y)$ and $T(x,y')$ are connected in $\cG_n$.

Let $(x_i)_{1\leq i\leq 2^k}$ be a Gray code sequence of all $2^k$ bitstrings of length $k$, i.e., for any $i=1,\ldots,2^k-1$, we have that $x_i$ and $x_{i+1}$ differ in exactly one bit.
We define
\begin{equation}
\label{eq:def-Tgray}
  \cT_n^{\gray} := \big\{T(x_i,x_i),T(x_i,x_{i+1}),T(x_{i+1},x_i) \mid i=1,\ldots,2^k-1\big\} \cup \big\{T(x_{2^k},x_{2^k})\big\} \enspace.
\end{equation}
By Observation~1, all trees on the right hand side of \eqref{eq:def-Tgray} are different, i.e., $|\cT_n^{\gray}|=3\cdot (2^k-1)+1$.
By Observation~2, for any $i=1,\ldots,2^k-1$, the nodes $T(x_i,x_i)$, $T(x_i,x_{i+1})$, $T(x_{i+1},x_i)$ and $T(x_{i+1},x_{i+1})$ form a 4-cycle $C_i$ in $\cG_n$. We let $\cG_n^{\gray}$ be the subgraph of $\cG_n$ on the node set $\cT_n^{\gray}$ with the edges of all these 4-cycles $C_i$, $i=1,\ldots,2^k-1$. Note that choosing any 3 from the 4 edges of $C_i$, independently for each $i=1,\ldots,2^k-1$, yields a different spanning tree of $\cT_n^{\gray}$. We conclude that $\cT_n^{\gray}$ has at least $4^{2^k-1}=\frac{1}{4}2^{2^{k+1}}=\frac{1}{4}2^{2^{\lfloor (n+1)/4\rfloor}}=t_n$ different spanning trees, proving the lemma.
\end{proof}

\subsection{Relation between the graphs \texorpdfstring{$\cG_n$}{Gn} and \texorpdfstring{$\cG(\cC_{2n+1}^1,\cX_{2n}^1(n,n+1))$}{G(C2n+11,X2n1)}}

We now tie together the results from all previous sections of this paper by showing that the graph $\cG_n$ is a spanning subgraph of $\cG(\cC_{2n+1}^1,\cX_{2n}^1(n,n+1))$. This implies that also $\cG(\cC_{2n+1}^1,\cX_{2n}^1(n,n+1))$ is connected and that it has the required number of spanning trees.

\begin{lemma}
\label{lemma:gn-gCX}
Let $\cC_{2n+1}^1$ be the 2-factor defined in \eqref{eq:2-factor-1}, $\cX_{2n}^1(n,n+1)$ the set of flippable pairs defined in \eqref{eq:flipp-1}, and $\cG(\cC_{2n+1}^1,\cX_{2n}^1(n,n+1))$ the graph defined in Section~\ref{sec:gCX}.
Furthermore, let $\cG_n$ be the graph defined in Section~\ref{sec:gn}.
For any $n\geq 1$, the graph $\cG_n$ is a spanning subgraph of $\cG(\cC_{2n+1}^1,\cX_{2n}^1(n,n+1))$.
\end{lemma}

\begin{proof}
By Lemma~\ref{lemma:2f-C1}, the number of nodes of $\cG(\cC_{2n+1}^1,\cX_{2n}^1(n,n+1))$ is $|\cT_n|$, so this graph has the same number of nodes as $\cG_n$. We map the nodes of $\cG(\cC_{2n+1}^1,\cX_{2n}^1(n,n+1))$ onto the nodes of $\cG_n$ as follows: Any node of $\cG(\cC_{2n+1}^1,\cX_{2n}^1(n,n+1))$ corresponds to a cycle $C$ of the 2-factor $\cC_{2n+1}^1$, and we map this node onto the plane tree $T^1(C)\in\cT_n$ (which is a node of $\cG_n$), where $T^1(C)$ is defined in \eqref{eq:def-TC-1}. It remains to show that under this bijection between the node sets of the two graphs, any directed edge of $\cG_n$ is also present in $\cG(\cC_{2n+1}^1,\cX_{2n}^1(n,n+1))$.

We fix an arbitrary $\tau_1$-edge of the graph $\cG_n$ from a plane tree $T$ to another plane tree $T'$, where $T,T'\in\cT_n$. Let $u$ be a thin leaf of $T$ such that $T'=\tau_1(T,u)$, let $u'$ be the neighbor of $u$ and $v$ the second neighbor of $u'$ in $T$ (see the left hand side of Figure~\ref{fig:tau12}). Recall that $T'$ is obtained from $T$ by replacing the edge $(u,u')$ by the edge $(u,v)$ such that in $T'$ the leaf $u$ is the clockwise-next leaf from $u'$.
We define the ordered rooted trees $\That:=\troot(T,(v,u'))$ and $\That':=\troot(T',(v,u))$ (both from $\cT_n^*$), and the corresponding lattice paths $\phat:=\psi^{-1}(\That)$ and $\phat':=\psi^{-1}(\That')$ (both from $D_{2n}^{=0}(n)$).
By Lemma~\ref{lemma:2f-C1} there are cycles $C,C'\in\cC_{2n+1}^1$ such that $T^1(C)=T$ and $T^1(C')=T'$, so by the definition in \eqref{eq:def-TC-1} we have $h(\That)\in\psi(\varphi(F(C)))$ and $h(\That')\in\psi(\varphi(F(C')))$.
We proceed to show that there is a flippable pair $(P,P')\in\cX_{2n}^1(n,n+1)$ such that $\varphi(F(P,P'))=(h(\phat),h(\phat'))$. This shows that there is a directed edge in $\cG(\cC_{2n+1}^1,\cX_{2n}^1(n,n+1))$ from $C$ to $C'$.
Indeed, $\phat$ and $\phat'$ satisfy the preconditions of Lemma~\ref{lemma:hinv-is-tau1} (compare the left hand sides of Figure~\ref{fig:hinv-tau} and Figure~\ref{fig:tau12}), i.e., their images under $h$ are of the form \eqref{eq:hphat-hq} and \eqref{eq:hpcheck-hq}, respectively.
It follows that $(h(\phat),h(\phat'))$ is contained in the set $H_1$ defined in \eqref{eq:def-H1}.
Therefore, by Lemma~\ref{lemma:flipp-pairs-special}, $(h(\phat),h(\phat'))$ is contained in $\varphi(F(\cX_{2n}^1(n,n+1)))$, i.e., there is indeed a flippable pair $(P,P')\in\cX_{2n}^1(n,n+1)$ with $\varphi(F(P,P'))=(h(\phat),h(\phat'))$.

We fix an arbitrary $\tau_2$-edge of the graph $\cG_n$ from a plane tree $T$ to another plane tree $T'$, where $T,T'\in\cT_n$. Let $u$ be a thick leaf of $T$ such that $T'=\tau_2(T,u)$, let $u'$ be the neighbor of $u$ and $v$ the clockwise-next leaf from $u$ in $T$ ($v$ is thin). Furthermore, among all neighbors of $u'$, let $r$ be the next one after $u$ in the cylic (=counterclockwise) ordering of neighbors (see the right hand side of Figure~\ref{fig:tau12}). Recall that $T'$ is obtained from $T$ by replacing the edge $(u,u')$ by the edge $(u,v)$.
We define the ordered rooted trees $\That:=\troot(T,(r,u'))$ and $\That':=\troot(T',(r,u'))$ (both from $\cT_n^*$), and the corresponding lattice paths $\phat:=\psi^{-1}(\That)$ and $\phat':=\psi^{-1}(\That')$ (both from $D_{2n}^{=0}(n)$).
By Lemma~\ref{lemma:2f-C1} there are cycles $C,C'\in\cC_{2n+1}^1$ such that $T^1(C)=T$ and $T^1(C')=T'$, so by the definition in \eqref{eq:def-TC-1} we have $h(\That)\in\psi(\varphi(F(C)))$ and $h(\That')\in\psi(\varphi(F(C')))$.
We proceed to show that there is a flippable pair $(P,P')\in\cX_{2n}^1(n,n+1)$ such that $\varphi(F(P,P'))=(h(\phat),h(\phat'))$. This shows that there is a directed edge in $\cG(\cC_{2n+1}^1,\cX_{2n}^1(n,n+1))$ from $C$ to $C'$.
Indeed, $\phat$ and $\phat'$ satisfy the preconditions of Lemma~\ref{lemma:hinv-is-tau2} (compare the right hand sides of Figure~\ref{fig:hinv-tau} and Figure~\ref{fig:tau12}), i.e., their images under $h$ are of the form \eqref{eq:hhphat} and \eqref{eq:hhpcheck}, respectively.
It follows that $(h(\phat),h(\phat'))$ is contained in the set $H_2$ defined in \eqref{eq:def-H2}.
Therefore, by Lemma~\ref{lemma:flipp-pairs-special}, $(h(\phat),h(\phat'))$ is contained in $\varphi(F(\cX_{2n}^1(n,n+1)))$, i.e., there is indeed a flippable pair $(P,P')\in\cX_{2n}^1(n,n+1)$ with $\varphi(F(P,P'))=(h(\phat),h(\phat'))$.

This completes the proof.
\end{proof}

We are now ready to prove Proposition~\ref{prop:main1} and \ref{prop:main2}.

\begin{proof}[Proof of Proposition~\ref{prop:main1}]
Combine Lemma~\ref{lemma:Gn-connected} and Lemma~\ref{lemma:gn-gCX}.
\end{proof}

\begin{proof}[Proof of Proposition~\ref{prop:main2}]
Combine Lemma~\ref{lemma:Gn-spanning} and Lemma~\ref{lemma:gn-gCX}.
\end{proof}

\section{Acknowledgements}

The author thanks Rajko Nenadov, Markus Sprecher and Franziska Weber for helpful discussions about this problem.

\bibliographystyle{alpha}
\bibliography{refs}

\end{document}